\documentclass[12pt]{scrartcl}    
\usepackage{a4} 
\usepackage{amsmath}    
\usepackage{paralist}
\usepackage{stackrel}
\usepackage{amssymb}  
\usepackage{amsfonts}   
\usepackage{mathrsfs}  
\usepackage{dsfont}
\usepackage{latexsym} 
\usepackage{xcolor}
\usepackage{bbm,exscale}
\definecolor{Myblue}{rgb}{0,0,0.6}  
\usepackage[colorlinks,citecolor=Myblue,linkcolor=Myblue,urlcolor=Myblue,pdfpagemode=None]{hyperref}
\usepackage{amsthm}
\usepackage{accents}
\usepackage[square,numbers,sort&compress]{natbib} 
\usepackage[all,cmtip]{xy}
\usepackage{ifthen} 
\usepackage{bbding}
\usepackage{stmaryrd}  
\usepackage{wasysym}
\usepackage{verbatim}
\usepackage{bbding} 
\usepackage{soul}  
\usepackage[yyyymmdd,hhmmss]{datetime}
\usepackage{booktabs}
\usepackage[inline,shortlabels]{enumitem}
\usepackage{color}
\usepackage{textcomp}
\usepackage{gensymb}
\usepackage{pdfpages}
\usepackage{tikz}
\usepackage{tikz-cd}
\usepackage{tikz-3dplot}
\usepackage{pgfplots}
\pgfplotsset{width=7cm,compat=1.8}
\usetikzlibrary{decorations.pathreplacing}
\usetikzlibrary{decorations.markings}
\usetikzlibrary{patterns}
\usepgflibrary{shapes.geometric}
\usepackage{datetime2}
\usepackage[export]{adjustbox} 
\usepackage[scr=boondoxo,scrscaled=1.00]{mathalfa}
\usepackage{subcaption}

\tikzset{
	string/.style={draw=#1, postaction={decorate}, decoration={markings,mark=at position .51 with {\arrow[color=#1]{>}}}},
	costring/.style={draw=#1, postaction={decorate}, decoration={markings,mark=at position .51 with {\arrow[draw=#1]{<}}}},
	ostring/.style={draw=#1, postaction={decorate}, decoration={markings,mark=at position .47 with {\arrow[draw=#1]{>}}}},
	ustring/.style={draw=#1, postaction={decorate}, decoration={markings,mark=at position .56 with {\arrow[draw=#1]{>}}}},
	oostring/.style={draw=#1, postaction={decorate}, decoration={markings,mark=at position .43 with {\arrow[draw=#1]{>}}}},
	uustring/.style={draw=#1, postaction={decorate}, decoration={markings,mark=at position .59 with {\arrow[draw=#1]{>}}}},
	directed/.style={string=blue!50!black}, 
	odirected/.style={ostring=blue!50!black}, 
	udirected/.style={ustring=blue!50!black}, 
	oodirected/.style={oostring=blue!50!black}, 
	uudirected/.style={uustring=blue!50!black},     
	redirected/.style={costring= blue!50!black},
	redirectedgreen/.style={costring= green!50!black},
	directedgreen/.style={string= green!50!black},
	redirectedlightgreen/.style={costring= green!65!black},
	directedlightgreen/.style={string= green!65!black},
	redirectedred/.style={costring= red!50!black},
	directedred/.style={string= red!50!black}%
}

\tikzset{-dot-/.style={decoration={
			markings,
			mark=at position 0.5 with {\fill circle (1.875pt);}},postaction={decorate}}}

\tikzset{
	Fdot/.style={circle, draw, fill, inner sep=0pt}, 
	Odot/.style={circle, draw, inner sep=0.1pt, minimum size=0.1cm}
}

\newcommand\tikzzbox[1]
{#1}

\vfuzz \hfuzz
\makeatletter
\newcommand{\raisemath}[1]{\mathpalette{\raisem@th{#1}}}
\newcommand{\raisem@th}[3]{\raisebox{#1}{$#2#3$}}
\makeatother

\newcommand{\pic}[2][1.0]{
	\begin{tikzpicture}[baseline={([yshift=-.5ex]current bounding box.center)}]
	\node at (0,0) {\includegraphics[scale=#1]{figures/#2}};
	\end{tikzpicture}
}

\def\mcA{\mathcal{A}}

\def\mcC{\mathcal{C}}
\def\mcD{\mathcal{D}}
\def\mcG{\mathcal{G}}

\def\mcL{\mathcal{L}}
\def\mcQ{\mathcal{Q}}
\def\mcR{\mathcal{R}}
\def\mcS{\mathcal{S}}
\def\mcT{\mathcal{T}}
\def\mcW{\mathcal{W}}
\def\mcX{\mathcal{X}}
\def\mcY{\mathcal{Y}}
\def\mcZ{\mathcal{Z}}
\def\mcACA{{_A \mcC_A}}

\def\opk{\Bbbk}

\def\opZ{\mathbb{Z}}
\def\opid{\mathbbm{1}}
\def\a{\alpha}

\def\g{\gamma}
\def\d{\delta}

\def\abar{\overline{\a}}

\newcommand{\taubar}[1]{\overline{\tau_{#1}}}

\def\pd{\partial}

\def\lra{\leftrightarrow}

\newcommand{\A}{\mathcal{A}}
\newcommand{\CC}{\mathcal{C}}

\newcommand{\RT}{Reshetikhin--Turaev}

\newcommand{\C}{\mathds{C}}
\newcommand{\D}{\mathds{D}}

\newcommand{\R}{\mathds{R}}
\newcommand{\Z}{\mathds{Z}}

\def\1{\ifmmode\mathrm{1\!l}\else\mbox{\(\mathrm{1\!l}\)}\fi}
\newcommand{\one}{\mathbbm{1}}
\newcommand{\be}{\begin{equation}}
  \newcommand{\ee}{\end{equation}}
\newcommand{\bes}{\begin{equation*}}
  \newcommand{\ees}{\end{equation*}}

\newcommand{\id}{\text{id}}

\newcommand{\End}{\operatorname{End}}

\newcommand{\ev}{\operatorname{ev}}

\newcommand{\tev}{\widetilde{\operatorname{ev}}}
\newcommand{\coev}{\operatorname{coev}}
\newcommand{\tcoev}{\widetilde{\operatorname{coev}}}
\newcommand{\too}{\longrightarrow}

\def\lra{\longrightarrow}

\def\lmt{\longmapsto}
\DeclareMathOperator{\tr}{tr}

\newcommand*{\longhookrightarrow}{\ensuremath{\lhook\joinrel\relbar\joinrel\rightarrow}}

\newcommand{\im}{\operatorname{im}}

\newcommand{\Borddef}{\operatorname{Bord}^{\mathrm{def}}}

\newcommand{\Borddefn}[1] {\operatorname{Bord}^{\mathrm{def}}_{#1}}
\newcommand{\Bordribn}[1] {\operatorname{Bord}^{\mathrm{rib}}_{#1}}
\newcommand{\Bordriben}[1] {\widehat{\operatorname{Bord}}{}^{\mathrm{rib}}_{#1}}
\newcommand{\Borddefen}[1] {\widehat{\operatorname{Bord}}{}^{\mathrm{def}}_{#1}}

\newcommand{\zz}{\mathcal{Z}}
\newcommand{\zrt}{\mathcal{Z}^{\textrm{RT,}\mathcal C}}
\newcommand{\zrtca}{\mathcal{Z}^{\text{RT,}\mathcal{C}_{\mathcal{A}}}}
\newcommand{\zzc}{\mathcal{Z}^{\mathcal C}}
\newcommand{\zza}{\mathcal{Z}_{\mathcal A}}
\newcommand{\zzca}{\mathcal{Z}^{\mathcal C}_{\mathcal A}}

\newcommand{\Vect}{\textrm{vect}}

\newcommand{\eps}{\varepsilon}

\newcommand{\X}{\mathcal{X}}

\newcommand\arxiv[2]      {\href{https://arXiv.org/abs/#1}{#2}}
\newcommand\doi[2]        {\href{https://dx.doi.org/#1}{#2}}

\newcommand{\eqrefO}[1]{\hyperref[eq:O#1]{\text{(O#1)}}}
\newcommand{\eqrefT}[1]{\hyperref[eq:T#1]{\text{(T#1)}}}

\allowdisplaybreaks

\deffootnote[1em]{1em}{1em}{\textsuperscript{\thefootnotemark}}

\theoremstyle{definition} 
\newtheorem{definition}{Definition}
\newtheorem{proposition}[definition]{Proposition}
\newtheorem{theorem}[definition]{Theorem}

\newtheorem{lemma}[definition]{Lemma}

\newtheorem{remark}[definition]{Remark}

\newtheorem{example}[definition]{Example}
\newtheorem{construction}[definition]{Construction}

\newtheorem{property}[definition]{Property}

\numberwithin{equation}{section}
\numberwithin{definition}{section}
\numberwithin{figure}{section}

\newcommand\void[1]{}

\begin{document}

\title{Reshetikhin–Turaev TQFTs \\ close under generalised orbifolds}

\author{%
	Nils Carqueville$^*$ \quad
	Vincentas Mulevi\v{c}ius$^\#$ \quad 
	Ingo Runkel$^\#$ \\ [0.3cm]
	Gregor Schaumann$^\vee$ \quad 
	Daniel Scherl$^\#$	
	\\[0.5cm]
	\normalsize{\texttt{\href{mailto:nils.carqueville@univie.ac.at}{nils.carqueville@univie.ac.at}}} \\  %
	\normalsize{\texttt{\href{mailto:vincentas.mulevicius@uni-hamburg.de}{vincentas.mulevicius@uni-hamburg.de}}} \\  %
	\normalsize{\texttt{\href{mailto:ingo.runkel@uni-hamburg.de}{ingo.runkel@uni-hamburg.de}}}\\[0.1cm]	\normalsize{\texttt{\href{mailto:gregor.schaumann@uni-wuerzburg.de}{gregor.schaumann@uni-wuerzburg.de}}} \\  %
	\normalsize{\texttt{\href{mailto:daniel.scherl@uni-hamburg.de}{daniel.scherl@uni-hamburg.de}}} \\[0.5cm]  %
	\hspace{-1.2cm} {\normalsize\slshape $^*$Universit\"at Wien, Fakult\"at f\"ur Physik, Wien, Austria}\\[-0.1cm]
	\hspace{-1.2cm} {\normalsize\slshape $^\#$Fachbereich Mathematik, Universit\"{a}t Hamburg, Germany}\\[-0.1cm]
	\hspace{-1.2cm} {\normalsize\slshape $^\vee$Institut f\"{u}r Mathematik, Universit\"{a}t W\"{u}rzburg, Germany}
}

\date{}
\maketitle

\begin{abstract}
We specialise the construction of orbifold graph TQFTs introduced in~\cite{CMRSS1} to Reshetikhin--Turaev defect TQFTs.
We explain that the modular fusion category $\mcC_\A$ constructed in \cite{MuleRunk} from
an orbifold datum $\A$ in a given modular fusion category $\mcC$ is a special case of the Wilson line
ribbon categories introduced as part of the general theory of orbifold
graph TQFTs. Using this, 
we prove that the Reshetikhin--Turaev TQFT obtained from $\mcC_\A$ is equivalent to the orbifold of the TQFT for $\mcC$ with respect to the orbifold datum $\A$.
\end{abstract}

\newpage

\tableofcontents

\section{Introduction and summary}

The generalised orbifold construction is a procedure to obtain a new TQFT from a given defect TQFT, that is, a TQFT defined on stratified bordisms \cite{CRS1}. The procedure is best understood as an internal state sum construction and indeed, the Turaev--Viro--Barrett--Westbury state sum TQFT 
is in this sense a generalised orbifold of the trivial TQFT \cite{CRS3}.

In \cite{CMRSS1} we extended the generalised orbifold construction to 3-dimensional graph TQFTs. 
By a graph TQFT we mean a TQFT defined on oriented bordisms with embedded ribbon graphs coloured by objects and morphisms in a given ribbon category. 
The strands and coupons of these ribbon graphs can be interpreted as Wilson lines and junction points.

The starting point for the generalised orbifold construction are 3-dimensional defect TQFTs as introduced in \cite{CMS,CRS1}.
Such TQFTs are symmetric monoidal functors whose source categories $\Borddef_3(\D)$ have stratified 3-dimensional bordisms as morphisms, 
with all strata carrying labels taken from a prescribed set of ``defect data''~$\D$. 
In particular, 1- and 2-strata are interpreted as line and surface defects, respectively. 

The generalised orbifold construction takes a defect TQFT 
\be 
\zz \colon \Borddef_3(\D)\lra\Vect
\ee 
and a particular set of defect labels $\A\subset\D$ as input to produce a graph TQFT 
\be 
\label{eq:GraphTQFTGeneral}
\zza \colon \Bordribn3(\mathcal W_\A) \lra \Vect \, . 
\ee 
As explained in \cite{CMRSS1} and reviewed in Section~\ref{subsec:RT_orbifold_graph_TQFT} below, this is done by decomposing bordisms into special stratifications that we call ``admissible skeleta'', decorating them with the chosen orbifold datum~$\A$, and then evaluating with~$\zz$ to obtain~$\zza$ as a  colimit. 
The defining conditions on~$\A$ are such that this construction is independent of the choice of skeleta. 
Embedded ribbon graphs are treated by projecting them on the 1- and 2-strata of the chosen skeleta. Compatibility constraints naturally lead to the ribbon category of Wilson lines~$\mathcal W_\A$, which is used to colour ribbon graphs of the orbifold theory.

\medskip

In this paper we will apply the construction in \cite{CMRSS1} to Reshetikhin--Turaev theories.
Given a modular fusion category~$\mcC$,
the Reshetikhin--Turaev construction \cite{turaevbook} provides a 3-dimensional graph TQFT 
\be 
\zrt \colon \Bordriben{3}(\mcC) \lra \Vect \, ,
\ee 
where the source category is given by (a central extension of) 3-dimensional bordisms with embedded $\mcC$-coloured ribbon graphs, cf.\ Section~\ref{subsec:RT_graph_TQFT} below. 
Every Reshetikhin--Turaev graph TQFT can be lifted to a defect TQFT 
\be 
\label{eq:RTDefectTQFTIntroduction}
\zzc \colon \Borddefen{3}(\D^\mcC)\lra\Vect
\ee 
with both line defects and surface defects. 
Indeed, as explained in \cite{ks1012.0911, fsv1203.4568, CRS2}, an internal 2-dimensional state sum construction identifies $\Delta$-separable symmetric Frobenius algebra in~$\mcC$ as labels for surface defects, and certain modules over such algebras as labels for line defects, see Section~\ref{subsec:RT_defect_TQFT} for more details. 
The graph TQFT $\zrt$ can be viewed as the restriction of~$\zzc$ to bordisms without surface defects. 

Our main result is to show that the general construction of orbifold graph TQFTs~\eqref{eq:GraphTQFTGeneral}, applied to the defect TQFT $\zzc$ in~\eqref{eq:RTDefectTQFTIntroduction}, produces a graph TQFT which is again of Reshetikhin--Turaev type. 
To state the precise result, we use that an orbifold datum\footnote{%
    What we call an ``orbifold datum'' here is called a ``special orbifold datum'' in \cite{CRS3,CMRSS1}. Since in this paper only special orbifold data appear, we will drop the ``special''.
    \label{fn:special}
} 
$\mcA$ for~$\zzc$ can be described as a collection of algebraic data in the underlying modular fusion category $\mcC$ (see Section~\ref{subsubsec:SpecialOrbifoldData}). 
The resulting ribbon category $\mcC_\A := \mathcal W_\A$ was first studied in \cite{MuleRunk}, where it was shown to be a modular fusion category, provided that $\A$ is simple.
A simple orbifold datum $\mcA$ for~$\zzc$ therefore yields two -- a priori distinct -- graph TQFTs:
$\zrtca$ obtained from the Reshetikhin--Turaev construction based on~$\mcC_\A$, and
\be 
\zzca \colon \Bordriben{3}(\mcC_\mcA)\lra\Vect
\ee 
obtained as the $\A$-orbifold of~$\zzc$. 
We show in Theorem~\ref{thm:tqfts_isomorphic} that they are the same:

\bigskip 

\noindent
\textbf{Theorem. }
Let $\A$ be a simple orbifold datum in a modular fusion category~$\mcC$. Then the graph TQFTs $\zrtca$ and $\zzca$ are isomorphic.
\bigskip 

The proof is inspired by a similar equivalence between the graph TQFTs of Turaev--Viro--Barrett--Westbury type obtained from a spherical fusion category~$\mcS$, and of Reshetikhin--Turaev type obtained from the Drinfeld centre $Z(\mcS)$. In fact, this result is a special case of our Theorem~\ref{thm:tqfts_isomorphic} as there is an orbifold datum~$\A^\mcS$ in the category $\mcC = \Vect$ of finite-dimensional vector spaces such that $\zz_{\A^\mcS}^{\Vect}$ is isomorphic to the Turaev--Viro--Barrett--Westbury TQFT based on~$\mcS$ \cite[Sect.\,4]{CRS3}, 
while the associated modular fusion category $\Vect_{\A^\mcS}$ is equivalent to $Z(\mcS)$ \cite[Sect.\,4.2]{MuleRunk}.
 
More specifically, we proceed analogously to \cite[Ch.\,17]{TVireBook}.
The key technical lemma (cf.\ Lemma~\ref{lem:funct_iso}) is that the following three conditions on monoidal functors $F,G$ with values in $\Vect$ 
imply a unique monoidal equivalence $F \cong G$: 
(i) $F$ is non-degenerate, 
(ii) $\dim F(X) \geqslant \dim G(X)$ for all objects~$X$ in the source category, and 
(iii) $F$ and~$G$ agree on all endomorphisms of the unit object. 
Applied to our case of $F = \zrtca$ and $G=\zzca$, we observe that
(i) all Reshetikhin--Turaev TQFTs are known to be non-degenerate \cite[Ch.\,IV]{turaevbook}, 
(ii) one can define surjective linear maps from the state spaces of $\zrtca$ to those of~$\zzca$ (Section~\ref{subsec:state_spaces}), and
(iii) the invariants associated to closed manifolds with embedded ribbon graphs agree (Section~\ref{subsec:InvariantsOf3Manifolds}). 
Apart from rather explicit computations in modular fusion categories, the main technical work in our proof is to obtain ``skeleta from surgery'' (Section~\ref{subsec:skeleta_from_surgery}), mediating between the surgery construction in Reshetikhin--Turaev theory and the decompositions needed for the orbifold construction. 

\medskip

There are a number of natural questions which arise from the above theorem. 
For example, given two modular fusion categories $\mcC$ and $\mcD$, one can ask whether
$\mathcal{Z}^{\textrm{RT,}\mathcal D}$ is isomorphic to a generalised orbifold of 
$\zrt$. It will be shown in \cite{Mu-prep} that this happens if and only if $\mcC$ and $\mcD$ are Witt equivalent \cite{DMNO}. 
In fact Witt equivalence is also the obstruction which determines if two TQFTs of Reshetikhin--Turaev type can be joined by a surface defect in the first place \cite{fsv1203.4568,KMRS}.

Another direction to continue this work is to not only consider line defects in the orbifold theory, but also surface defects, and, in fact, more generally surface defects between different orbifolds. 
This amounts to a 3-dimensional version of the orbifold completion introduced in two dimensions in \cite{CR}. 
The expected result is that Reshetikhin--Turaev TQFTs for a given Witt class (with defects of all dimensions, and with 3-strata labelled by possibly different modular fusion categories in the given Witt class) are already orbifold complete. Ideas related to orbifold completion have also been pursued in the context of  framed (as opposed to oriented) tangential structures and 
higher idempotent completions in \cite{Gaiotto:2019xmp}.

\subsubsection*{Acknowledgements} 

N.\,C.\ is supported by the DFG Heisenberg Programme. 
V.\,M.\ is partially supported by the DFG Research Training Group 1670.
I.\,R.\ is partially supported by the Cluster of Excellence EXC 2121 ``Quantum Universe'' - 390833306.

\section{Background} 
\label{sec:background}

We review some prerequisites on modular fusion categories (Section~\ref{subsubsec:ModularFusionCategories}), algebras and modules (Section~\ref{subsubsec:AlgebrasAndModules}), and orbifold constructions (Section~\ref{subsubsec:SpecialOrbifoldData}) which will be used extensively in the later sections. 
     
\subsection{Modular fusion categories}
\label{subsubsec:ModularFusionCategories}

A \textsl{modular fusion category}~$\CC$ over an algebraically closed field~$\Bbbk$ is a braided spherical fusion category over~$\Bbbk$ whose braiding is non-degenerate. 
Recall, e.\,g.\ from \cite{turaevbook, EGNO-book}, that this means in particular that~$\CC$ is a semisimple $\Bbbk$-linear monoidal category with only finitely many isomorphism classes of simple objects, and with simple tensor unit.
Moreover, invoking the spherical structure, every object $X\in\CC$ has a simultaneous left- and right- dual~$X^*$ with the (co)evaluation maps denoted by
\begin{align}
\ev_X = \pic[1.25]{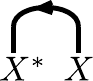}
\!
\colon 
X^* \otimes X \lra \one 
\, , \qquad 
\coev_X = \pic[1.25]{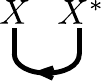}
\!
\colon \one \lra X \otimes X^* \, , 
\\ 
\tev_X = \pic[1.25]{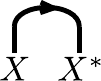}
\!\!
\colon 
X \otimes X^* \lra \one 
\, , \qquad 
\tcoev_X = \pic[1.25]{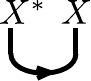}
\!
\colon \one \lra X^* \otimes X  \, , 
\end{align}
satisfying the pivotality conditions, see e.\,g.\ \cite[Lem.\,2.12]{CR}.
Our convention is to read string diagrams from bottom to top and from right to left, such that the braiding morphisms $c_{X,Y} \colon X\otimes Y \lra Y\otimes X$ as well as their inverses are depicted as follows: 
\be
c_{X,Y} = \pic[1.25]{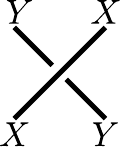}
\, , \qquad
c^{-1}_{X,Y} = \pic[1.25]{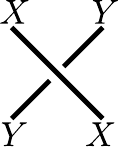}
\, . 
\ee
Non-degeneracy of the braiding means that if $c_{X,Y} \circ c_{Y,X} = \textrm{id}_{Y\otimes X}$ for all $X\in\CC$, then~$Y$ is isomorphic to $\one^{\oplus n}$ for some $n\in\opZ_{\geqslant 0}$.

\medskip

The above structure maps are subject to the standard coherence constraints (see e.\,g.\ \cite[Sect.\,2.10\,\&\,4.7\,\&\,8.1]{EGNO-book}), and to the \textsl{sphericality condition} which states that for every endomorphism $f\colon X \lra X$, the \textsl{left} and \textsl{right traces} agree: 
\begin{equation}
\textrm{tr}_\CC(f) := \pic[1.25]{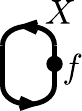} = \pic[1.25]{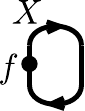}
\in \End(\one) \, .    
\end{equation}
The \textsl{quantum dimension} of an object $X\in\CC$ is 
\be 
\dim_\CC(X) := \textrm{tr}_\CC(\id_{X}) = \pic[1.25]{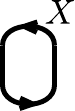} = \pic[1.25]{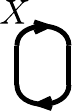}
\, , 
\ee 
and the \textsl{global dimension of}~$\CC$ is 
\be
\dim\CC := \sum_{i\in\textrm{Irr}_\CC} \dim_\CC(i)^2 \, , 
\ee 
where here and below by $\textrm{Irr}_\CC$ we denote a 
\medskip
set of representatives of isomorphism classes of simple objects in~$\CC$ such that $\one \in \textrm{Irr}_\CC$. 
For a modular fusion category $\mcC$ one has $\dim \mcC \neq 0$ (see \cite[Sect.\,II.3.2]{turaevbook}).
The construction of the TQFT in Section~\ref{sec:RT_theory} requires the choice of a square root of the global dimension,
\be 
\label{eq:DC}
\mathscr D_\CC \in \Bbbk
\quad \text{such that} \quad
(\mathscr D_\CC)^2 = \dim\CC \,.
\ee 

An important invariant 
of a modular fusion category~$\CC$ is its ``anomaly''. 
To give the definition, we recall the \textsl{twist} automorphisms 
\be
\label{eq:DC_ppmC_def}
\theta_X = \pic[1.25]{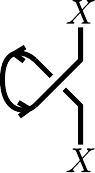} = \pic[1.25]{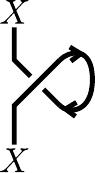}
\, , \qquad 
\theta_X^{-1} = \pic[1.25]{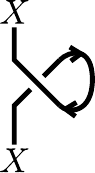} = \pic[1.25]{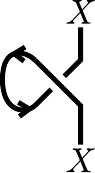}
\, , 
\ee 
and set\footnote{Note that we slightly abuse notation by identifying $\theta_i = \theta_i \cdot \id_i$ for simple objects~$i$; similarly we often view traces $\textrm{tr}_\CC(f)$ and quantum dimensions $\dim_\CC(X)$ as scalars.}
\be 
\label{eq:pC}
p_\CC^\pm := \sum_{i\in\textrm{Irr}_\CC} \theta_i^{\pm 1} \cdot \dim_\CC(i)^2 
\, .
\ee 
Then the \textsl{anomaly} of the pair~$(\CC,\mathscr D_\CC)$ 
is the scalar 
\be 
\label{eq:anomaly}
\delta_\CC 
:= 
\frac{p_\CC^+}{\mathscr D_\CC}
= 
\frac{\mathscr D_\CC}{p_\CC^-} 
\, . 
\ee 
That the two ratios are indeed equal is shown in \cite[Cor.\,3.1.11]{BakK}. For $\Bbbk=\C$, $\delta_\CC$ is a phase, which is why one considers the ratios in \eqref{eq:anomaly} rather than just $p_\CC^\pm$.

\subsection{Algebras and modules}
\label{subsubsec:AlgebrasAndModules}

A \textsl{$\Delta$-separable symmetric Frobenius algebra} is an object $A\in\CC$ with the structure of an associative unital algebra and a coassociative counital coalgebra, such that, with 
$
\!\!\!
\tikzzbox{\begin{tikzpicture}[very thick,scale=0.25,color=green!50!black, baseline=0.05cm]
	\draw[-dot-] (3,0) .. controls +(0,1) and +(0,1) .. (2,0);
	\draw (2.5,0.75) -- (2.5,1.5); 
	\fill (2,0) circle (0pt) node[left] (D) {};
	\fill (3,0) circle (0pt) node[right] (D) {};
	\fill (2.5,1.5) circle (0pt) node[right] (D) {};
	\end{tikzpicture}}
\!\!\!, 
\tikzzbox{\begin{tikzpicture}[very thick,scale=0.25,color=green!50!black, baseline=-0.3cm, rotate=180]
	\draw[-dot-] (3,0) .. controls +(0,1) and +(0,1) .. (2,0);
	\draw (2.5,0.75) -- (2.5,1.5); 
	\fill (2,0) circle (0pt) node[left] (D) {};
	\fill (3,0) circle (0pt) node[right] (D) {};
	\fill (2.5,1.5) circle (0pt) node[right] (D) {};
	\end{tikzpicture}}
\!\!
, 
\,
\tikzzbox{\begin{tikzpicture}[very thick,scale=0.25,color=green!50!black, baseline=-0.03cm]
	\draw (0,0) node[Odot] (unit) {};
	\draw (0,1) -- (unit);
	\end{tikzpicture}}
, 
\,
\tikzzbox{\begin{tikzpicture}[very thick,scale=0.25,color=green!50!black, baseline=0.05cm]
	\draw (0,1) node[Odot] (unit) {};
	\draw (0,0) -- (unit);
	\end{tikzpicture}}
$ 
denoting the (co)multiplication and (co)unit, we have 
\begin{equation}
\hspace{-0.5cm}
\begin{array}{ccc}
\pic[1.25]{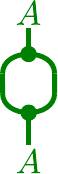} = \pic[1.25]{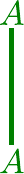}, &
\pic[1.25]{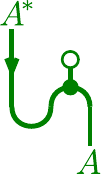} = \pic[1.25]{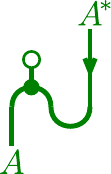}, &
\pic[1.25]{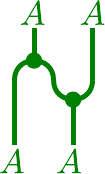} = \pic[1.25]{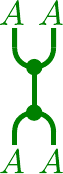} = \pic[1.25]{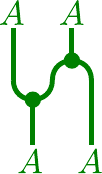}.\\
\text{(``$\Delta$-separable'')} &
\text{(``symmetric'')} &
\text{(``Frobenius'')}
\end{array}
\end{equation}

A (left) module $X$ of a Frobenius algebra $A$ is automatically a (left) comodule with the coaction
\begin{equation}
\label{eq:coaction}
\pic[1.25]{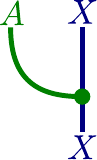} := \pic[1.25]{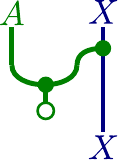} .
\end{equation}
If $A$ is a $\Delta$-separable symmetric Frobenius algebra, the tensor product over $A$ of a right module $X$ and a left module $Y$ can be realised as the image of a projector, namely
\begin{equation}
\label{eq:prod_over_A_proj}
X \otimes_A Y \cong \im P_{X\otimes Y} \, , \quad \text{where }
P_{X\otimes Y} := \pic[1.25]{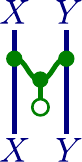} = \pic[1.25]{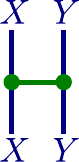} \, ,
\end{equation}
where in the last equality the expression \eqref{eq:coaction} was used.\footnote{
    Since $A$ is symmetric one does not need to distinguish between action and coaction, and hence we will not do so in our graphical notation.}
The relative tensor product $\otimes_A$ equips the category $\mcACA$ of $A$-$A$-bimodules with a monoidal structure.

\medskip

If $A$ is any algebra, then we endow $A\otimes A$ with an algebra structure with multiplication and unit 
\begin{equation}
\pic[1.25]{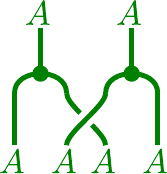}
\qquad\text{and}\qquad 
\pic[1.25]{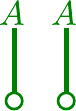}    
\end{equation} 
respectively. 
Then a right $(A\otimes A)$-module is equivalently an object~$X$ with two $A$-actions such that
\begin{equation}
\label{eq:AA_actions_commute}
\pic[1.25]{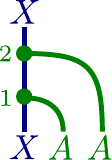} = \pic[1.25]{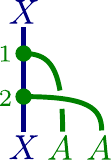} , 
\end{equation}
where we used indices $1,2$ to distinguish between the two actions.
For a left $A$-module~$Y$, we denote by $X\otimes_1 Y$ and $X\otimes_2 Y$ the relative tensor products with respect to the corresponding $A$-actions. 

\subsection{Orbifold data and associated modular fusion categories}
\label{subsubsec:SpecialOrbifoldData}

We review the notion of an \textsl{orbifold datum in a modular fusion category~$\CC$} as introduced in \cite[Sect.\,3.2]{CRS3}.\footnote{Recall from footnote~\ref{fn:special} that we always implicitly assume an orbifold datum to be special.}
It is a tuple 
\be 
\A = (A,T,\alpha,\overline{\alpha}, \psi,\phi)
\ee 
that consists of
\begin{itemize}
	\item
	a $\Delta$-separable symmetric Frobenius algebra~$A$ in~$\CC$, 
	\item 
	an $A$-$A\otimes A$-bimodule~$T$, 
	\item 
	$A$-$A\otimes A\otimes A$-bimodule maps $\alpha\colon T\otimes_2 T \lra T\otimes_1 T$ and $\overline{\alpha}\colon T\otimes_1 T \lra T\otimes_2 T$, 
	\item 
	an invertible $A$-$A$-bimodule map $\psi\colon A\lra A$, 
	\item 
	an invertible scalar $\phi \in \Bbbk^\times$ 
\end{itemize}
subject to the conditions \eqrefO{1}--\eqrefO{8} which we list in Figure~\ref{fig:SpecialOrbifoldDatum} in Appendix~\ref{app:identities}.
Here and below we use the following abbreviations for the endomorphisms obtained from $\psi$ via different $A$-actions on $T$ and on an arbitrary $A$-$A$-bimodule~$X$:
\begin{align}
\nonumber
&\pic[1.25]{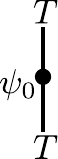} := \pic[1.25]{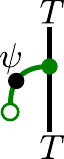} \, , \qquad
 \pic[1.25]{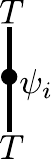} := \pic[1.25]{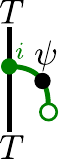} \, , \qquad
i \in \{1,2\} \, ,
\\ 
\label{eq:psi_omega_notation}
&\pic[1.25]{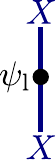} := \pic[1.25]{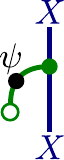} \, , \qquad
 \pic[1.25]{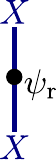} := \pic[1.25]{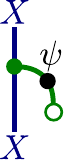} \, ,  \qquad
 \pic[1.25]{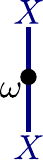} := \pic[1.25]{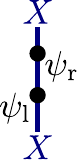} \, .
\end{align}
We also abbreviate $\omega := \psi_{\textrm{r}}\circ\psi_{\textrm{l}}$ and write an index next to the morphisms $\psi_{\textrm{l}}$, $\psi_{\textrm{r}}$, $\omega$ 
to distinguish between different bimodules if needed.
Note that if $X=A$, then $\omega = \psi^2$.

In keeping with the notation $\otimes_1$, $\otimes_2$ in \eqref{eq:AA_actions_commute}, given a right $A$-module $M$ it is convenient to write $M \otimes_0 T$ for $M \otimes_A T$. 

\medskip

\begin{definition}
Given an orbifold datum $\mcA$ in $\mcC$, as in \cite[Sect.\,3]{MuleRunk} we define the category $\mcC_\mcA$ to have
\begin{itemize}
\item
objects: tuples $\mcX = (X,\tau_1,\tau_2,\taubar{1},\taubar{2})$, where
\begin{itemize}
\item $X$ is an $A$-$A$-bimodule,
\item
$\tau_i\colon X \otimes_0 T \lra T \otimes_i X$,
$\taubar{i}\colon T \otimes_i X\lra X \otimes_0 T$,
$i \in \{1,2\}$, are $A$-$A \otimes A \otimes A$-bimodule morphisms (to be referred to as \textsl{crossings}),
\end{itemize}
which satisfy the identities \eqrefT{1}--\eqrefT{7} in Figure \ref{fig:CA_identities}; 
\item
morphisms: $f\colon \mcX \lra \mcY$ is an $A$-$A$-bimodule morphism such that
\begin{equation}
\label{eq:CA_morphism_cond}
\tau_i^X \circ (f \otimes_0 \id_T) = (\id_T \otimes_i f) \circ \tau_i^Y, \quad
i \in \{1,2\} \, .
\end{equation}
\end{itemize}
\end{definition}

Let us state some properties of the category $\mcC_\mcA$.
We refer to~\cite{MuleRunk} for more details and the proofs.
\begin{enumerate}[wide, labelwidth=!, labelindent=0pt, label=\arabic*.]
\item
$\mcC_\mcA$ is a multifusion category with the tensor product
\begin{equation}
\mcX \otimes \mcY = (X\otimes_A Y,\tau_1^{X,Y},\tau_2^{X,Y},\taubar{1}^{X,Y},\taubar{2}^{X,Y})    
\end{equation}
where the $T$-crossings are
\begin{equation}
\tau_i^{X,Y} = \pic[1.25]{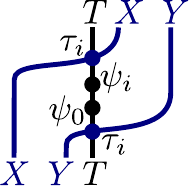} \quad , \qquad
\taubar{i}^{X,Y} = \pic[1.25]{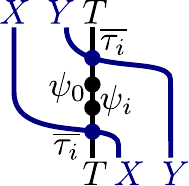} \, , \quad i \in \{1,2\} \, ,
\end{equation}
and unit is taken to be $\opid_{\mcC_\mcA} = (A,\tau_1^A,\tau_2^A,\taubar{1}^A,\taubar{2}^A)$, where
\begin{equation}
\label{eq:CA_prod_crossings_unit}
\tau_i^A = \pic[1.25]{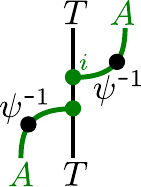} \quad , \qquad
\taubar{i}^A = \pic[1.25]{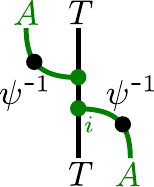} \, , \quad i \in \{1,2\} \, .
\end{equation}
An orbifold datum $\mcA$ is called \textsl{simple} if $\mcC_\mcA$ is fusion, i.\,e.\ $\dim_\Bbbk \mcC_\mcA(\opid_{\mcC_\mcA},\opid_{\mcC_\mcA}) = 1$.
\item
$\mcC_\mcA$ is spherical where the dual $\mcX^*$ of an object $\mcX\in\mcC_\mcA$ is defined to have $X^*$ as the underlying $A$-$A$-bimodule and the crossing morphisms are induced from those of $\mcX$.
The adjunction maps are:
\begin{align}
\ev_\mcX = \pic[1.25]{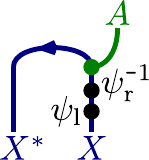} \,   , \qquad 
\coev_\mcX = \pic[1.25]{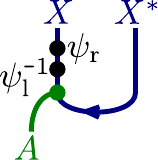} \, , 
\\ 
\tev_\mcX = \pic[1.25]{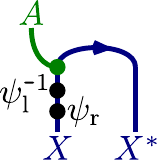}  \, , \qquad
\tcoev_\mcX = \pic[1.25]{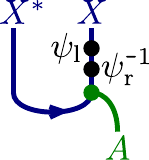} \, .
\end{align}
If $\mcA$ is a simple orbifold datum one can compute the trace of a morphism $f\in\End_{\mcC_\mcA}(\mcX)$ as follows:
\begin{equation}
\label{eq:trace_in_CA}
    \tr_{\mcC_\mcA} f = \frac{\tr_\mcC(\omega^2 \circ f)}{\tr_\mcC \psi^4} \, ,
\end{equation}
with $\tr_\mcC \psi^4 \neq 0$ holding automatically.
On the right-hand side of \eqref{eq:trace_in_CA}, $f$ is treated as a morphism in $\End_\mcC (A)$.
\item
$\mcC_\mcA$ is braided with the braiding morphisms of $\mcX,\mcY\in\mcC_\mcA$ given by
\begin{equation}
c_{\mcX,\mcY} = \phi^2 \cdot \pic[1.25]{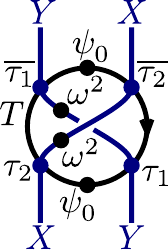} \, , \qquad
c^{-1}_{\mcX,\mcY} = \phi^2 \cdot \pic[1.25]{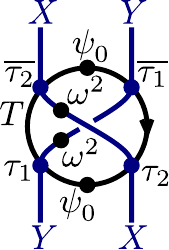} \, .
\end{equation}
If $\mcA$ is a simple orbifold datum, one has (\cite[Thm.\,3.17]{MuleRunk}):
\begin{theorem}
\label{thm:CA_is_modular}
For a simple orbifold datum $\mcA$ in $\mcC$, $\mcC_\mcA$ is a modular fusion category with
\begin{equation}
\label{eq:dimCA}
\dim \mcC_\mcA = \frac{\dim \mcC}{ \phi^8 \cdot (\tr_\mcC \psi^4)^2 } \, .
\end{equation}
\end{theorem}

It will be important in our main comparison theorem in Section~\ref{sec:proof} that a simple orbifold datum even gives a preferred choice of the square root \eqref{eq:DC} of the global dimension of $\mcC_\mcA$:
\begin{equation}
\label{eq:DCA}
\mathscr{D}_{\mcC_\mcA} := \frac{\mathscr{D}_\mcC}{\phi^4 \cdot \tr_\mcC \psi^4} \, .
\end{equation}
In the Lemma~\ref{lem:C_CA_anomalies} below we will also show that the pairs $(\mcC,\mathscr{D}_\mcC)$ and $(\mcC_\mcA,\mathscr{D}_{\mcC_\mcA})$
have the same anomalies, i.\,e.\
\begin{equation}
    \delta_\mcC = \delta_{\mcC_\mcA} \,.
\end{equation}

\item
Given two objects $\mcX,\mcY\in\mcC_\mcA$ and any morphism $f\colon X\lra Y$ between the underlying objects in $\mcC$, the averaged morphism defined by
\begin{equation}
\label{eq:f_average}
\overline{f} = \phi^4 \cdot \pic[1.25]{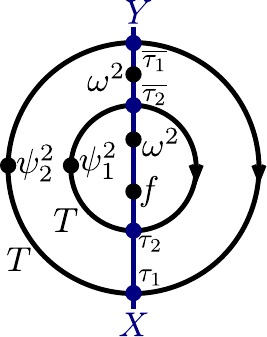}
\end{equation}
is a bimodule morphism and satisfies \eqref{eq:CA_morphism_cond}, i.\,e. it is a morphism in $\mcC_\mcA$.
In fact, the mapping $f\lmt\overline{f}$ acts as an idempotent projecting onto $\mcC_\mcA(\mcX,\mcY)$ seen as a subspace of $\mcC(X,Y)$.
\item
A useful tool to work with the category $\mcC_\mcA$ is the \textsl{pipe functor} 
\be \label{eq:P-functor}
P\colon\mcACA \lra \mcC_\mcA 
\ee 
from \cite[Sect.\,3.3]{MuleRunk}. 
It is biadjoint to the forgetful functor $\mcC_\mcA \lra \mcACA$ and can be seen as the free construction of an object of $\mcC_\mcA$.
For a bimodule $X\in\mcACA$ and a bimodule morphism $f\colon X \lra Y$, the underlying $A$-$A$-bimodule over the object $P(X)$ as well the morphism $P(f)$ are defined respectively by
\begin{equation}
\label{eq:pipe_functor}
\im\pic[1.25]{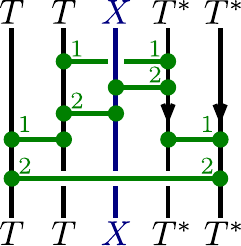} \qquad \text{and} \qquad \pic[1.25]{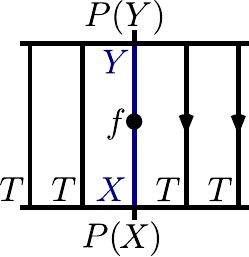} \,, 
\end{equation}
where the horizontal lines in the definition of $P(f)$ denote the inclusion into and projection onto the image defining $P(X)$ and $P(Y)$.
\end{enumerate}

\begin{remark}
The definition~\eqref{eq:pipe_functor} of the pipe functor has a more intuitive interpretation in terms of the Reshetikhin--Turaev defect TQFT as depicted in Figure~\ref{fig:pipe_strand} below.
We review this in more detail later in Section~\ref{subsec:RT_defect_TQFT}.
More generally, pipe functors can be analogously defined for arbitrary 3-dimensional defect TQFTs and the corresponding notion of an orbifold datum in it.
We do this in Appendix~\ref{app:pipe_functors}.
\end{remark}

\section{Reshetikhin--Turaev theory}
\label{sec:RT_theory}

We review the definitions and some properties of three variants of Reshetikhin--Turaev type TQFTs: 
the graph TQFT (Section~\ref{subsec:RT_graph_TQFT}), the defect TQFT (Section~\ref{subsec:RT_defect_TQFT}), and the orbifold graph TQFT (Section~\ref{subsec:RT_orbifold_graph_TQFT}).
The latter is an explicit example of the orbifold graph TQFTs introduced in~\cite{CMRSS1}.

\subsection{Reshetikhin--Turaev graph TQFT}
\label{subsec:RT_graph_TQFT}

We fix a modular fusion category $\CC$. The Reshetikhin--Turaev graph TQFT is a symmetric monoidal functor 
\be 
\zrt \colon \Bordriben{3}(\CC) \lra \Vect ~,
\ee 
and we give a brief description of the source category and of the construction of $\zrt$ below. 
A detailed description can be found in \cite[Ch.\,IV]{turaevbook}.

\subsubsection*{Bordisms with ribbon graphs}

We use the terminology and conventions of
\cite[Sect.\,2.3\,\&\,3.2]{CMRSS1} (adapted from \cite{TVireBook})
in which the category $\Bordribn{3}(\CC)$ is said to have \textsl{$\mcC$-coloured punctured surfaces} as objects and \textsl{$\mcC$-coloured ribbon bordisms} as morphisms.
A $\mcC$-coloured ribbon bordism is a pair $(M,\mcR)$ consisting of an oriented $3$-dimensional bordism $M$ and an embedded $\mcC$-coloured ribbon graph $\mcR$.
A $\mcC$-coloured punctured surface is an oriented $2$-dimension manifold with a finite set of distinguished points called \textsl{punctures} (or \textsl{marked points}).
The punctures serve as possible endpoints of strands of embedded ribbon graphs and are labelled by triples $(X,v,\epsilon)$, where $X\in\mcC$ would label the adjacent strand, $v$ is a tangent vector to account for its framing, and $\epsilon$ is a sign indicating its direction.

\medskip 

The \RT\ TQFT~$\zrt$ acts on an extension $\Bordriben{3}(\CC)$ of the category $\Bordribn{3}(\CC)$, which is defined as follows. 
Objects of $\Bordriben{3}(\CC)$ are pairs $(\Sigma,\lambda)$, where~$\Sigma$ is a $\CC$-coloured punctured surface and $\lambda\subset H_1(\Sigma,\R)$ is a Lagrangian subspace. 
A morphism $(\Sigma,\lambda) \lra (\Sigma',\lambda')$ in $\Bordriben{3}(\CC)$ is a pair $(M,n)$, where $M\colon\Sigma\lra\Sigma'$ is a morphism in $\Bordribn{3}(\CC)$, and $n\in\Z$. 
The unit of an object $(\Sigma,\lambda)$ is the cylinder $(\Sigma \times [0,1], 0)$.
The Lagrangian subspaces are only needed for composition: 
for $(M,n)\colon (\Sigma,\lambda) \lra (\Sigma',\lambda')$ and $(M',n')\colon(\Sigma',\lambda') \lra (\Sigma'',\lambda'')$, their composition in $\Bordriben{3}(\CC)$ is given by
\be
\label{eq:Bordhat_composition}
(M',n') \circ (M,n) 
= 
\Big( M' \circ M, n+n' - \mu\big( M_*(\lambda), \lambda', {M'}^*(\lambda'') \big) \Big)
\, , 
\ee 
where~$\mu$ denotes the Maslov index of the corresponding Lagrangian subspaces associated to~$\Sigma'$, while $M_*$ and~$M'^*$ denote the Lagrangian actions induced by~$M$ and~$M'$ on $H_1(\Sigma,\R)$ and $H_1(\Sigma'',\R)$, respectively, see \cite[Sect.\,IV.3\,\&\,IV.4]{turaevbook}. 
The monoidal product of $\Bordriben{3}(\CC)$ is given by disjoint union, where for bordisms the integers are added. The symmetric structure of $\Bordribn{3}(\CC)$ is analogous to that of $\Bordriben{3}(\CC)$.

\subsubsection*{Invariants of closed manifolds}
Since~$\CC$ is a ribbon category, one has a functor $F_\CC\colon \textrm{Rib}_\CC \lra \CC$, where $\textrm{Rib}_\CC$ is the category of $\CC$-coloured ribbon graphs.
In particular, for every closed $\CC$-coloured ribbon graph $\mathcal R\colon \varnothing\lra\varnothing$ one obtains a scalar invariant $F_\CC(\mathcal R) \in \End(\one) \cong \Bbbk$.

\medskip

Recall, e.\,g.\ from \cite[Sect.\,II.2.1]{turaevbook}, that all compact connected closed oriented $3$-manifolds can be obtained by performing surgery along an (uncoloured and undirected) ribbon link in $S^3$.
The procedure is the following:
Given a ribbon link $L = L_1 \sqcup \dots \sqcup L_k \subseteq S^3$, let $U\subseteq S^3$ be a tubular neighbourhood of $L$ with disjoint components $U_i$, such that $L_i$ lies in the centre of $U_i$.
Let $\g_i$ be a curve on the boundary of the
closure of $U_i$ induced by the framing of the component $L_i$ and let
\begin{equation}
\varphi: \bigsqcup_i S^1 \times S^1 \lra -\pd(S^3 \setminus U)
\end{equation}
be a diffeomorphism, such that $\varphi(S^1 \times \{1\}) = \g_i$.
The manifold
\begin{equation}
M := \left( \bigsqcup_i B^2 \times S^1 \right) \cup_\varphi S^3\setminus U
\end{equation}
is then said to be obtained by surgery on $S^3$ along $L$.
In other words, $M$ is obtained from $S^3$ by cutting out solid tori and glueing them back via diffeomorphisms of their boundaries.
We note that it is possible for different links to yield the same manifold, e.\,g.\ $S^3$ can be obtained both from the empty link or from the link with one component which has a single twist.
In general, two links yield the same 3-manifold iff they are related by a finite sequence of the so-called Kirby moves, see e.\,g.\ \cite[Sect.\,II.3.1]{turaevbook}.

\medskip

If~$M$ is presented by surgery along a framed link $L\subset S^3$, we obtain a $\CC$-coloured ribbon graph $\mathcal L_\CC$ in~$S^3$ by labelling each component of~$L$ with the object
\be
\label{eq:Kirby_obj}
K_\CC := \bigoplus_{i\in\textrm{Irr}_{\CC}} i \in \CC \, , 
\ee 
and adding a single vertex labelled by the morphism 
\begin{equation}
\label{eq:Kirby_col}
\xi_\CC  := \bigoplus_{i\in\textrm{Irr}{\CC}} \dim_\CC(i) \cdot \id_i \colon K_\CC \lra K_\CC 
\end{equation}
on each component circle. From~$\mathcal L_\CC$ one then obtains the invariant (see~\cite[Thm.\,II.2.2.2]{turaevbook})
\begin{equation}
\label{eq:tau_M}
\tau(M) := \delta_\CC^{-\sigma(L)} \cdot \mathscr D_\CC^{-|L|-1} \cdot F_\CC(\mathcal L_\CC) \, , 
\end{equation}
where $\sigma(L)$ is the signature of the linking number matrix of~$L$, $|L|$ is the number of components of $L$, and the scalars~$\delta_\CC$ and~$\mathscr D_\CC$ are as in~\eqref{eq:DC} and~\eqref{eq:anomaly}, respectively. 

The invariant~$\tau$ is readily lifted to extended $\CC$-coloured ribbon bordisms of the form $(M,n)\colon \varnothing\lra\varnothing$. 
Indeed, if~$M$ comes with an embedded $\CC$-coloured ribbon graph~$\mathcal R$, then the same formula as in~\eqref{eq:tau_M} applies except for an additional anomaly factor~$\delta_\CC^n$, and the evaluation functor~$F_\CC$ is instead applied to the $\CC$-coloured ribbon graph obtained by combining~$\mathcal L_\CC$ and~$\mathcal R$, which by abuse of notation we denote $\mathcal L_\CC \sqcup \mathcal R$ (even though~$\mathcal L_\CC$ and~$\mathcal R$ may be non-trivially entangled): 
\begin{equation}
\label{eq:tau_Mn}
\tau(M,n) := \delta_\CC^{n-\sigma(L)} \cdot \mathscr D_\CC^{-|L|-1} \cdot F_\CC(\mathcal L_\CC \sqcup \mathcal R) \, ,
\end{equation}
see \cite[Thm.\,II.2.3.2]{turaevbook} and \cite[Sect.\,IV.9.2]{turaevbook} for more details.

\subsubsection*{Graph TQFT}

We are now ready to define the Reshetikhin--Turaev graph TQFT~$\zrt$. 

For a given object $\Sigma = (\Sigma,\lambda) \in \Bordriben{3}(\CC)$, let $\mathcal{T}(\Sigma)$ (respectively $\mathcal{T}'(\Sigma)$) be the $\Bbbk$-vector space freely generated by extended $\CC$-coloured ribbon bordisms of the form $\varnothing \lra \Sigma$ (respectively $\Sigma \lra \varnothing$). 
In terms of the pairing 
\begin{align}
\beta_\Sigma \colon  \mathcal{T}'(\Sigma) \otimes_\Bbbk \mathcal{T}(\Sigma) & \lra \Bbbk 
\nonumber
\\
M' \otimes_\Bbbk M & \lmt  \tau(M' \circ M)
\end{align}
one defines the state space associated to~$\Sigma$ to be 
\begin{equation}\label{eq:zrt-state-space-quotient}
\zrt(\Sigma) := \mathcal{T}(\Sigma) \big/ \textrm{rad}_{\textrm{r}} \beta_\Sigma \, , 
\end{equation}
where $\textrm{rad}_{\textrm{r}} \beta_\Sigma$ is the right radical of the pairing $\beta_\Sigma$, consisting of those elements $v\in\mathcal T(\Sigma)$ with $\beta_\Sigma(w,v) = 0$ for all $w\in \mathcal T'(\Sigma)$.

For a morphism $M = [(M,n) \colon \Sigma \lra \Sigma']$ in $\Bordriben{3}(\CC)$, the action of~$\zrt$ is defined as  
\begin{align}
\zrt(M) \colon \zrt(\Sigma) & \lra \zrt(\Sigma')
\nonumber 
\\ \label{eq:zrt-functor}
\big[\varnothing \stackrel{H}{\lra} \Sigma \big] &  \lmt  \big[\varnothing \stackrel{H}{\lra} \Sigma \stackrel{M}{\lra} \Sigma'\big] \, . 
\end{align}

\begin{remark}
\label{remark:univ-constr}
	\begin{enumerate}
		\item
		The definition above uses the universal construction of \cite{BHMV}.
		It does not in general guarantee that the functor obtained this way has a symmetric monoidal structure, but this was proven in \cite{DGGPR} for a generalisation of the \RT\ TQFT obtained from a not necessarily semisimple modular tensor category. 
		This generalisation yields the above construction in the semisimple case.
		The proof of monoidality fails if the objects of $\Bordribn{3}(\CC)$ are not equipped with Lagrangian subspaces, so the use of the extended category $\Bordriben{3}(\CC)$ is necessary (see \cite{GW} for more details on this).
		\item
		The original definition of the \RT\ TQFT in \cite[Ch.\,IV]{turaevbook} gives a symmetric monoidal functor $\Bordriben{3}(\CC)\lra\Vect$, but it does not use the universal construction, so a priori it is not clear that the TQFT of \cite{turaevbook} and~$\zrt$ as defined above are isomorphic.
		However this follows from Lemma~\ref{lem:funct_iso} below, as the two TQFTs give the same invariants on closed manifolds (cf.\ \cite[Thm.\,II.2.3.2]{turaevbook}) and have isomorphic state spaces (compare \cite[(IV.\,1.4.a)]{turaevbook} and \cite[Prop.\,4.16]{DGGPR}).
	\end{enumerate}
\end{remark}

Next we summarise some properties of $\zrt$ that will be needed below.

\begin{property}
\label{ZRTProp:regular}
From the definition it follows that the graph TQFT $\zrt$ is linear with respect to addition and scalar multiplication of coupons of embedded ribbon graphs.
Moreover, upon evaluation one can compose the coupons as well as remove a coupon labelled with an identity morphism, i.\,e.\ the following exchanges are allowed:
\begin{equation}\label{eq:regular-TQFT-def}
\pic[1.25]{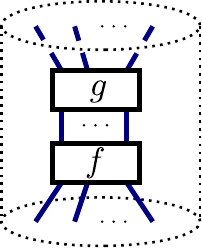}
\leftrightarrow
\pic[1.25]{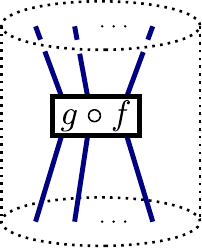} , \; 
\pic[1.25]{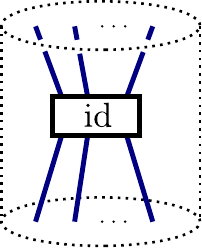}
\leftrightarrow
\pic[1.25]{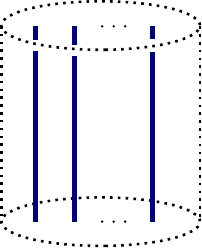}.
\end{equation}
By the terminology of \cite[Sect.\,15.2.3]{TVireBook}, graph TQFTs with this property are called \textsl{regular}. In addition, one can replace crossings etc.\ by the corresponding morphisms, in other words, one can perform graphical calculus on the embedded ribbon graphs.
For example, if some part $\mcR'$ of a ribbon graph $\mcR$ in a $3$-bordism $M$ is contained in the interior of an embedded closed 3-ball, then 
it can be exchanged for a single coupon labelled with $F_\mcC(\mcR')\in\End_\mcC(\opid)$ without changing the value of $\zrt$.
\end{property}
\begin{property}
\label{zrt:coupons_in_spheres}
Let $S^2_P \in \Bordriben{3}(\mcC)$ be a $2$-sphere with $n$ negatively oriented punctures labelled by objects $X_1,\dots,X_n\in\mcC$ and $m$ positively oriented punctures labelled by $Y_1,\dots,Y_m\in\mcC$.
One has the following isomorphism of vector spaces:
\begin{equation}
\label{eq:RTSPhereSpaces}
\begin{array}{c}
\mcC(X_1\otimes\cdots\otimes X_n, Y_1\otimes\cdots\otimes Y_m) \cong \zrt(S^2_P)\vspace{0.5cm}\\
f  \lmt  [B_f\colon\varnothing\lra S^2_P]
\end{array} \;\; \text{ where }
B_f = \pic[1.25]{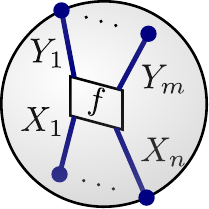}
\end{equation}
is a solid ball with an $f$-labelled coupon.
This follows from the description of the TQFT via the universal construction, see Remark \ref{remark:univ-constr}.
In general, the state space of any punctured surface is spanned by solid handlebodies with embedded ribbon graphs.
\end{property}
\begin{property}
The invariants of some simple examples of closed $3$-manifolds are:
\begin{align}
\label{eq:ZRTCS3}
\zrt(S^3,0) = \mathscr{D}_\mcC^{-1} \, , \qquad
\zrt(S^2 \times S^1,0) = 1 \, ,\\
\zrt(S^1 \times S^1 \times S^1,0) = |\operatorname{Irr}_\mcC| 
~ \in \Bbbk \, .
\end{align}
The $S^3$-invariant allows one to rewrite the formula \eqref{eq:tau_Mn} for the invariant of a ribbon $3$-manifold $(M,\mcR,n)$ represented by the surgery link $L\sqcup\mcR$ in $S^3$ as
\begin{equation}
\label{eq:M_inv_as_S3_inv}
\zrt(M,\mcR,n) = \delta_\CC^{n-\sigma(L)} \cdot \mathscr D_\CC^{-|L|} \cdot \zrt(S^3, \mcL_\mcC \sqcup \mcR, 0) \, .
\end{equation}
The other two invariants are simply instances of a more general observation:
Due to the monoidal functoriality of $\zrt$ (or in fact any graph TQFT), for any surface $\Sigma\in\Bordriben{3}(\mcC)$ the invariant of $\Sigma\times S^1$ is the trace of the identity map on $\zrt(\Sigma)$.
If $\operatorname{char} \Bbbk = 0$, it is equal to the dimension of the state space of $\Sigma$, otherwise to the dimension modulo $\operatorname{char} \Bbbk$, i.\,e.
\begin{equation}
\label{eq:dim_ss_as_invariant}
\zrt(\Sigma \times S^1,0) = \dim \zrt(\Sigma) 
~ \in \Bbbk \, .
\end{equation}
We have already seen explicitly in~\eqref{eq:RTSPhereSpaces} that $\dim \zrt(S^2) = 1$, and the state space of a punctureless torus $S^1\times S^1$ has a basis consisting of solid tori with non-contractible loops labelled by  $i\in\operatorname{Irr}_\mcC$.
\end{property}

\subsection{Reshetikhin--Turaev TQFT with line and surface defects}
\label{subsec:RT_defect_TQFT}
In this section we outline the construction of a defect TQFT 
\be 
\zzc \colon \Borddefen{3}(\D^{\CC}) \lra \Vect \, , 
\ee 
which lifts the \RT\ TQFT~$\zrt$ reviewed in Section~\ref{subsec:RT_graph_TQFT} to decorated bordisms with embedded ribbons as well as embedded surfaces. 
The details of the definition of~$\zzc$ are the main content of \cite{CRS2}.

\subsubsection*{Defect bordisms for Reshetikhin--Turaev TQFT}
Let us briefly describe the source category $\Borddefen{3}(\D^{\CC})$ of the defect TQFT $\zzc$.

For details on defect bordisms and TQFTs we refer to \cite{CMS, CRS1} and the summary in \cite[Sect.\,3]{CMRSS1}.
Here we just mention that the objects and morphisms of the category $\Borddef_3(\D)$ of \textsl{defect surfaces} and \textsl{defect bordisms} are stratified surfaces and $3$-dimensional bordisms, whose strata are oriented 
and carry labels from a so-called set of \textsl{defect data} $\D$.
The latter consists of sets $D_0, D_1, D_2, D_3$ where elements of $D_j$ are used to label the $j$-dimensional strata of defect bordisms, as well as information about which labels can be assigned to adjacent strata.
For every defect TQFT $\zz\colon\Borddef_3(\D)\lra\Vect$ there is the so-called $D_0$-completion, where one expands the set $D_0$ by labelling the $0$-strata with elements of state spaces of small stratified $2$-spheres surrounding them, see \cite[Sect.\,2.4]{CRS1}.

\medskip

For a modular fusion category $\CC$, there is a set of defect data~$\D^\CC$ with 
\begin{itemize}
\item $D^\mcC_3$: $3$-strata have no labels, or equivalently they all carry the same label $\mcC$.
\item $D^\mcC_2$: $2$-strata are labelled by $\Delta$-separable symmetric Frobenius algebras in~$\mcC$.
\item $D^\mcC_1$: 
all $1$-strata are assigned a framing.
A $1$-stratum that has no adjacent $2$-strata is labelled by an object of $\mcC$.

Suppose a $1$-stratum $L$ has $n>0$ adjacent $2$-strata. 
Let the labels of them be $\Delta$-separable symmetric Frobenius algebras $A_1,\dots, A_n$.
We then label $L$ with an $A_1$-$\cdots$-$A_n$-\textsl{multimodule} $M\in\mcC$, i.\,e.\ a module over 
$A_1 \otimes \cdots \otimes A_n$, or equivalently a module over all the algebras $A_i$ simultaneously such that for $i\neq j$ the $A_i$- and $A_j$-actions commute with respect to the braiding (as in \eqref{eq:AA_actions_commute} for $i=1$, $j=2$).
The framing of an $M$-labelled $1$-stratum will be taken to be the 
vector field normal to the adjacent $2$-stratum with the label $A_1$.
\item $D^\mcC_0$: $0$-strata are added by the usual completion procedure mentioned above.
\end{itemize}

\begin{remark}
\begin{enumerate}[(i)]
\item
In \cite{CRS2} a slightly different set of labels $D_1$ for $1$-strata is used.
In particular, the multimodules are equipped with a so-called \textsl{cyclic structure}, the purpose of which is to eliminate the need for $1$-strata to have framings.
\item
The graph TQFT $\zrt$ from Section~\ref{subsec:RT_graph_TQFT} is a special case of of the defect TQFT $\zzc$ where only line defects are allowed, see~\cite[Rem.\,5.9\,(i)]{CRS2}, \cite[Sect.\,4.2]{CMS}.
\item
It is possible to further generalise the defect TQFT $\zzc$ so that the 3-strata can be labelled by different modular fusion categories, see~\cite{KMRS}.
\end{enumerate}
\end{remark}

The extension $\Borddefen{3}(\D^\CC)$ relates to $\Borddefen{3}(\D^\CC)$ in the same way as $\Bordriben{3}(\CC)$ relates to $\Bordriben{3}(\CC)$, i.\,e.\ the objects (defect surfaces) are paired with Lagrangian subspaces, and the morphisms (defect bordisms) are paired with integers so that the composition is as in \eqref{eq:Bordhat_composition}.
For convenience we will often suppress these extra data of the extension below, as they only interact trivially with our arguments.

\subsubsection*{Ribbonisation}
We will define the \RT\ defect TQFT~$\zzc$ in terms of the ``ribbonisation'' map 
\be
\label{eq:ribbonisation}
R \colon \Borddefen{3}(\D^{\CC}) \lra \Bordriben{3}(\CC) \, ,
\ee 
which, as will momentarily become apparent, is not a functor since it does not preserve identity morphisms and depends on auxiliary data.

We use the terminology of \cite{CMRSS1} where a \textsl{skeleton} (here also \textsl{$1$-skeleton}) of an oriented $2$-manifold $F$ is a stratification, dividing $F$ into a finite collection of open discs and (in case $F$ has boundary) half-discs, and each $0$-stratum has exactly three adjacent $1$-strata.
A skeleton is then called \textsl{admissible} if its $0$- and $1$-strata are oriented by a \textsl{local order} on the $2$-strata, meaning that the $1$-strata adjacent to a $0$-stratum cannot all be directed towards or away from it.
A choice of a skeleton on $F$ induces a stratification on the boundary $\pd F$, consisting of a finite set of oriented points dividing it into a collection of open intervals; we will call such a stratification a \textsl{$0$-skeleton} of $\pd F$.
Any $0$-skeleton on $\pd F$ can be extended to a $1$-skeleton on $F$.

\medskip

Let $\Sigma\in\Borddefen{3}(\D^{\CC})$ be a defect surface.
Pick a set $\tau$ of $0$-skeleta for the $1$-strata of $\Sigma$.
We define 
\begin{equation}
    R(\Sigma,\tau) \in \Bordriben{3} (\CC)
\end{equation}
to be the underlying oriented surface $\Sigma$ with the set of punctures $\Sigma^{(0)} \cup \bigcup_L \tau(L)$, where $\Sigma^{(0)}$ is the set of $0$-strata of $\Sigma$ and $L$ runs over the $1$-strata of $\Sigma$, 
and $\tau(L)$ is the set of vertices in the chosen 0-skeleton of~$L$. 
For a $1$-stratum $L\subseteq\Sigma$ labelled with $A\in D^\mcC_2$, each point $p\in\tau(L)$ is made into a puncture by decorating it with triple $(A,v,\epsilon)$, where $v\in T_p\Sigma$ is a tangent vector normal to the orientation of $L$, and $\epsilon$ is the orientation of $p$.

\medskip

For a defect bordism $[(M,n)\colon\Sigma\lra\Sigma']\in\Borddefen{3}(\D^{\CC})$, let $t$ be a set of $1$-skeleta for the $2$-strata of $\Sigma$ which restricts to sets of $0$-skeleta $\tau, \tau'$ for the $1$-strata of $\Sigma, \Sigma'$.
We define
\begin{equation}
    \big[R(M, t)\colon R(\Sigma,\tau) \lra R(\Sigma',\tau')\big] \in \Bordriben{3} (\CC)
\end{equation}
to be the underlying $3$-manifold $M$ with an embedded ribbon graph $R$ obtained as follows: the strands are the (already framed, oriented and labelled) $1$-strata of $M$ and the $1$-strata of the skeleta $t(F)$ for each $2$-stratum $F\subseteq M$, framed by tangent vectors normal to $F$ and labelled by the same algebra object $A\in D^\mcC_2$ as $F$.
The $0$-strata of $t(F)$ are then replaced by coupons labelled with the (co)multiplication of $A$, while the intersection points of $t(F)$ with a $1$-stratum $L$ of $M$ are labelled by the appropriate (co)action of $A$ on the multimodule labelling $L$.

\subsubsection*{\RT\ defect TQFT} 

Let~$\Sigma$ be an object in $\Borddefen{3}(\D^\CC)$, and let $\tau,\tau'$ be two sets of 0-skeleta for its 1-strata.
Consider the defect cylinder $C_\Sigma := (\Sigma \times [0,1], 0)$.
For any choice of a set $t$ of $1$-skeleta for the $2$-strata of $C_\Sigma$ which restricts to $\tau$, $\tau'$ on the boundary, we obtain a linear map 
\be 
\Phi_\tau^{\tau'} := \Big[\zrt \big( R(C_\Sigma,t) \big) \colon \zrt \big( R(\Sigma,\tau) \big) \lra \zrt\big( R(\Sigma,\tau') \big)\Big] \,.
\ee 
By the properties of $\Delta$-separable symmetric Frobenius algebras, $\Phi_\tau^{\tau'}$ depends only on $\tau, \tau'$ and not on the choice of $t$ in the interior of $C_\Sigma$ (see e.\,g.\ \cite[Sect.\,5.1]{tft1}).
This implies
\be
\Phi_{\tau'}^{\tau''} \circ \Phi_{\tau}^{\tau'} = \Phi_{\tau}^{\tau''} \, ,
\ee 
and in particular that each map $\Phi_{\tau}^{\tau}$ is an idempotent.

\medskip

With this preparation, we can now reformulate the construction in \cite[Thm.\,5.8\,\&\,Rem.\,5.9]{CRS2}:

\begin{construction}
\label{constr:DefectRT}
Let~$\mathcal C$ be a modular fusion category. 
The \RT\ defect TQFT 
\be 
\label{eq:zzcNoHat}
\zzc \colon \Borddefen{3}(\D^\CC) \lra \Vect
\ee  
is defined as follows: 
\begin{enumerate}
\item 
For an object $\Sigma \in \Borddefen{3}(\D^\CC)$, we set 
\be 
\zzc(\Sigma) = \textrm{colim} \big\{ \Phi_{\tau}^{\tau'} \big\} \, , 
\ee 
where $\tau, \tau'$ range over all sets of $1$-skeleta for the 1-strata of $\Sigma$. 
\item 
For a morphism $M \colon \Sigma \lra \Sigma'$ in $\Borddefn{3}(\D^\CC)$, we set $\zzc (M)$ to be 
\be 
\label{eq:ZConMorphisms}
\begin{tikzcd}
\zzc(\Sigma) \arrow[hook]{r} 
& \zrt \big( R(\Sigma, \tau) \big) \arrow{rr}{\zrt( R(M,t) )}
& 
& \zrt\big(  R(\Sigma', \tau´') \big) \arrow[two heads]{r}
& \zzc(\Sigma') \, , 
\end{tikzcd}  
\ee 
where $t$ is an arbitrary set of $1$-skeleta for the $2$-strata of $M$ that restricts to sets of $0$-skeleta $\tau$ and $\tau'$ for the $1$-strata of $\Sigma$ and $\Sigma'$, respectively. The injection is part of the universal property of the colimit, and the surjection is part of its data.
\end{enumerate}
\end{construction}

Note that the definition of $\zzc$ in~\eqref{eq:ZConMorphisms} does not depend on the choice $t$ of $1$-skeleta. 
Moreover, by construction the state spaces of~$\zzc$ are isomorphic to the images of the idempotents $\Phi_\tau^\tau$, 
\be
\label{eq:ZC_sos_as_im}
\zzc(\Sigma) \cong \im \Phi_{\tau}^{\tau}  \, . 
\ee 

Next we collect some properties of~$\zzc$ that we will need in Section~\ref{sec:proof}.

\begin{figure}
	\captionsetup{format=plain}
	\centering
	\begin{subfigure}[b]{0.32\textwidth}
		\centering
		\pic[1.25]{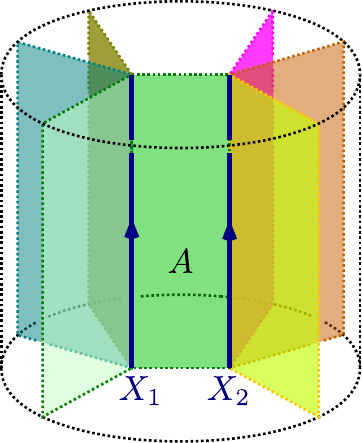}
		\caption{}
		\label{fig:local_neigh_two_lines}
	\end{subfigure}
	\begin{subfigure}[b]{0.32\textwidth}
		\centering
		\pic[1.25]{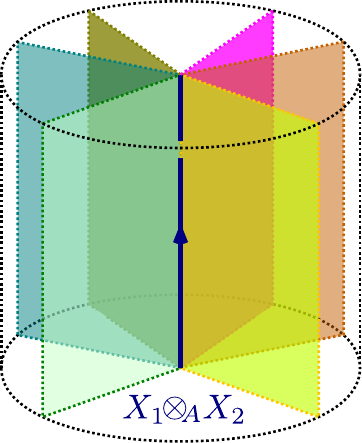}
		\caption{}
		\label{fig:local_neigh_fused_lines}
	\end{subfigure}
	\begin{subfigure}[b]{0.32\textwidth}
		\centering
		\pic[1.25]{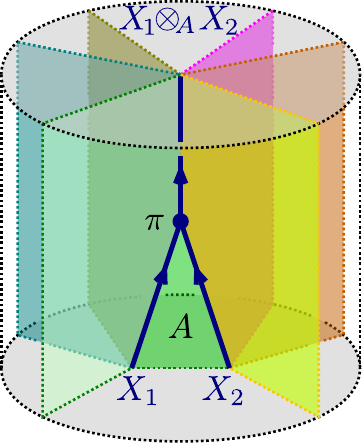}
		\caption{}
		\label{fig:local_neigh_two_lines_fusion_cyl}
	\end{subfigure}
	\caption{
(a) Part of a defect bordism containing two parallel line defects labelled by $X_1, X_2$. (b) The line defects can be fused to a single line defect labelled by $X_1 \otimes_A X_2$. (c) Evaluating the cylinder containing the projection morphism $\pi\colon X_1\otimes X_2 \lra X_1 \otimes_A X_2$ with $\zzc$  provides an isomorphism between the two state spaces. Its inverse is obtained from a similar cylinder containing the inclusion $\imath\colon X_1 \otimes_A X_2 \lra X_1 \otimes X_2$.
	}
	\label{fig:line_fusion}
\end{figure}
\begin{figure}
	\captionsetup{format=plain}
	\centering
	\begin{subfigure}[b]{0.32\textwidth}
		\centering
		\pic[1.25]{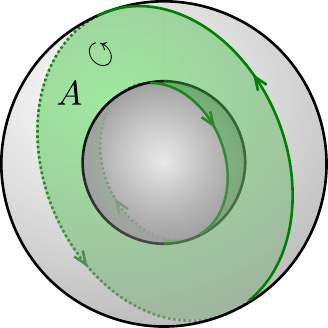}
		\caption{}
		\label{fig:S2wloop_cylinder}
	\end{subfigure}
	\begin{subfigure}[b]{0.32\textwidth}
		\centering
		\pic[1.25]{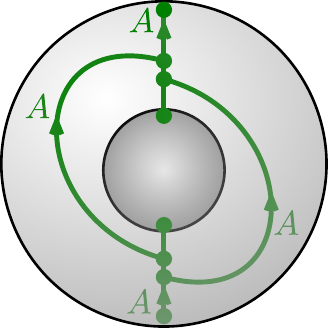}
		\caption{}
		\label{fig:S2wloop_cylinder_rib}
	\end{subfigure}
	\begin{subfigure}[b]{0.32\textwidth}
		\centering
		\pic[1.25]{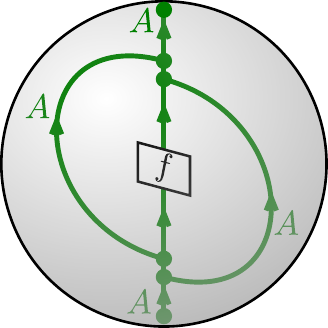}
		\caption{}
		\label{fig:S2wloop_cylinder_rib_img}
	\end{subfigure}
	\caption{%
	     (a) A cylinder over a defect sphere~$S$ with a single $A$-labelled 1-stratum. 
	     (b) A choice or ribbonisation of the cylinder. 
	     (c) An element of the state space $\zzc(S)$ is obtained from $f\in\mcC(A,A)$ by applying the projection $\mcC(A,A) \lra \mcACA(A,A)$, which sends~$f$ to the string diagram displayed in the picture. 
	}
	\label{fig:S2wloop_space}
\end{figure}
\begin{figure}
	\captionsetup{format=plain}
	\begin{equation*}
	\zzc\left( \pic[1.25]{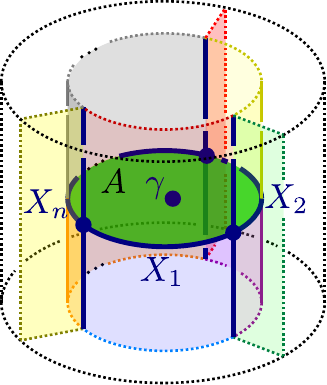} \right) =
    \zzc\left( \pic[1.25]{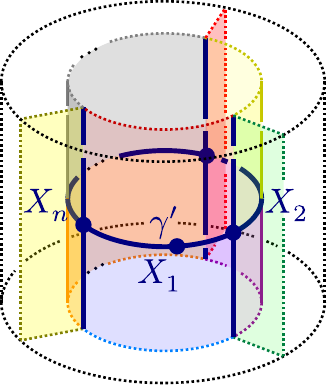} \right) ~.
	\end{equation*}
	\caption{
	     Upon evaluation with~$\zzc$, contractible $A$-labelled 2-strata~$F$ in the interior
	     can be removed by forgetting all $A$-actions on the multimodules~$X_i$ which decorate adjacent 1-strata. 
	     If~$F$ is punctured by a $\gamma$-labelled 0-stratum~$p$, then both~$F$ and~$p$ can be removed at the cost of puncturing one of the 1-strata adjacent to~$F$ with a new 0-stratum whose label~$\gamma'$ is obtained from~$\gamma$ and the $A$-action.
	}
	\label{fig:disc_remove}
\end{figure}

\begin{property}
Let $M$ be a defect bordism, and let $L_1, L_2$ be two parallel $1$-strata, connected by a $2$-stratum $F$.
Denote the multimodule which labels $L_i$ by $X_i$ and the algebra which labels $F$ by $A$ (see Figure \ref{fig:local_neigh_two_lines}).
Upon ribbonisation, $F$ gets replaced by a ribbon graph connecting $X_1$ and $X_2$, which upon evaluation with $\zrt$ can be replaced by the idempotent \eqref{eq:prod_over_A_proj}, projecting $X_1\otimes X_2$ onto $X_1 \otimes_A X_2$.
The same result can be obtained by replacing $L_1, L_2$ and $F$ with a single $1$-stratum labelled with $X_1 \otimes_A X_2$ (see Figure \ref{fig:local_neigh_fused_lines}).
If $L_1, L_2$ end on the boundary or on $0$-strata in $M$, the endpoints also need to be relabelled.
In this case, a cylinder as in Figure \ref{fig:local_neigh_two_lines_fusion_cyl} provides an isomorphism between the surfaces with different decorations.
\end{property}

\begin{property}
\label{ZCProp:compose_0-strata}
The set of labels for $0$-strata lying in a $2$-stratum labelled with an algebra $A\in D^\mcC_2$ is given by bimodule maps $\mcACA(A,A)$:
By definition it is the vector space assigned to a stratified $2$-sphere surrounding the $0$-stratum and having a single $A$-coloured loop.
Its cylinder and a choice for its ribbonisation are depicted in Figures \ref{fig:S2wloop_cylinder} and \ref{fig:S2wloop_cylinder_rib}.
By Property~\ref{zrt:coupons_in_spheres}, the Reshetikhin--Turaev TQFT assigns the space $\mcC(A,A)$ to the $2$-sphere with two $A$-coloured punctures such that an element $f\in\mcC(A,A)$ corresponds to a closed ball with an embedded $f$-labelled coupon and two $A$-strands.
The image of $f$ is then a closed ball with an embedded ribbon graph as in Figure \ref{fig:S2wloop_cylinder_rib_img}.
The two $A$-lines on the sides of the coupon act precisely as the projection $\mcC(A,A) \lra \mcACA(A,A)$.

\smallskip

\noindent
An analogous argument for $0$-strata separating two $1$-strata labelled with multimodules $X,Y$ yields that they are labelled with multimodule morphisms $X\lra Y$ (see \cite[Lem.\,3.1]{CRS3}).
This also implies that two $0$-strata which are connected by a $1$-stratum can be composed, e.\,g.\ for three multimodules $X, Y, Z$ and multimodule morphisms $f\colon X\lra Y$, $g \colon Y \lra Z$ one has
\begin{equation}
\zzc\left( \pic[1.25]{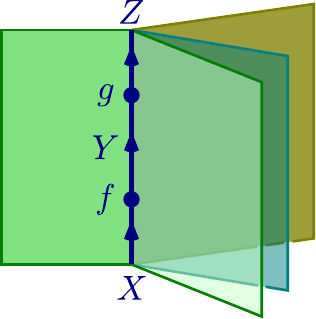} \right) =
\zzc\left( \pic[1.25]{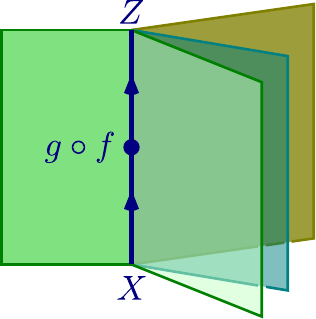} \right).
\end{equation}
Similarly, the $0$-strata within the same $2$-stratum can also be composed.
\end{property}

\begin{property}
\label{prop:GetRidOfDiscs}
Let $F$ be a $2$-stratum labelled with an algebra $A\in D^\mcC_2$, which is bounded by $1$-strata labelled by multimodules $X_1, X_2,\dots, X_n$ (and by some $0$-strata between them). Let~$p$ be a 0-stratum adjacent to no 2-stratum other than~$F$,
labelled by a bimodule morphism $\gamma\colon A\lra A$. 
Assume that $F\cup p$ is contractible, as illustrated on the left-hand side of Figure~\ref{fig:disc_remove}. 
Then, upon evaluation with $\zzc$, $F$ can be removed with $p$ pushed aside to form a $0$-stratum on one of the bounding 1-strata, for example the one labelled with $X_i$, and relabelled with the $\g'\colon X_i\lra X_i$ where $\g' = \g_{\textrm{l}}$ or $\g'=\g_{\textrm{r}}$ (recall the notation \eqref{eq:psi_omega_notation}), depending on which action of $A$ on $X_i$ is used.
The action of $A$ on $X_1, X_2, \dots, X_n$ is forgotten afterwards to make them into valid labels once $F$ is removed.
This is summarised in Figure~\ref{fig:disc_remove}, where $i=1$.
To show this identity, one uses the properties of $\Delta$-separable symmetric Frobenius algebras and their modules to simplify the ribbon graph that replaces $F$ after ribbonisation.
\end{property}

\subsection{Orbifold graph TQFT}
\label{subsec:RT_orbifold_graph_TQFT}
Let $\mcC$ be a modular fusion category, let $\mcA = (A,T,\a,\abar,\psi,\phi)$ be a simple orbifold datum in it, and consider the associated modular fusion category $\mcC_\mcA$ (see Section~\ref{subsubsec:SpecialOrbifoldData}).
We are now ready to define the Reshetikhin--Turaev orbifold graph TQFT\footnote{In the notation of \cite{CMRSS1}, $\zzca$ would be denoted as $(\zz^{\mathcal C})_\A^\Gamma$.}
\begin{equation}
\zzca \colon \Bordriben{3}(\mcC_\mcA) \lra \Vect \, .
\end{equation}
Orbifold graph TQFTs were introduced in \cite{CMRSS1} as graph TQFTs obtained by performing an internal state sum construction in a given $3$-dimensional defect TQFT.
Our construction is a generalisation of the state sum TQFT in \cite{TVireBook} which one recovers in the special case $\mcC=\Vect$.
In this section we specialise the general construction in \cite{CMRSS1} to the case of the Reshetikhin--Turaev defect TQFT $\zzc$ from Section \ref{subsec:RT_defect_TQFT}.
We refer to \cite{CMRSS1} for more details on orbifold graph TQFTs.

\subsubsection*{Orbifold datum as defect labels}
\begin{figure}
	\captionsetup{format=plain}
	\centering
	\pic[1.1]{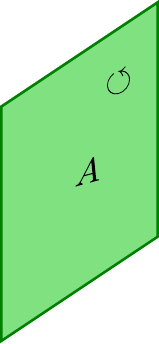} \hspace{-0.5cm}
	\pic[1.1]{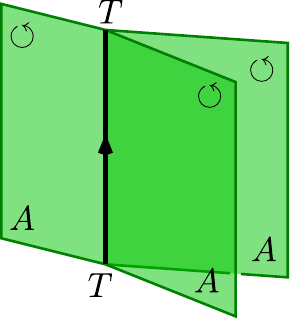} \hspace{-0.5cm}
	\pic[1.1]{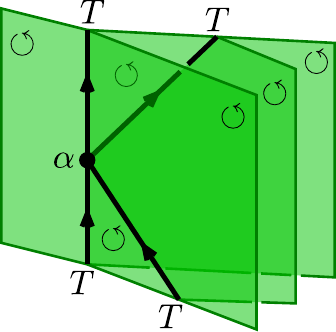} \hspace{-0.5cm}
	\pic[1.1]{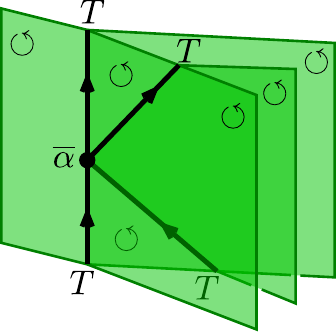}
	\caption{
	Orbifold datum $\mcA=(A,T,\a,\abar,\psi,\phi)$ as defect labels (for pictures involving the point defects~$\psi$ and~$\phi$, see Figure~\ref{fig:SpecialOrbifoldDataPhiPsi}).
	}
	\label{fig:orb_datum}
\end{figure}
The defect TQFT $\zzc$ allows one to interpret the orbifold datum $\mcA$ as a set of distinguished defect labels (see Figure~\ref{fig:orb_datum}):
    The label~$A$ for a $2$-stratum,
    $T$ for a $1$-stratum having three adjacent $2$-strata labelled by~$A$,
    $\a$ and $\abar$ for the ``intersection'' points of four $T$-labelled $1$-strata,
    $\psi$ for $0$-strata on $A$-labelled $2$-strata, and 
    $\phi$ for $0$-strata inside $3$-strata.
The identities in Figure~\ref{fig:SpecialOrbifoldDatum} then can be interpreted graphically as depicted in Figure~\ref{fig:SpecialOrbifoldDataPhiPsi}, where each identity corresponds to a local change on the appropriately labelled defect configuration that can be performed upon evaluation with $\zzc$ without changing the resulting invariant.

\medskip

\begin{figure}
	\captionsetup{format=plain}
	\centering
	\pic[1.1]{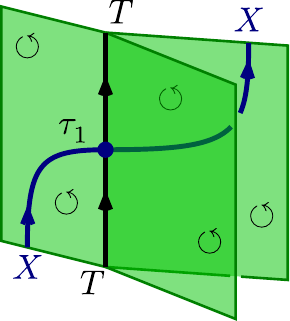} \hspace{-0.5cm}
	\pic[1.1]{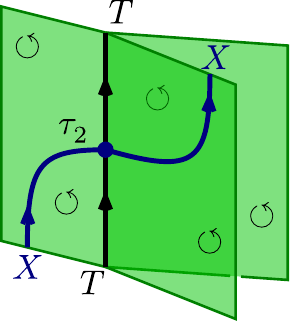} \hspace{-0.5cm}
	\pic[1.1]{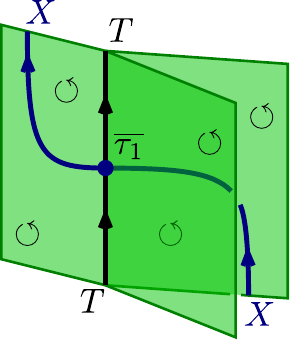} \hspace{-0.5cm}
	\pic[1.1]{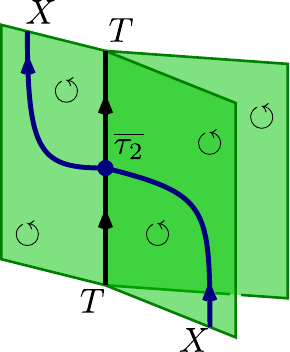}
	\caption{
	An object $(X,\tau_1,\tau_2,\taubar{1},\taubar{2})\in\mcC_\mcA$ as a set of defect labels.
	}
	\label{fig:CA_object}
\end{figure}
An object $(X,\tau_1,\tau_2,\taubar{1},\taubar{2})\in\mcC_\mcA$ can also be interpreted as a collection of defect labels (see Figure \ref{fig:CA_object}):
    The label~$X$ for a $1$-stratum lying within an $A$-labelled $2$-stratum and
    $\tau_i$, $\taubar{i}$ for the ``intersection'' points of $X$- and a $T$-labelled $1$-strata.
The algebraic identities in Figure~\ref{fig:CA_identities} then correspond to changes of defect configurations as in Figure~\ref{fig:CrossingIdentitiesPhiPsi}.

\subsubsection*{Ribbon diagrams and foamification}
The graph TQFT $\zzca$ is based on the ``foamification'' map
\begin{equation}
F \colon \Bordriben{3}(\mcC_\mcA) \lra \Borddefen{3}(\D^{\CC})\, ,
\end{equation}
which is an analogue of the ribbonisation map~\eqref{eq:ribbonisation} used to define the defect TQFT $\zzc$. In particular, $F$ will not be a functor.

\medskip

Let $(M,\mcR)$ be a $3$-manifold with an embedded $\mcC_\mcA$-coloured ribbon graph $\mcR$.
We denote the underlying unlabelled ribbon graph by $R$. The following notions are taken from \cite{CMRSS1} (and are in turn based on the analogous notions in \cite{TVireBook}).
\begin{itemize}[leftmargin=12pt]
\item
An \textsl{admissible $2$-skeleton} for $M$ is a stratification $\operatorname{T}$ dividing $M$ into a finite collection of open balls and (in case $M$ has boundary) half-balls, such that each point $x\in\operatorname{T^{(2)}}\setminus\pd M$ has a neighbourhood isomorphic to one of the stratified spaces depicted in Figure~\ref{fig:orb_datum} (at this point unlabelled).
Here $\operatorname{T}^{(d)}$, $d \in \{0,1,2,3\}$ denotes the union of strata of $\operatorname{T}$ with dimension~$\leqslant d$.
The qualifier ``admissible'' refers to the orientations of the strata, which are as in Figure~\ref{fig:orb_datum} with the points labelled by $\a$/$\abar$ having the orientation $+/-$.
Note that an admissible $2$-skeleton on $M$ induces an admissible $1$-skeleton on $\pd M$.
\item
An \textsl{admissible positive ribbon diagram} 
for $(M,R)$ is an admissible $2$-skeleton $\operatorname{T}$ for $M$ together with an embedding $\imath\colon R\longhookrightarrow \operatorname{T}^{(2)}\subseteq M$, which is isotopic to the identity embedding $R\longhookrightarrow M$ and such that all strands/coupons lie in the $2$-strata with framings induced by their orientations, except for finitely many intersection points between the strands of $\imath(R)$ and $1$-strata, which have neighbourhoods isomorphic to one of the stratified spaces in Figure~\ref{fig:CA_object} (at this point unlabelled).
The qualifier ``positive'' refers to the orientations imposed by Figure~\ref{fig:CA_object}.
\end{itemize}

Since the skeleta and ribbon diagrams we consider below are always admissible (and positive), we omit mentioning ``admissible'' and ``positive'' below.

\begin{itemize}[leftmargin=12pt]
\item
An \textsl{$\mcA$-coloured ribbon diagram} for $(M,\mcR)$ is a ribbon diagram $\mcT=(\operatorname{T},\imath)$ with the following labels assigned:
\begin{itemize}
\item 
$2$- and $1$-strata of $\operatorname{T}$ are labelled with $A$ and $T$, $0$-strata of $\operatorname{T}$ by $\a$ or $\abar$ as in Figure \ref{fig:orb_datum};
\item
the intersection points of the $1$-strata of $\operatorname{T}$ with strand of $\imath(\mcR)$ that is labelled with $(X,\tau_1,\tau_2,\taubar{1},\taubar{2})\in\mcC_\mcA$ are labelled with $\tau_i$ or $\taubar{i}$ as in Figure \ref{fig:CA_object};
\end{itemize}
and the following $0$-strata added:
\begin{itemize}
\item
for a $2$-stratum $F$ of $\operatorname{T}$, each component $F_i$ of $F \setminus \imath(\mcR)$ gets an additional $0$-stratum labelled with $\psi^{\chi_{\operatorname{sym}}(F_i)}$ where
\begin{equation}
\label{eq:symm_Euler_char}
\chi_{\operatorname{sym}}(-) = 2\chi(-) - \chi(\pd-)
\end{equation}
is the symmetric Euler characteristic;
the boundary segments of coupons are treated as boundary components of $F_i$;
\item
the 2-strata adjacent to the left and the right edges of each coupon get an additional $\psi$-insertion;\footnote{As explained in~\cite[App.\,B]{CMRSS1}, these additional $\psi$-insertions are needed so that the coupons of the embedded ribbon graphs can be composed when evaluating with the orbifold graph TQFTs.}
\item
each $3$-stratum $U$ of $\operatorname{T}$ gets an additional $0$-stratum labelled with $\phi^{\chi_{\operatorname{sym}}(U)}$.
\end{itemize}
\end{itemize}
An $\mcA$-coloured ribbon diagram $\mcT=(\operatorname{T},\imath)$ for $(M,\mcR)$ can be seen as a stratification for $M$ consisting of $\operatorname{T}$ together with additional $0$- and $1$-strata obtained from the embedding $\imath(\mcR)$ as well as the $\psi$- and $\phi$-insertions.
On the boundary $\pd M$ it restricts to an \textsl{$\mcA$-coloured $1$-skeleton of punctured surfaces}, i.\,e.\ a $1$-skeleton with additional $0$-strata at the punctures.

\medskip

For a $\mcC_\mcA$-coloured surface $\Sigma\in\Bordriben{3}(\mcC_\mcA)$, pick an $\mcA$-coloured 1-skeleton $\mcG$.
We define the foamification of $\Sigma$, 
\begin{equation}
F(\Sigma, \mcG) \in \Borddefen{3}(\D^{\CC}) \, , 
\end{equation}
to be a defect surface obtained from the underlying surface $\Sigma$ with the stratification and labels as in $\mcG$.

\medskip

For a $\mcC_\mcA$-coloured ribbon bordism $[(M,\mcR)\colon\Sigma\lra\Sigma']\in\Bordriben{3}(\mcC_\mcA)$, let $\mcT$ be an $\mcA$-coloured ribbon diagram.
The foamification 
\begin{equation}
\big[F(M,\mcT)\colon F(\Sigma,\mcG) \lra F(\Sigma',\mcG')\big] \in \Borddefen{3}(\D^{\CC})
\end{equation}
of $M$ is defined to be the defect bordism with the stratification given by $\mcT$ and restricting to the $1$-skeleta $\mcG, \mcG'$ of punctured surfaces on the boundary components $\Sigma, \Sigma'$.

\subsubsection*{\RT\ orbifold graph TQFT}
We define the orbifold graph TQFT $\zzca$ in terms of the Reshetikhin--Turaev defect TQFT $\zzc$ (see Section \ref{subsec:RT_defect_TQFT}) in a way similar to the definition of the latter in terms of the Reshetikhin--Turaev graph TQFT $\zrt$ (see Section \ref{subsec:RT_graph_TQFT}).
The construction is based on the following result (see \cite[Cor.\,4.13]{CMRSS1}).
\begin{theorem}
\label{thm:foamification_indep}
Let $([M,\mcR)\colon\Sigma\lra\Sigma']\in\Bordriben{3}(\mcC_\mcA)$ be a $\mcC$-coloured ribbon bordism and let $\mcS, \mcT$ be two $\mcA$-coloured ribbon diagrams which agree on the boundary components $\Sigma, \Sigma'$.
Then
\begin{equation}
\zzc\big( F(M,\mcS) \big) = \zzc\big( F(M,\mcT) \big) \, .
\end{equation}
\end{theorem}

From here on we proceed similarly as in Section \ref{subsec:RT_defect_TQFT}, following \cite[Sect.\,4.3.2]{CMRSS1}. 
For two $\mcA$-decorated skeleta $\mcG, \mcG'$ of a punctured surface $\Sigma\in\Bordriben{3}(\mcC_\mcA)$,
let $\mcT$ be an $\mcA$-decorated ribbon diagram of the cylinder bordism $C_\Sigma := \Sigma \times [0,1]$ which restricts to $\mcG$ and $\mcG'$ on the incoming and outgoing boundaries, respectively.
Define the linear map
\begin{equation}
\Psi_\mcG^{\mcG'} := \big[\zzc(F(C_\Sigma, \mcT)) \colon \zzc(F(\Sigma,\mcG)) \lra \zzc(F(\Sigma,\mcG'))\big] \,.
\end{equation}
By Theorem \ref{thm:foamification_indep} it is independent of the stratification in the interior and hence the following equality holds for all $\mcA$-decorated skeleta $\mcG, \mcG', \mcG''$ of the punctured surface $\Sigma$:
\begin{equation}
\Psi_{\mcG'}^{\mcG''}\circ\Psi_\mcG^{\mcG'} = \Psi_\mcG^{\mcG''} \,.
\end{equation}
\begin{construction}
\label{constr:zzCA}
Let $\mcA$ be a simple orbifold datum in a modular fusion category $\mcC$. The orbifold graph TQFT
\begin{equation}
\zzca \colon \Bordriben{3}(\mcC_\mcA) \lra \Vect
\end{equation}
is defined as follows:\footnote{%
    This construction works in the same way if $\mcA$ is not simple. The only change is that $\mcC_\mcA$ is multifusion. 
    Since here we are concerned with the case that $\mcC_\mcA$ is again a modular fusion category, we prefer not to switch between simple and not necessarily simple $\mcA$ in the presentation.
}
\begin{enumerate}
\item
For an object $\Sigma\in\Bordriben{3}(\mcC_\mcA)$ we set
\begin{equation}
\zzca(\Sigma) = \textrm{colim} \{\Psi_\mcG^{\mcG'}\} \,,
\end{equation}
where $\mcG, \mcG'$ range over $\mcA$-decorated skeleta of the punctured surface $\Sigma$.
\item
For a morphism $[(M,\mcR)\colon\Sigma\lra\Sigma']\in\Bordriben{3}(\mcC_\mcA)$ we set $\zzca(M,\mcR)$ to be
\begin{equation}
\begin{tikzcd}
	\zzca(\Sigma) \arrow[hook]{r} 
	& \zzc\big((F(\Sigma,\mcG)\big) \arrow{rr}{\zzc(F(M,\mcT))}
	& 
	& \zzc\big(F(\Sigma',\mcG')\big) \arrow[two heads]{r}
	& \zzca(\Sigma')  \, , 
\end{tikzcd}  
\end{equation}
where $\mcT$ is an arbitrary $\mcA$-decorated ribbon diagram of $(M,\mcR)$ restricting to $\mcG, \mcG'$ on the boundary components $\Sigma, \Sigma'$.
\end{enumerate}
\end{construction}
As in \eqref{eq:ZC_sos_as_im}, the state spaces of $\zzca$ are isomorphic to the images of the idempotents $\Psi_\mcG^\mcG$,
\begin{equation}
\label{eq:ZCA_sos_as_im}
\zzca(\Sigma) \cong \im \Psi_\mcG^\mcG \, .
\end{equation}

\begin{figure}
	\centering
	\begin{subfigure}[b]{0.32\textwidth}
		\centering
		\pic[1.25]{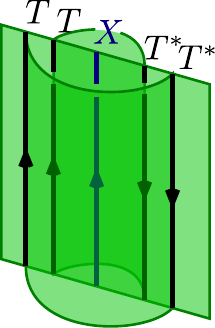}
		\caption{}
		\label{fig:pipe_strand}
	\end{subfigure}
	\begin{subfigure}[b]{0.32\textwidth}
		\centering
		\pic[1.25]{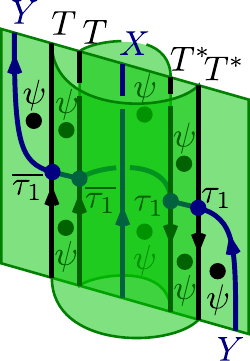}
		\caption{}
		\label{fig:pipe_crossing_1}
	\end{subfigure}
	\begin{subfigure}[b]{0.32\textwidth}
		\centering
		\pic[1.25]{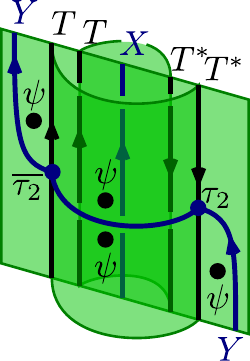}
		\caption{}
		\label{fig:pipe_crossing_2}
	\end{subfigure}
	\caption{Stratifications corresponding to 
	(a) the object $P(X)$, 
	(b) the morphism $c_{PX, Y}$, 
	(c) the morphism $c^{-1}_{Y, PX}$.
	}
	\label{fig:pipes}
\end{figure}

Let us collect some useful properties of $\zzca$.

\begin{property}
\label{ZCAProp:regular}
The orbifold graph TQFT $\zzca$ (like the graph TQFT $\zrtca$, see Section \ref{subsec:RT_graph_TQFT}) is a regular graph TQFT, i.\,e.\ the moves in \eqref{eq:regular-TQFT-def} become identities upon evaluation.\footnote{
This as well as Property~\ref{proper:pipe} actually hold in general for the graph TQFTs of \cite{CMRSS1}, but we only use them here for $\zzca$.
}
This follows from Property~\ref{ZCProp:compose_0-strata} 
of the defect TQFT $\zzc$.
Indeed, upon choosing an $\mcA$-coloured ribbon diagram and evaluating with $\zzc$, without loss of generality one can assume e.\,g.\ that the two coupons being composed lie in the same $A$-labelled $2$-stratum and interpret them as $0$-strata labelled by the corresponding $A$-$A$-bimodule morphisms.
There is an additional complication due to the $\psi$-insertions needed in the $\mcA$-colouring of a ribbon diagram, but the composition of the $0$-strata still remains the same as for bimodule morphisms.
In fact, the requirement to add a $\psi$-insertion to the left and to the right of each coupon was made to ensure that this holds, see \cite[App.\,B]{CMRSS1}.
\end{property}

\begin{property}
\label{proper:pipe}
Consider an $\mcA$-coloured ribbon diagram $\mathcal{R}$ obtained from the foamification procedure. If $\mathcal{R}$ contains a strand labelled by an object $P(X)$ obtained from applying the pipe functor \eqref{eq:pipe_functor} to a bimodule $X\in\mcACA$, we can exchange this strand for a pipe made of $T$-labelled $1$-strata with an $X$-labelled line in the middle (see Figure \ref{fig:pipe_strand}) without changing the value of the defect TQFT $\zzc$.
Similarly, coupons labelled  with braiding morphisms involving the object $P(X)$ can be exchanged for stratifications as depicted in Figures \ref{fig:pipe_crossing_1} and \ref{fig:pipe_crossing_2}.
\end{property}

\begin{figure}
	\captionsetup{format=plain}
	\centering
	\begin{subfigure}[b]{0.32\textwidth}
		\centering
		\pic[1.25]{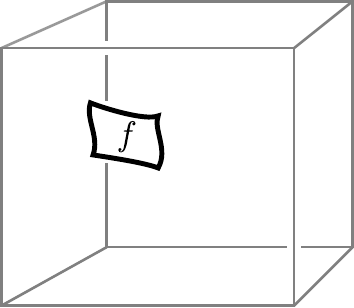}
		\caption{}
		\label{fig:S3c_original}
	\end{subfigure}
	\begin{subfigure}[b]{0.32\textwidth}
		\centering
		\pic[1.25]{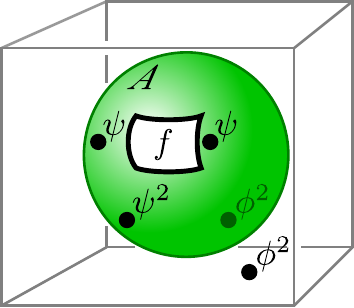}
		\caption{}
		\label{fig:S3c_foamified}
	\end{subfigure}
	\begin{subfigure}[b]{0.32\textwidth}
		\centering
		\pic[1.25]{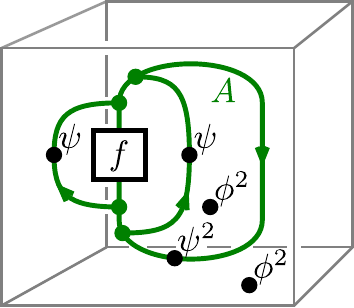}
		\caption{}
		\label{fig:S3c_ribbonised}
	\end{subfigure}
	\caption{
	(a) The $3$-sphere with an embedded $f\in\End_{\mcC_\mcA}$ labelled coupon $S^3_f$ depicted as a cube with boundary representing the point at infinity.
	(b) The $\mcA$-coloured ribbon diagram $\mcT$ for $S^3_f$.
	(c) A ribbonisation for the defect $3$-manifold $F(S^3_f,\mcT)$.
	}
	\label{fig:S3c_example}
\end{figure}

\begin{example}
\label{eg:S3c_example}
Let $f\in\End_{\mcC_\mcA}(A)$ and let $S^3_f = [(S^3_f,0)\colon\varnothing\lra\varnothing]$ be the morphism in $\Bordriben{3}(\mcC_\mcA)$ represented by the $3$-sphere with a single embedded $f$-labelled coupon. 
Since $A$ is the tensor unit in $\mcC_\mcA$, the coupon does not need to have adjacent strands, see Figure \ref{fig:S3c_original}.
We compute the invariant $\zzca(S^3_f)\in\opk$.
Let $\mcT$ be an $\mcA$-decorated ribbon diagram for $S^3_f$ consisting of an embedded $2$-sphere containing the coupon.
The resulting foamification $F(S^3_f,\mcT)$ is depicted in Figure \ref{fig:S3c_foamified}.
The defect sphere becomes a closed disc after removing the coupon; since the edges are treated as its boundary components, its symmetric Euler characteristic as given by~\eqref{eq:symm_Euler_char} is $2$, which explains the $\psi^2$-insertion.
The other two $\psi$-insertions are the ones that each coupon gets to the left and to the right of it.
$\mcT$ has two $3$-strata (the ``inside'' and the ``outside'' of the defect sphere), each of which gets a $\phi^2$-insertion.

\smallskip
\noindent
Next we consider the ribbonisation $R(\, F(S^3_f, \mcT) , \, t \,)$ where $t$ is the $1$-skeleton of the defect sphere with point insertions as shown in Figure \ref{fig:S3c_ribbonised}.
Using the invariant \eqref{eq:ZRTCS3} of $S^3$ one obtains:
\begin{align}
\nonumber
\zzca(S^3_f) &= 
\zzc\big( \, F(S^3_f,\mcT) \, \big) =
\zrt\big( \, R\big( \, F(S^3_f,\mcT), \, t \,\big) \, \big) \\
&=
\phi^4 \cdot \tr_\mcC(f\circ\psi^4) \cdot \zrt(S^3) =
\phi^4 \cdot \tr_\mcC(f\circ\psi^4) \cdot \mathscr{D}_\mcC^{-1} \,,
\end{align}
where in the third equality we collected the $\phi$-insertions into a single prefactor and simplified the ribbon graph into an $A$-labelled loop with an $f\circ\psi^4\colon A \lra A$ insertion. Since $\mcA$ is simple,
one can identify the scalar $f\in\End_{\mcC_\mcA}(A)\cong\opk$ 
with the trace $\tr_{\mcC_\mcA} f$.
Applying the expression \eqref{eq:DCA} then yields
\begin{equation}
\label{eq:S3c_computation}
\zzca(S^3_f)
= \phi^4 \cdot \tr_\mcC\psi^4 \cdot \mathscr{D}_\mcC^{-1} \cdot f
= \mathscr{D}_{\mcC_\mcA}^{-1} \cdot f \, .
\end{equation}
The resulting invariant is therefore precisely $\zrtca(S^3_f)$.
In fact, since any $\mcC_\mcA$-coloured ribbon graph $\mcR$ in $S^3$ can be simplified into a single coupon, this shows that more generally the equality
\begin{equation}
\label{eq:S3_invariants_equal}
\zzca(S^3,\mcR,0) = \zrtca(S^3,\mcR,0)
\end{equation}
holds.
In Section~\ref{subsec:InvariantsOf3Manifolds} this will play a role in showing that the two graph TQFTs are isomorphic.
Note that at this point we have not yet shown that the invariants of $(S^3,\mcR,n)$ for an arbitrary $n\in\opZ$ are equal as we have not shown that the anomalies of $\mcC$ and $\mcC_\mcA$ coincide, see~\eqref{eq:tau_Mn}.
This is done in Lemma~\ref{lem:C_CA_anomalies} below.
\end{example}

\begin{figure}
	\captionsetup{format=plain}
	\centering
	\begin{subfigure}[b]{0.38\textwidth}
		\centering
		\pic[1.25]{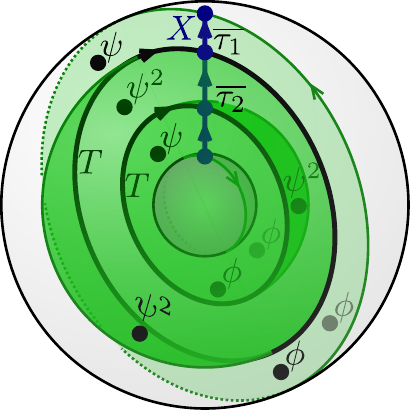}
		\caption{}
		\label{fig:S2wX_cylinder}
	\end{subfigure}
	\begin{subfigure}[b]{0.2\textwidth}
		\centering
		\pic[1.25]{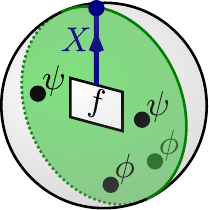}\\
		\vspace{1cm}
		\caption{}
		\label{fig:S2wX_ball}
	\end{subfigure}
	\begin{subfigure}[b]{0.38\textwidth}
		\centering
		\pic[1.25]{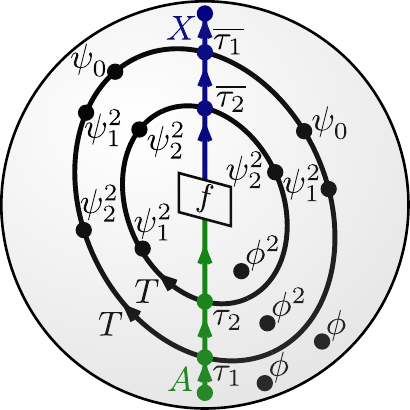}
		\caption{}
		\label{fig:S2wX_cylinder_rib_img}
	\end{subfigure}
	\caption{
	(a) An $\mcA$-coloured ribbon diagram $\mcT$ for the cylinder $C=S^2_\mcX \times [0,1]$ depicted as a closed ball with an open ball removed.
	On the boundary $\mcT$ restricts to $1$-skeleta $\mcG$ of the punctured surface $S^2_\mcX$ consisting of the $X$-labelled puncture with an adjacent $A$-labelled $1$-stratum.
	In the interior of $\mcT$ the two $T$-labelled $1$-strata have adjacent $A$-coloured $2$-strata forming ``bubbles'' between the two boundary components.
	(b) The bordism $B^{\operatorname{def}}_f\colon\varnothing \lra S^2_\mcG$.
	(c) Ribbonisation of $C_\mcT \circ B^{\operatorname{def}}_f\colon\varnothing\lra S^2_\mcG$.
	}
	\label{fig:S2wX_space}
\end{figure}
\begin{example}
Let $S^2_\mcX \in \Bordriben{3}(\mcC_\mcA)$ be the $2$-sphere with a single $\mcX\in\mcC_\mcA$ labelled puncture.
We compute the vector space $\zzca(S^2_\mcX)$.
Let $\mcT$ be the $\mcA$-coloured ribbon diagram of the cylinder $C=S^2_\mcX \times [0,1]$ restricting to a $1$-skeleton $\mcG$ on the boundary as depicted in Figure \ref{fig:S2wX_cylinder}, and let
\begin{equation}
    \big[C_\mcT\colon S^2_\mcG \lra S^2_\mcG\big] := F(C,\mcT) \in \Borddefen{3}(\D^\mcC)
\end{equation}
be the corresponding foamification.
By definition one has 
\begin{equation}
\zzca(S^2_\mcX) \cong \im\Psi_\mcG^\mcG \, , \quad \text{where } \Psi_\mcG^\mcG = \big[\zzc(C_\mcT)\colon \zzc(S^2_\mcG) \lra \zzc(S^2_\mcG)\big]  \, .
\end{equation}
The vector space $\zzc(S^2_\mcG)$ can be identified with $\mcACA(A,X)$ by sending an $A$-$A$-bimodule morphism $f\colon A \lra X$ to $\zzc(B^{\operatorname{def}}_f)$ where $B^{\operatorname{def}}_f\colon\varnothing\lra S^2_\mcG$ is the stratified closed ball bordism as in Figure \ref{fig:S2wX_ball}.\footnote{
 Since $\zzc(B^{\operatorname{def}}_f)$ is a linear map from $\Bbbk$ to $\zzc(S^2_\mcG)$, we should write $\zzc(B^{\operatorname{def}}_f)(1)$, but will usually keep the argument $1$ implicit.
}
One has the equality
\begin{equation}
    \Psi_\mcG^\mcG(f) = \overline{f} \, ,
\end{equation}
where $[\overline{f}\colon A \lra \X]\in\mcC_\mcA$ is the averaged morphism introduced in \eqref{eq:f_average}.
Indeed, $\Psi_\mcG^\mcG(f)$ corresponds to the evaluation $\zzc(C_\mcT \circ B^{\operatorname{def}}_f)$ which can be computed by ribbonising the argument and rearranging the result as in Figure \ref{fig:S2wX_cylinder_rib_img}
(for that we use the notation \eqref{eq:psi_omega_notation} and the crossing morphisms for $A$ \eqref{eq:CA_prod_crossings_unit}).
We conclude that for a simple orbifold datum $\mcA$ one has
\begin{equation}
\zzca(S^2_\mcX) 
\cong \mcC_\mcA(A,X) 
\cong \zrtca(S^2_\mcX) \, .
\end{equation}
Note that for a morphism $f\in\mcC_\mcA(A,X)$, the defect bordism $B^{\operatorname{def}}_f$ is a choice of foamification of the bordism $B_f\colon\varnothing\lra S^2_\mcX$ as introduced in~\eqref{eq:RTSPhereSpaces}.
\end{example}

\section{Proof of isomorphism}
\label{sec:proof}

Let $\mcA$ be a simple orbifold datum in a modular fusion category $\mcC$. 
From this we obtain the modular fusion category $\mcC_{\mcA}$ (cf.\ Section~\ref{subsubsec:SpecialOrbifoldData}) and the associated Reshetikhin--Turaev TQFT $\zrtca$ (Section~\ref{subsec:RT_graph_TQFT}). 
In this section we prove that this is the orbifold graph TQFT~$\zzca$ of Section~\ref{subsec:RT_orbifold_graph_TQFT}, hence showing that Reshetikhin--Turaev theories close under orbifolds: 
     
\begin{theorem}
\label{thm:tqfts_isomorphic}
Let $\mcA$ be a simple orbifold datum
in a modular fusion category $\mcC$. Then the 
orbifold graph TQFT~$\zzca$ for the choice of square root $\mathscr{D}_{\mcC}$
is equivalent (as monoidal functors) to the
graph TQFT~$\zrtca$ for the choice of square root $\mathscr{D}_{\mcC_\mcA}$ as in \eqref{eq:DCA}. Moreover, the monoidal natural isomorphism is unique.
\end{theorem}
The proof is similar to that of the isomorphism between the TQFTs of Turaev--Viro type obtained from a spherical fusion category $\mcS$ with $\dim\mcS \neq 0$ and the Reshetikhin--Turaev TQFTs obtained from the  Drinfeld centre $Z(\mcS)$, 
see \cite[Ch.\,17]{TVireBook}.
It relies on the following technical lemma, cf.\ \cite[Lem.\,17.2]{TVireBook}.
Since the uniqueness part of the statement is not part of the statement in
loc.\,cit.\ we repeat the main idea of the proof.
\begin{lemma}
	\label{lem:funct_iso}
	Let $\mathcal{M}$ be a rigid monoidal category and let $F,G\colon
        \mathcal{M} \lra \Vect$ be monoidal functors where we assume without
        loss of generality that $F(\one)=G(\one)=\Bbbk$.
	Moreover, suppose that     
	\begin{enumerate}[i)]
		\item
		\label{lem:funct_iso:cond1}
		$F$ is \textsl{non-degenerate},  i.\,e.\ for all objects $X\in\mathcal{M}$ one has
		\begin{equation}
        F(X) \cong \textrm{span}\big\{ F(f)(1) \,\big|\, f \in  \mathcal{M}(\one, X) \,\big\} \, ;
        \end{equation}
		\item
		\label{lem:funct_iso:cond2}
		for all $X \in \mathcal{M}$ one has
		\begin{equation}
		\dim_{\Bbbk} F(X) \geqslant \dim_{\Bbbk} G(X) \, ;
		\end{equation}
		\item
		\label{lem:funct_iso:cond3}
		for all $\varphi\in\End_\mathcal{M}(\one)$ one has
		\begin{equation}
		F(\varphi) = G(\varphi) \, .
		\end{equation}
	\end{enumerate}
	Then there is a unique natural monoidal isomorphism between $F$ and $G$.
\end{lemma}
\begin{proof}
 Suppose $\eta \colon F \lra G$ is a natural transformation. For all
  $X \in \mathcal{M}$ and all $f_X \in \mathcal{M}(\one,X)$, by naturality of $\eta$
the diagram
\begin{equation}
  \label{equation:commDiagUniqueNat}
    \begin{tikzcd}
      F(\one)=\Bbbk \ar{r}{=} \ar{d}[swap]{F(f_X)} & \Bbbk=G(\one) \ar{d}{G(f_X)} \\
	      F(X) \ar{r}[swap]{\eta_X}  &  G(X)
    \end{tikzcd}
  \end{equation}
 commutes. Since by assumption the vectors $F(f_X)$ span $F(X)$, this
  requirement determines $\eta$ uniquely. In the proof  of
  \cite[Lem.\,17.2]{TVireBook} it is shown that defining $\eta$ by the diagram
  \eqref{equation:commDiagUniqueNat} yields a well-defined monoidal natural
  transformation that has an inverse due to the assumption in part \ref{lem:funct_iso:cond2}.
\end{proof}

\begin{proof}[Proof of Theorem~\ref{thm:tqfts_isomorphic}]
We apply Lemma~\ref{lem:funct_iso} for $\mathcal{M} = \Bordriben{3}(\CC_\A)$, $F = \zrtca$ and $G = \zzca$. Thus to complete the proof of the theorem, we need to show that conditions \ref{lem:funct_iso:cond1}, \ref{lem:funct_iso:cond2}, \ref{lem:funct_iso:cond3} hold in this case. 

Condition \ref{lem:funct_iso:cond1} is fulfilled by definition of the Reshetikhin--Turaev TQFT, see~\eqref{eq:zrt-state-space-quotient}.
The proof that conditions \ref{lem:funct_iso:cond2} and \ref{lem:funct_iso:cond3} hold is more lengthy and has been moved to Sections~\ref{subsec:skeleta_from_surgery}--\ref{subsec:state_spaces} below. Specifically, condition \ref{lem:funct_iso:cond3} is shown in Lemma~\ref{lem:invariants_equal} and condition \ref{lem:funct_iso:cond2} in Lemma~\ref{lem:RT-orb_state_spaces}. 
\end{proof}

Note that for $\operatorname{char}\opk = 0$ condition \ref{lem:funct_iso:cond2} follows form condition \ref{lem:funct_iso:cond3}, 
since for all $\Sigma\in\Bordriben{3}(\CC_\A)$, the invariant of $\Sigma\times S^1$ already yield the dimension of the corresponding state space, see \eqref{eq:dim_ss_as_invariant}.
The more involved argument in Section~\ref{subsec:state_spaces} is needed to cover also the case $\operatorname{char}\opk \neq 0$.

\subsection{Skeleta from surgery}
\label{subsec:skeleta_from_surgery}
Let $M=[(M,\mathcal{R},n)\colon\varnothing\lra\varnothing]\in\Bordriben{3}(\mcC_\mcA)$ be a closed connected $\mathcal{C}_\A$-coloured ribbon $3$-manifold with an embedded ribbon graph $\mathcal{R}$.
For a choice $L$ of surgery link for $M$ we construct an $\A$-coloured ribbon diagram $\mcT_L$ for $(M,\mathcal{R})$:

\begin{enumerate}[i),wide, labelwidth=!, labelindent=0pt]
\item
\label{SL_step1}
Following Section \ref{subsec:RT_graph_TQFT},
let $L\sqcup\mcR$ be the ribbon graph in $S^3 \simeq \mathbb{R}^3 \cup \{\infty\}$ representing $(M,\mcR)$ i.\,e.\ consisting of possibly linked components of $L$ and $\mcR$ such that surgery along $L$ yields $(M,\mcR)$.
Project $L \sqcup \mathcal{R}$ on the outwards oriented unit $2$-sphere in $\R^3$ so that the coupons do not intersect the strands, the intersections of strands are pairwise distinct and transversal, and the framings of components of $L \sqcup \mathcal{R}$ agree with the orientation of the $2$-sphere.
The intersection points are marked to distinguish between ``overcrossings'' and ``undercrossings''.
For simplicity we assume that there are no crossings between the strands of $\mcR$, which can be achieved by exchanging them with coupons labelled with braiding morphisms. 
We identify~$L$ and~$\mathcal{R}$ with their projection images in~$S^2$ just described.
\begin{figure}
	\captionsetup{format=plain}
	\centering
	\begin{subfigure}[b]{0.6\textwidth}
		\centering
		\pic[1.25]{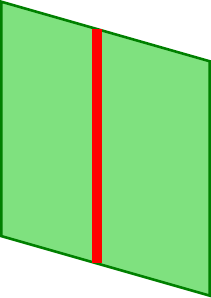}
		\hspace{-5pt}
		$\lra$
		\hspace{-5pt}
		\pic[1.25]{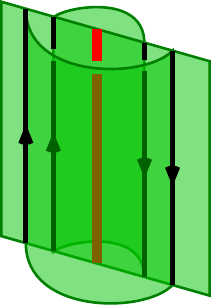}
		\caption{}
		\label{fig:LP-exchange}
	\end{subfigure}\\
	\begin{subfigure}[b]{0.48\textwidth}
		\centering
		\pic[1.25]{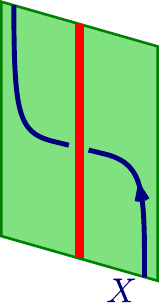} \hspace{-5pt}$\lra$\hspace{-5pt}
		\pic[1.25]{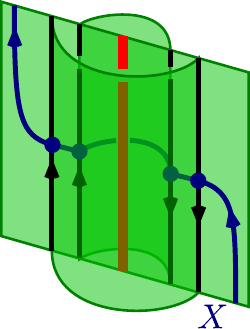}
		\caption{}
		\label{fig:PX_overcross-exchange}
	\end{subfigure}
	\begin{subfigure}[b]{0.48\textwidth}
		\centering
		\pic[1.25]{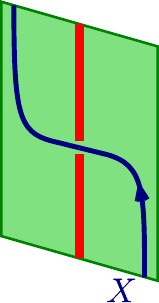} \hspace{-5pt}$\lra$\hspace{-5pt}
		\pic[1.25]{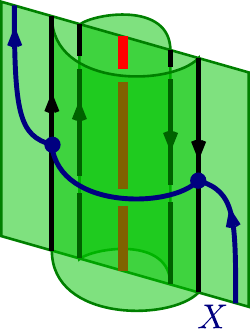}
		\caption{}
		\label{fig:PX_undercross-exchange}
	\end{subfigure}\\
	\begin{subfigure}[b]{0.8\textwidth}
		\centering
		\pic[1.25]{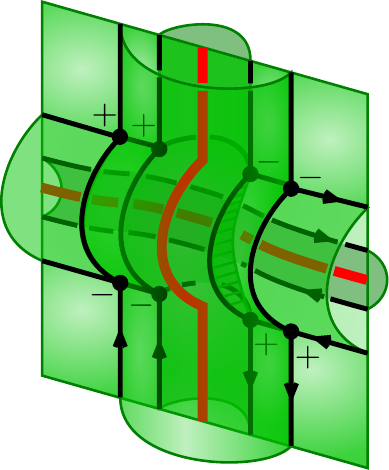}
		\caption{}
		\label{fig:PP_crossing-exchange}
	\end{subfigure}
	\caption{
		Exchanging surgery lines with tubes.
		Each $2$-stratum has the paper plane orientation, $1$-strata have orientations as indicated.
		In our application, strands that belong to the ribbon graph $\mathcal{R}$ will carry a label of some object
		     $\mathcal{X}= (X,\tau_1,\tau_2,\taubar{1},\taubar{2})\in\mathcal{C}_\A$, 
	    red strands will carry the Kirby object~\eqref{eq:Kirby_obj}, black strands will carry the label~$T$ from the orbifold datum~$\A$. Vertices will be labelled by  $\alpha$, $\bar\alpha$ from $\A$ or by
	    the crossing morphisms from~$\mathcal X$,
as appropriate.	    
	}
	\label{fig:L_tube}
\end{figure}
\begin{figure}
	\captionsetup{format=plain}
	\centering
	\begin{subfigure}[b]{0.8\textwidth}
		\centering
		\pic[1.25]{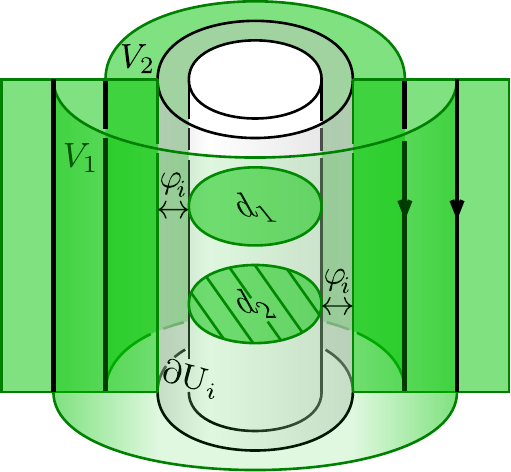}
		\caption{}
		\label{fig:gluing_1}
	\end{subfigure}\\
	\begin{subfigure}[b]{0.5\textwidth}
		\centering
		\pic[1.25]{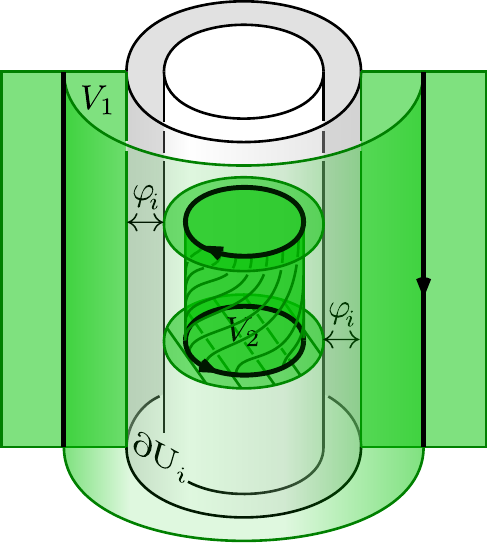}
		\caption{}
		\label{fig:gluing_2}
	\end{subfigure}
	\begin{subfigure}[b]{0.48\textwidth}
		\centering
		\pic[1.25]{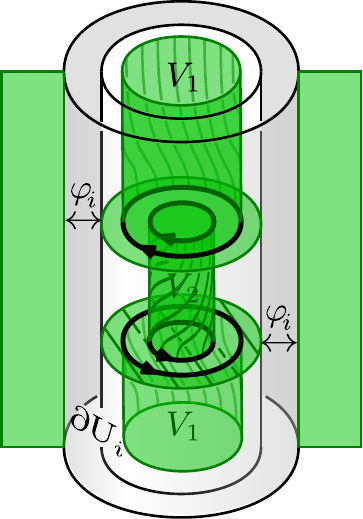}
		\caption{}
		\label{fig:gluing_3}
	\end{subfigure}
	\caption{
	For each component $L_i$, the 3-strata $V_1$, $V_2$ in the interior of the surrounding pipe become contractible after surgery.
	The pictures only show a strip of the pipe: it is actually closed and can be adjacent to other 2-strata or the strands of $R$ (which are ignored in this illustration).
	The inner white tube represents the solid torus with the two meridian disks $d_1$, $d_2$, that is being glued along $\varphi_i$ during surgery (for this tube the top and the bottom can be thought of as identified).
	Moreover, $\varphi_i$ maps the boundaries of the discs $d_1$, $d_2$ in
$\pd U_i$ to the meridian lines on the boundary of the solid torus.
	The pictures (a), (b) and (c) all depict the same configuration, in (b) and (c) the pipe stratification is moved inside the solid torus along $\varphi_i$.
	}
	\label{fig:wrapping_step}
\end{figure}
\item
\label{SL_step2}
Surround each strand of $L$ by additional $1$- and $2$-strata as shown in Figure \ref{fig:LP-exchange}.
As we will see later, this is done intentionally to mimic the stratification corresponding to a ``pipe'' object of $\mcC_\mcA$ as explained in Section \ref{subsec:RT_graph_TQFT}.
The crossings between the components of $L$ and $\mcR$ are exchanged for the stratifications as in Figures \ref{fig:PX_overcross-exchange} and \ref{fig:PX_undercross-exchange}, analogous to those corresponding to the braiding morphisms involving the pipe objects.
The crossings between the components of $L$ are handled in the same way, which results in a more intricate stratification shown in Figure \ref{fig:PP_crossing-exchange}.

At this point, the 3-strata in the stratification of $S^3$ are given by two open balls, and by two open tori for each component $L_i$ of $L$.

\item
Perform surgery along $L$ (see Section \ref{subsec:RT_graph_TQFT}):
for each component $L_i$ of $L$ let $U_i$ be a tubular neighbourhood of it, which lies inside the pipe stratification and let $\varphi_i\colon S^1 \times S^1 \lra -\pd(S^3 \setminus U_i)$ be the diffeomorphism along which the gluing with a solid torus $B^2\times S^1$ is performed, i.\,e.\ such that $\varphi_i(S^1\times\{1\})$ coincides with the framing of $L_i$.
We make this into a gluing of two stratified manifolds as follows:
Provide the solid tori with stratifications by two oriented meridian discs
\begin{equation}
\label{eq:meridian_discs}
d_1 = B^2\times\{\textrm{e}^{\textrm{i} \pi / 2}\} \quad\text{and}\quad
d_2 = (-B^2)\times\{\textrm{e}^{- \textrm{i} \pi / 2}\} \, .
\end{equation}
Since the $1$-strata on $\pd U_i$ run parallel to the curve left by the framing of $L_i$, we can assume that $\varphi_i$ is an isomorphism of stratified surfaces, i.\,e.\ $\varphi_i$ maps $S^1\times\{ \textrm{e}^{\pm \textrm{i} \pi / 2} \}$ to the $1$-strata of $\pd U_i$, see Figure \ref{fig:gluing_1}.
This yields the manifold $M$ together with a stratification $\operatorname{T}_L$, whose strata (besides those belonging to $\mcR$) are unlabelled.
\item
\label{SL_step4}
To show that $\operatorname{T}_L$ is a ribbon diagram for $(M,\mcR)$, one needs to argue that all its $3$-strata are contractible.
The two 3-strata of $S^3$ given by open balls do not intersect the tubular neighbourhoods $U_i$ and are not affected by the surgery. 
For the 3-strata that were open tori before surgery we proceed as follows:
\\
For a component $L_i$ of $L$, let $V_1, V_2$ be the $3$-strata inside the pipe stratification added in step \ref{SL_step2}.
Before the surgery, $V_1$ and $V_2$ are homeomorphic to open solid tori.
After the surgery their topology is changed: one can move the pipe stratification together with $V_1, V_2$ across the diffeomorphism $\varphi_i$ into the interior of the solid torus where they look like open solid cylinders and are evidently contractible (see Figures \ref{fig:gluing_2} and \ref{fig:gluing_3}).
\item
Define the $\mcA$-coloured ribbon diagram $\mcT_L$ by labelling the strata of $\operatorname{T}_L$ as in Section~\ref{subsec:RT_orbifold_graph_TQFT}: the $2$-strata with $A$, the $1$-strata either with $T$ or with the label of the corresponding strand of $\mcR$, and the $0$-strata either with $\a$/$\abar$, with the morphisms labelling the coupons of $\mcR$, or with the crossings morphisms between the $T$-lines and the strands of $\mcR$. Finally, add the appropriate $\psi$- and $\phi$-insertions.
\end{enumerate}
The number of $3$-strata in the ribbon diagram $\mcT_L$ is $2 + 2|L|$, where $|L|$ is the number of components of $L$.
Indeed, the $2$-sphere in step \ref{SL_step1} creates two $3$-strata and for each component of $L$ the pipe stratification creates two more $3$-strata.

\subsection{Invariants of closed 3-manifolds}
\label{subsec:InvariantsOf3Manifolds}

Let $M=[(M,\mathcal{R},n)\colon\varnothing\lra\varnothing]\in\Bordriben{3}(\mcC_\mcA)$ be a connected $\mcC_\mcA$-coloured ribbon $3$-manifold and let $L \subset S^3$ be a choice of surgery link for $M$. 
Let $\mcL_P \sqcup \mcR$ be the $\mcC_\mcA$-coloured ribbon graph in $S^3$ obtained from $L\sqcup\mcR$ by colouring each component of $L$ with $P(A \otimes K_\mcC \otimes A)$ and adding a $P(\psi^2 \otimes \xi_\mcC \otimes \psi^2)$-insertion, as well as exchanging the crossings with braiding morphisms.
\begin{lemma}
\label{lem:ZCA_formula}
The following formula holds:
\begin{equation}
\label{eq:ZCA_formula}
\zzca(M,\mcR,n) = 
\d_\mcC^{n-\sigma(L)} \cdot 
\mathscr{D}_\mcC^{-|L|} \cdot
\phi^{4|L|} \cdot
\mathscr{D}_{\mcC_\mcA}^{-1} \cdot
F_{\mcC_\mcA}(\mcL_P \sqcup \mcR) \, ,
\end{equation}
where since $\mcA$ is a simple orbifold datum, one has $F_{\mcC_\mcA}(\mcL_P \sqcup \mcR) \in \End_{\mcC_\mcA}(A)$ $\cong \opk$.
\end{lemma}

Besides the general proof below, we illustrate some steps in the proof of Lemma~\ref{lem:ZCA_formula} in Figure~\ref{fig:S3_W_computation} for the explicit example of the invariant of the ribbon $3$-manifold $(S^3,\varnothing,0)$, for which we use the surgery link $W$ having one component with a single twist, i.\,e.\
\begin{equation}
\label{eq:W_link}
W = \pic[1.25]{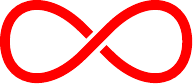} \, .
\end{equation}
This example will later help us compute the anomaly of the category $\mcC_\mcA$.

\begin{figure}
	\captionsetup{format=plain}
	\centering
	\begin{subfigure}[b]{0.49\textwidth}
	    \captionsetup{labelformat=empty}
		\centering
		\pic[1.25]{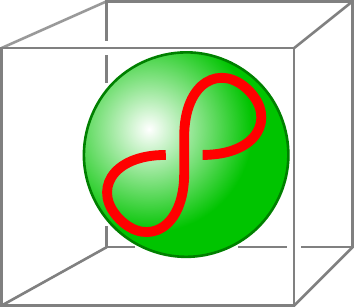}
		\caption{(a)}
		\label{fig:W_link_projected}
	\end{subfigure}
	\begin{subfigure}[b]{0.49\textwidth}
	    \captionsetup{labelformat=empty}
		\centering
		\pic[1.25]{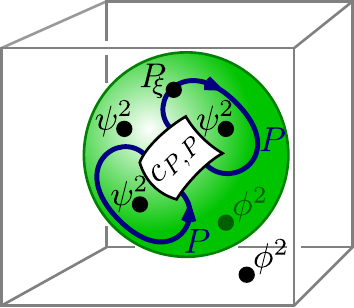}
		\caption{(c)}
		\label{fig:W_link_relabelled}
	\end{subfigure}\\
	\begin{subfigure}[b]{0.99\textwidth}
	    \captionsetup{labelformat=empty}
		\hspace{-30pt}
		\pic[1.25]{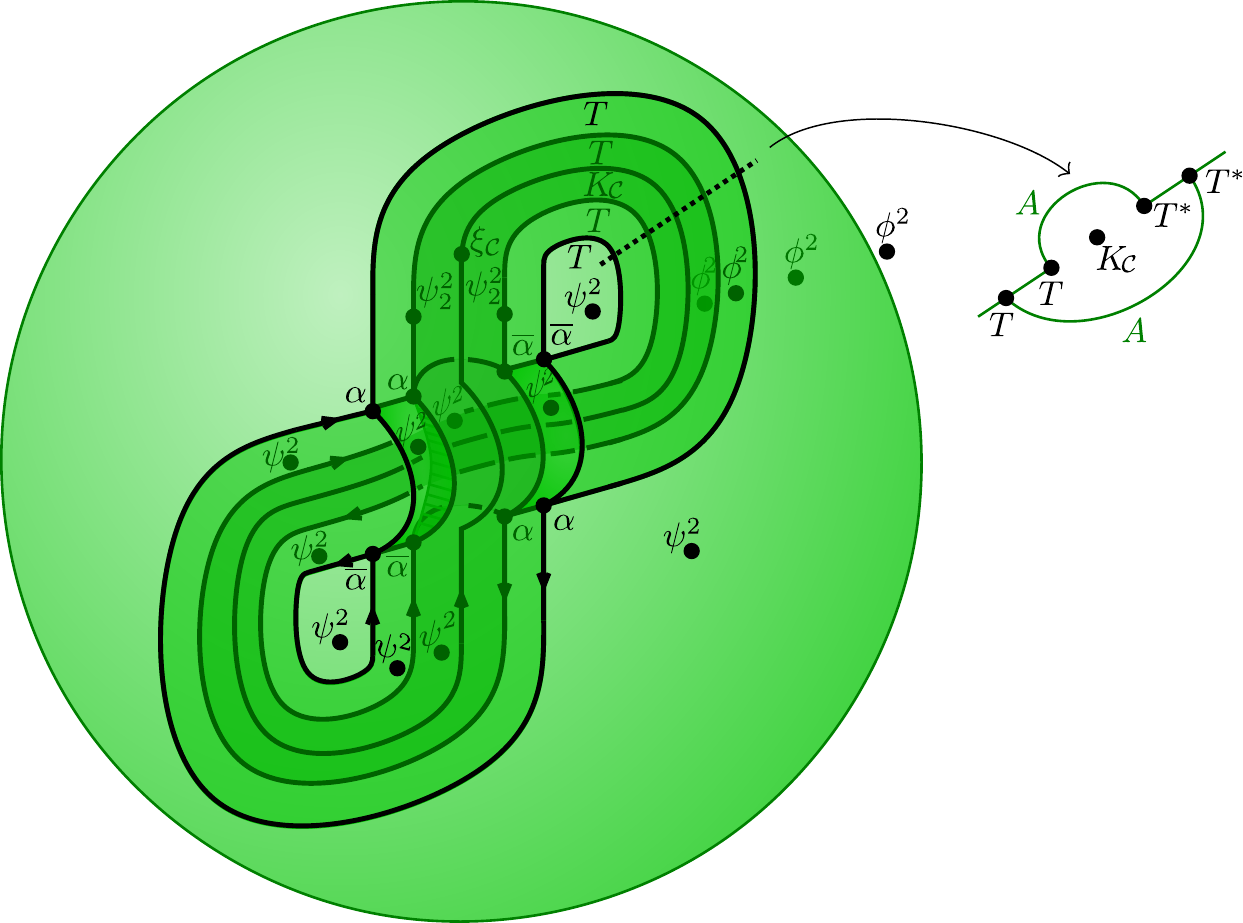}
		\caption{(b)}
		\label{fig:W_link_pipes}
	\end{subfigure}
	\caption{
	Illustration of the proof of Lemma~\ref{lem:ZCA_formula} for the case of surgery link $W$ in~\eqref{eq:W_link}:
	(a) Projecting $W$ on a 2-sphere in $S^3$;
	(b) Surrounding the strand of $W$ with pipe stratifications and adding labels.
	    For clarity the cross-section of the pipe ``from above'' is provided.
	    Note that the inner $T$-labelled lines have only two adjacent 2-strata as one of the was removed during step (i) of the proof.
	(c) Exchanging the pipes with $P := P(A \otimes K_\mcC \otimes A)$-labelled strands with $P_\xi := P(\psi^2 \otimes \xi_\mcC \otimes \psi^2)$-insertions and replacing the crossing with a coupon.
	}
	\label{fig:S3_W_computation}
\end{figure}

\begin{proof}[Proof of Lemma \ref{lem:ZCA_formula}]
    In terms of the $\mcA$-decorated ribbon diagram $\mcT_L$ from Section \ref{subsec:skeleta_from_surgery}, by definition of the invariants of closed manifolds one has
\begin{equation}
\zzca(M,\mcR,n) = \zzc \big( \, F(\, (M,\mcR,n),\, \mcT_L \,) \, \big) \, , 
\end{equation}
where~$F$ is the foamification discussed in Section~\ref{subsec:RT_orbifold_graph_TQFT}. 
We apply the properties of the defect TQFT $\zzc$ as reviewed in Section \ref{subsec:RT_defect_TQFT} to modify the stratification $\mcT_L$ of $M$ in several steps:
\begin{enumerate}[i),wide, labelwidth=!, labelindent=0pt]
\item
\label{ZAorb_eval_step_1}
For each component $L_i$ of $L$, the meridian discs of the solid tori glued in during surgery (i.\,e.\ $d_1, d_2$ as in \eqref{eq:meridian_discs}) are part of
contractible $2$-strata of $F(M,\mcT_L)$ each of which has a $\psi^2$-insertion as given by \eqref{eq:symm_Euler_char}.
From Section~\ref{subsec:RT_defect_TQFT} we know that they can be removed with the $\psi^2$-insertions moved aside to give $\psi^2_2$-insertions on the adjacent $T$-labelled $1$-strata (recall Property~\ref{prop:GetRidOfDiscs} and \eqref{eq:psi_omega_notation}).
The rest of $\mcT_L$ can then be moved so as to not intersect the boundaries of the solid tori along which the surgery was performed.
\item
\label{ZAorb_eval_step_2}
We can now undo the surgery used to obtain $M$, that is, we replace $M$ with $S^3$ and embedded surgery link $L$ with $\mcT_L$ in the complement of $L$.
By definition of $\zzc$ via the Reshetikhin--Turaev TQFT $\zrt$ and the formula \eqref{eq:M_inv_as_S3_inv} for the invariants of closed manifolds, one can further proceed by replacing the link $L$ with its coloured version $\mcL_\mcC$ (i.\,e.\ with components labelled by $K_\mcC$ and having $\xi_\mcC$-insertions, see \eqref{eq:Kirby_obj} and \eqref{eq:Kirby_col}) and multiplying with the overall factor $\d_\mcC^{n-\sigma(L)} \cdot \mathscr{D}_\mcC^{-|L|}$.
If for the example surgery link $W$ in~\eqref{eq:W_link} we choose the projection to~$S^2$ depicted in Figure~\hyperref[fig:W_link_pipes]{\ref{fig:S3_W_computation}a}, the resulting stratification of $S^3$ is as shown in Figure~\hyperref[fig:W_link_pipes]{\ref{fig:S3_W_computation}b}.
\item
\label{ZAorb_eval_step_3}
Following Section~\ref{subsec:RT_orbifold_graph_TQFT}, the components of the labelled surgery link $\mcL_\mcC$, along with the pipe stratifications surrounding them, can be exchanged for the $1$-strata labelled by the objects $P(A\otimes K_\mcC \otimes A)\in\mcC_\mcA$, i.\,e.\ by the pipe objects obtained from the induced bimodule $A \otimes K_\mcC \otimes A \in \mcACA$.
As an intermediate step in this change one uses the isomorphisms $T \otimes_2 A \cong T$ of $A$-$A\otimes A$-bimodules in order to apply the definition~\eqref{eq:pipe_functor} with $X = A \otimes K_\mcC \otimes A$.
surrounding stratification obtained in step \ref{ZAorb_eval_step_1} can be collected into a single $P(\psi^2 \otimes \xi_\mcC \otimes \psi^2)$-insertion.
The stratifications that replaced the crossings of strands in the construction of $\mcT_L$ (i.\,e.\ the ones depicted in Figures~\ref{fig:PX_overcross-exchange}, \ref{fig:PX_undercross-exchange} and~\ref{fig:PP_crossing-exchange} with the appropriate $\psi$-insertions) can in turn be replaced with coupons labelled by the braiding morphisms.

\item
The previous step yields an overall factor $\phi^{4|L|}$ (which previously served as $\phi^2$-insertions for the $3$-strata inside the pipe stratifications surrounding the components of $L$) and simplifies the stratification of $S^3$ to an $A$-labelled $2$-sphere with the graph $\mcL_P \sqcup \mcR$ projected on it as well as the two remaining $\phi^2$-insertions in its $3$-strata (see Figure~\hyperref[fig:W_link_relabelled]{\ref{fig:S3_W_computation}c} 
for how it looks for the example surgery link $W$).
The ribbon graph $\mcL_P \sqcup \mcR$ can be further exchanged for a single coupon labelled with $f = [F_{\mcC_\mcA}(\mcL_P \sqcup \mcR)\colon A \lra A]$.
At this point one recognises the remaining stratification of $S^3$ as the ribbon diagram used in Example~\ref{eg:S3c_example}.
Using the computation \eqref{eq:S3c_computation} and collecting all prefactors, one obtains the formula~\eqref{eq:ZCA_formula}.
\end{enumerate}
\end{proof}

\medskip

We proceed to compare the expression \eqref{eq:ZCA_formula} with the invariant $\zrtca(M,\mathcal{R},n)$, which is obtained by adapting \eqref{eq:tau_Mn} for the modular fusion category $\mcC_\mcA$.
We start by addressing the different colourings of the surgery link.

\medskip

Let $\mcR$ be any morphism in the category $\textrm{Rib}_{\mathcal{C}_\A}$ of $\mcC_\mcA$-coloured ribbon graphs (for the sake of generality, we do not assume that $\mcR$ is closed, i.\,e.\ it can have free incoming/outgoing strands), and let $L$ be an uncoloured oriented framed link.
Denote as before by $L\sqcup \mcR$ a ribbon graph consisting of possibly entangled components of $L$ and $\mcR$ and by $\mcL_{\mcC_\mcA} \sqcup \mcR$ and $\mcL_P \sqcup \mcR$ its versions with $L$ carrying the respective colouring by objects $K_{\mcC_\mcA}$ and $P(A \otimes K_\mcC \otimes A)$ with additional point insertions labelled by the endomorphisms $\xi_{\mcC_\mcA}$ and $P(\psi^2 \otimes \xi_\mcC \otimes \psi^2)$.

\begin{lemma}
\label{lem:red_ribb_in_CA_id}
One has the following identity of morphisms in $\mathcal{C}_\A$:
\begin{equation}
\label{eq:LK_vs_LP_labelling}
F_{\mcC_\mcA}(\mcL_{\mcC_\mcA} \sqcup \mcR)
= \frac{1}{(\tr_\mcC\psi^4)^{|L|}} \cdot F_{\mcC_\mcA}(\mcL_P \sqcup \mcR) \, .
\end{equation}
\end{lemma}
\begin{proof}
We will need the following decompositions for simple
objects $i\in\mathcal{C}$, $\mu\in{}_A\mathcal{C}_A$, $\Delta \in \mathcal{C}_\A$ (the label above the isomorphism symbol indicates the category in which it holds):
\begin{equation}
\begin{array}{llll}
\text{a)} & \displaystyle \mu \stackrel{\mathcal{C}}{\cong} \bigoplus\limits_{k\in\textrm{Irr}_{\mathcal{C}}} k \otimes \mathcal{C}(k,\mu) \, , &
\text{b)} & \displaystyle \Delta  \stackrel{{}_A\mathcal{C}_A}{\cong} \bigoplus\limits_{\nu\in\textrm{Irr}_{\!{}_A\mathcal{C}_A}} \nu \otimes {}_A\mathcal{C}_A(\nu,\Delta) \, ,\\ \label{eq:adj_decompositions}
\text{c)} & \displaystyle A\otimes i\otimes A \stackrel{{}_A\mathcal{C}_A}{\cong} \bigoplus\limits_{\nu\in\textrm{Irr}_{\!{}_A\mathcal{C}_A}} \nu \otimes \mathcal{C}(i,\nu) \, , &
\text{d)} & \displaystyle P(\mu) \stackrel{\mathcal{C}_\A}{\cong}\bigoplus\limits_{\Lambda\in\textrm{Irr}_{\mathcal{C}_\A}} \Lambda \otimes _A\mathcal{C}_A(\mu,\Lambda) \, .
\end{array}
\end{equation}
For the isomorphisms c) and d) we used that
the pipe functor $\mcACA\lra\mcC_\mcA$ and the induction functor $\mcC\lra\mcACA$ are biadjoint to the respective forgetful functors.

For a simple object $\nu\in{}_A\mathcal{C}_A$ and a morphism $f\in\End_{{}_A\mathcal{C}_A}(\nu)$, let us denote by $\langle f \rangle \in \Bbbk$ the number such that $f = \langle  f \rangle \cdot \id_\nu$.

\medskip

\noindent
It is enough to show \eqref{eq:LK_vs_LP_labelling} for the case when $L$ has one component, i.\,e. $|L|=1$, as one can then iterate the argument.
For an object $\mcX\in\mcC_\mcA$ and a morphism $f\in\End_{\mcC_\mcA}\mcX$, let $\mcL(\mcX,f)$ be the colouring of $L$ by $\mcX$ with an $f$-insertion so that
\begin{equation}
\mcL_{\mcC_\mcA} = \mcL\big(\, K_{\mcC_\mcA}, \, \xi_{\mcC_\mcA} \,\big) \, , \quad
\mcL_P = \mcL\big( \,P(A\otimes K_{\mcC} \otimes A ), \, P(\psi^2 \otimes \xi_\mcC \otimes \psi^2) \,\big) \,.
\end{equation}
In addition we abbreviate $\mcQ(\mcX,f):=\mcL(\mcX,f) \sqcup \mcR$ and $\mcQ(\mcX) := \mcQ(\mcX,\id_\mcX)$.
One has:
\begin{align}
\nonumber
F_{\mcC_\mcA}(\mcL_{\mcC_\mcA} \sqcup \mcR)
 & \stackrel{\eqref{eq:Kirby_obj},\eqref{eq:Kirby_col}}{=}\sum_{\Lambda \in \textrm{Irr}_{\mathcal{C}_\A}} \dim_{\mcC_\mcA}\Lambda \cdot F_{\mcC_\mcA}\big(\, \mcQ(\Lambda) \,\big) 
 \nonumber \\ 
 \nonumber 
 & \stackrel{\eqref{eq:trace_in_CA}}{=}\sum_{\Lambda \in \textrm{Irr}_{\mathcal{C}_\A}} \frac{\tr_\mathcal{C} \omega^2_\Lambda}{\tr_\mathcal{C} \psi^4} \cdot F_{\mcC_\mcA}\big( \, \mcQ(\Lambda) \, \big)\\ \nonumber
&\overset{\text{(\ref{eq:adj_decompositions}\,b)}}= \sum_{\substack{\Lambda \in \textrm{Irr}_{\mathcal{C}_\A}\\ \nu\in\textrm{Irr}_{\!\mcACA}}} \frac{\tr_\mathcal{C} \omega^2_\nu \cdot \dim {}_A\mathcal{C}_A(\nu,\Lambda)}{\tr_\mathcal{C} \psi^4} \cdot F_{\mcC_\mcA}\big(\, \mcQ(\Lambda) \, \big)
 \\ \nonumber 
 &\overset{\text{(\ref{eq:adj_decompositions}\,d)}}=\sum_{\nu\in\textrm{Irr}_{\!\mcACA}} \frac{\tr_\mathcal{C} \omega^2_\nu}{\tr_\mathcal{C} \psi^4} \cdot
F_{\mcC_\mcA}\big(\, \mcQ(P(\nu)) \,\big)\\ \nonumber
&=\sum_{\nu\in\textrm{Irr}_{\!\mcACA}} \frac{\dim_\mcC\nu \cdot \langle\omega_\nu^2\rangle}{\tr_\mathcal{C} \psi^4} \cdot
F_{\mcC_\mcA}\big(\,\mcQ(P(\nu)) \,\big)
\\ \nonumber 
 & =\sum_{\nu\in\textrm{Irr}_{\!\mcACA}} \frac{\dim_\mcC\nu}{\tr_\mathcal{C} \psi^4} \cdot
F_{\mcC_\mcA}\big(\, \mcQ(P(\nu),\, P(\omega_\nu^2)) \,\big) \\ \nonumber
&\overset{\text{(\ref{eq:adj_decompositions}\,a)}}=\sum_{\substack{\nu\in\textrm{Irr}_{\!\mcACA}\\ k\in\textrm{Irr}_{\mcC}}} \frac{\dim_\mcC k \cdot \dim \mcC(k,\nu)}{\tr_\mathcal{C} \psi^4} \cdot
F_{\mcC_\mcA}\big(\, \mcQ(P(\nu), \, P(\omega_\nu^2)) \,\big) \\ \nonumber
&\overset{\text{(\ref{eq:adj_decompositions}\,c)}}=\sum_{k\in\textrm{Irr}_{\mcC}} \frac{\dim_\mcC k}{\tr_\mathcal{C} \psi^4} \cdot
F_{\mcC_\mcA}\big(\, \mcQ(P(A \otimes k \otimes A), \, P(\omega_{A\otimes k \otimes A}^2)) \,\big)\\
&\stackrel{\eqref{eq:Kirby_obj},\eqref{eq:Kirby_col}}{=}\frac{1}{\tr_\mathcal{C} \psi^4} \cdot F_{\mcC_\mcA}(\mcL_P \sqcup \mcR) \, .
\end{align}
\end{proof}

\medskip

The identity in Lemma~\ref{lem:red_ribb_in_CA_id} provides the penultimate step in comparing the invariants of closed ribbon $3$-manifolds.
We apply it first to compare anomalies (defined in~\eqref{eq:anomaly}): 
\begin{lemma}
\label{lem:C_CA_anomalies}
For the choice of square roots of global dimensions related as in \eqref{eq:DCA},
the modular fusion categories $\mathcal{C}$ and $\mathcal{C}_\A$ have the same anomaly: $\delta_\mathcal{C} = \delta_{\mathcal{C}_\A}$.
\end{lemma}
\begin{proof}
We compare two computations of the invariant $\zzca(S^3,\varnothing,0)$, one using the empty surgery link as in the Example~\ref{eg:S3c_example} and the other using the surgery link $W$ in  \eqref{eq:W_link}.
The former one is obtained from \eqref{eq:S3_invariants_equal} and \eqref{eq:ZRTCS3} and gives $\zzca(S^3,\varnothing,0) = \mathscr{D}_{\mcC_\mcA}^{-1}$.
For the latter, note that one has $\sigma(W)=1$ and it follows from~\eqref{eq:pC} 
that $F_{\mcC_\mcA}(\mcW_{\mcC_\mcA}) = p^+_{\mcC_\mcA}$.
One uses Lemmas~\ref{lem:ZCA_formula} and~\ref{lem:red_ribb_in_CA_id} to compute:
\begin{align}
\nonumber
\zzca(S^3,\varnothing,0)
\, &= \,
\delta_\mcC^{-1} \cdot
\mathscr{D}_\mcC^{-1} \cdot 
\phi^4 \cdot
\mathscr{D}_{\mcC_\mcA}^{-1} \cdot
F_{\mcC_\mcA}(\mcW_P)
=
\delta_\mcC^{-1} \cdot
\left( \frac{\mathscr{D}_\mcC}{\phi^4 \cdot \tr_\mcC \psi^4} \right)^{-1} 
\hspace{-.8em}
\cdot
\mathscr{D}_{\mcC_\mcA}^{-1} \cdot
p^+_{\mcC_\mcA}
\\
&
\hspace{-1.4em} \stackrel{\text{\eqref{eq:DCA},\eqref{eq:anomaly}}}{=} \,
\delta_\mcC^{-1} \cdot
\mathscr{D}_{\mcC_\mcA}^{-1} \cdot
\delta_{\mcC_\mcA} \, .
\end{align}
Since by well-definedness of $\zzca$ the two computations have to agree, the statement follows. 
\end{proof}

The next lemma finally establishes condition~\ref{lem:funct_iso:cond3} in Lemma~\ref{lem:funct_iso}.

\begin{lemma}
\label{lem:invariants_equal}
Let $(M,\mathcal{R},n)$ be a closed $\mcC_\mcA$-coloured ribbon $3$-manifold viewed as a morphism $\varnothing\lra\varnothing$ in $\Bordriben{3}(\mcC_\mcA)$.
One has:
\begin{equation}
\zrtca(M,\mathcal{R},n) = \zzca(M,\mathcal{R},n) \, .
\end{equation}
\end{lemma}

\begin{proof}
Since both invariants are multiplicative with respect to connected components, it is enough to consider the case when $M$ is connected.
One has:
\begin{align}
\nonumber
\zrtca(M,\mathcal{R},n)
&~\,\stackrel{\eqref{eq:tau_Mn}}{=}&&
\delta_{\mcC_\mcA}^{n-\sigma(L)} \cdot 
\mathscr{D}_{\mcC_\mcA}^{-|L|-1} \cdot 
F_{\mcC_\mcA}(\mcL_{\mcC_\mcA} \sqcup \mcR)\\ \nonumber
&\stackrel[\eqref{eq:LK_vs_LP_labelling}]{\text{Lem.\ref{lem:C_CA_anomalies}}}{=}&&
\delta_{\mcC}^{n-\sigma(L)} \cdot
\mathscr{D}_{\mcC_\mcA}^{-|L|-1} \cdot
\frac{1}{(\tr_\mcC \psi^4)^{|L|}} \cdot
F_{\mcC_\mcA}(\mcL_P \sqcup \mcR)\\ \nonumber
&~\,\stackrel{\eqref{eq:ZCA_formula}}{=}&&
\mathscr{D}_{\mcC_\mcA}^{-|L|} \cdot
\frac{1}{(\tr_\mcC \psi^4)^{|L|}} \cdot
\mathscr{D}_{\mcC}^{|L|} \cdot
\phi^{-4|L|} \cdot
\zzca(M,\mathcal{R},n) \\ \nonumber
&~~\,=&&
\mathscr{D}_{\mcC_\mcA}^{-|L|} \cdot
\left( \frac{\mathscr{D}_\mcC}{\phi^4 \cdot \tr_\mcC\psi^4} \right)^{|L|} \cdot
\zzca(M,\mathcal{R},n) \\
&~\,\stackrel{\eqref{eq:DCA}}{=}&&
\zzca(M,\mathcal{R},n) \, .
\end{align}
\end{proof}

\subsection{State spaces}
\label{subsec:state_spaces}

In this section we will show condition~\ref{lem:funct_iso:cond2} in Lemma~\ref{lem:funct_iso}, that is the inequality $\dim \zrtca(\Sigma) \geqslant \dim \zzca(\Sigma)$ for a punctured surface $\Sigma\in\Bordriben{3}(\mcC_\mcA)$. To do this, we first give an explicit description of the TQFT state spaces on both sides of the inequality.

It is enough to consider the case of $\Sigma$ having a single $\mcX=(X,\tau_1,\tau_2,\taubar{1},\taubar{2})\in\mcC_\mcA$ labelled puncture (this follows from \cite[Lem.\,15.1]{TVireBook} and the regularity of both TQFTs, see Properties~\ref{ZRTProp:regular} and~\ref{ZCAProp:regular}).
In Example~\ref{eg:S3c_example} we already looked at the case when $\Sigma$ has genus 0 so here we will assume $\Sigma$ to have genus $g > 0$.

\medskip

The vector spaces assigned to punctured surfaces by the Reshetikhin--Turaev TQFT were already mentioned in Property~\ref{zrt:coupons_in_spheres}.
For the genus-$g$ surface $\Sigma$ with an $\mcX$-labelled puncture 
    one has, with $\mathbb{L} = \bigoplus_{\Delta \in\textrm{Irr}_{\mcC_\mcA}} \Delta \otimes \Delta^*$,
\begin{align}\nonumber
&\zrtca(\Sigma) \cong
\mcC_\mcA(\one,\mcX \otimes \mathbb{L}^{\otimes g})
\\
&
\cong
\bigoplus
\mcC_\mcA(\Delta_1,\mcX\Delta_1') \otimes
\mcC_\mcA(\Delta_1',\Delta_1\Gamma_1) \otimes
\mcC_\mcA(\Gamma_1\Delta_2,\Delta_2\Gamma_2) \otimes \cdots \otimes
\mcC_\mcA(\Gamma_{g-1}\Delta_g,\Delta_g) \, ,
\label{eq:RT-state-space-direct-sum}
\end{align}
where the direct sum is over $\Delta_1,\Delta_1',\Delta_2,\dots,\Delta_g, \Gamma_1,\dots,\Gamma_{g-1}\in\textrm{Irr}_{\mcC_\mcA}$.
Explicitly, the overall isomorphism is given by
\begin{equation}
\zrtca\big( \, H_g(\varphi,\gamma_1,\gamma_2,\dots,\gamma_g) \, \big) \longmapsfrom
\varphi \otimes \gamma_1 \otimes \gamma_2 \otimes \cdots \otimes \gamma_g \, .
\end{equation}
Here $H_g(\varphi,\gamma_1,\gamma_2,\dots,\gamma_g)$ is the ribbon bordism $\varnothing\lra\Sigma$, diffeomorphic to the solid genus-$g$ handlebody and having the following ribbon graph lying at its core:
\begin{equation}
\label{eq:HgR}
\pic[1.25]{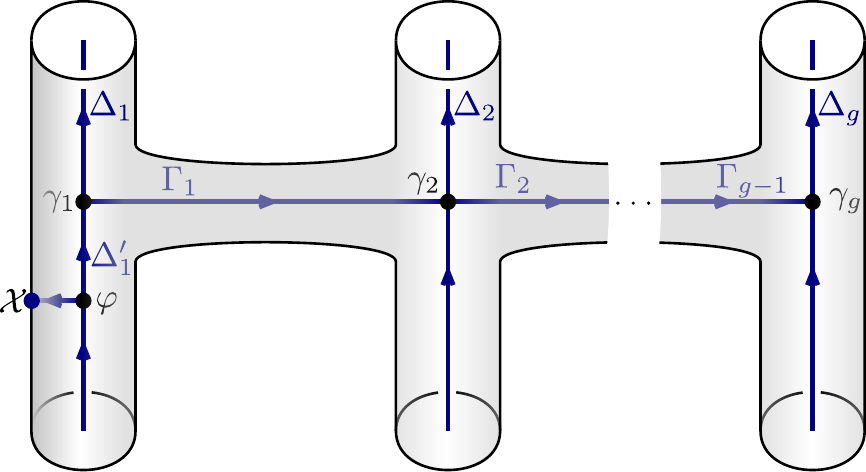} \, ,
\end{equation}
where each of the $g$ handles is depicted by a vertical cylinder with top and bottom identified.

\medskip

\begin{figure}
	\captionsetup{format=plain}
	\centering
	\pic[1.25]{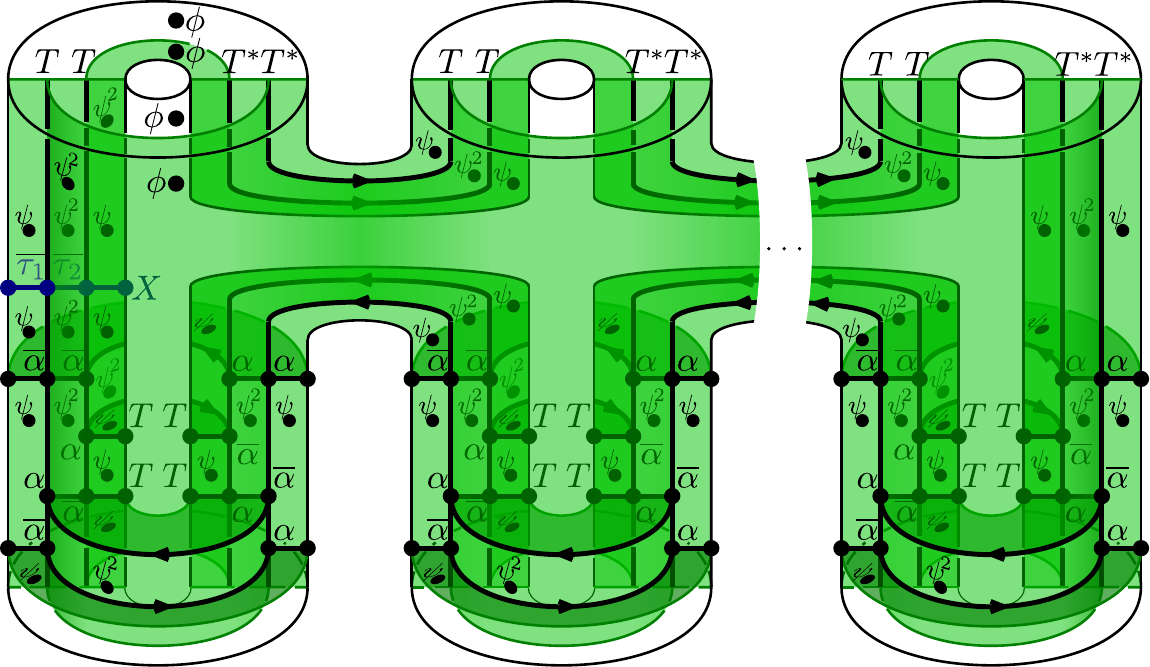}
	\caption{
	Depiction of an $\mcA$-coloured ribbon diagram $\mcT$ for the cylinder over an object $\Sigma \in \Bordriben{3}(\mcC_\mcA)$ with a single puncture labelled by $\mcX \in \mcC_\mcA$. 
	}
	\label{fig:Cyl_ribbon_diag}
	
	\captionsetup{format=plain}
	\centering
	\hspace{-10pt}
	\begin{subfigure}[b]{0.33\textwidth}
		\centering
		\pic[1.25]{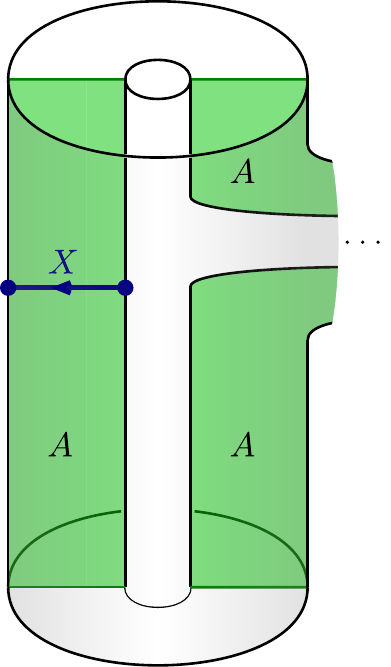}
		\caption{}
		\label{fig:stratification_1}
	\end{subfigure}
	\begin{subfigure}[b]{0.33\textwidth}
		\centering
		\pic[1.25]{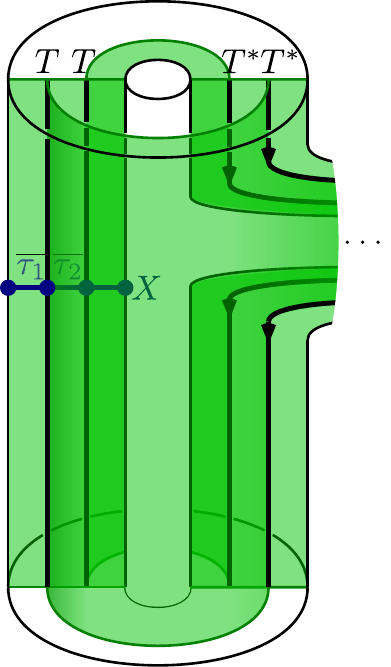}
		\caption{}
		\label{fig:stratification_2}
	\end{subfigure}
	\begin{subfigure}[b]{0.33\textwidth}
		\centering
		\pic[1.25]{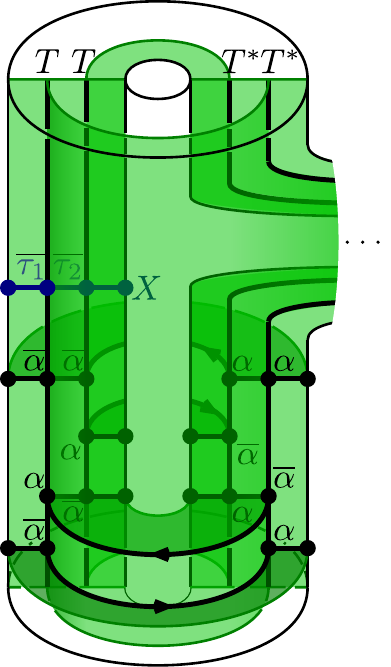}
		\caption{}
		\label{fig:stratification_3}
	\end{subfigure}
	\caption{The intermediate steps in the construction of the $\A$-coloured ribbon diagram in Figure~\ref{fig:Cyl_ribbon_diag}; only the leftmost handle is shown. 
	}
	\label{fig:stratification}
\end{figure}

We now turn to computing the state space $\zzca(\Sigma)$, for which we will use the formula \eqref{eq:ZCA_sos_as_im}.
To this end, let $\mcT$ be the $\mcA$-coloured ribbon diagram for the cylinder $C_\Sigma = \Sigma\times [0,1]$ as depicted in Figure~\ref{fig:Cyl_ribbon_diag}, where a similar presentation as that of the ribbon handlebody in~\eqref{eq:HgR} is used:
$C_\Sigma$ is represented by a closed solid genus-$g$ handlebody with an open solid handlebody removed from its interior;
each of the~$g$ handles is depicted as a vertical cylinder over an annulus with identified top and bottom.
To arrive at the ribbon diagram~$\mathcal T$ we proceed in four steps, the first three of which are illustrated in Figure~\ref{fig:stratification}, where we show only the part involving the leftmost handle:
\begin{itemize}
\item
one starts by adding to $C_\Sigma$ 
vertical $A$-labelled $2$-strata connecting the boundary components as shown in Figure~\ref{fig:stratification_1};
\item
the two boundary components are then separated by further stratification with $A$-labelled 2-strata, $T$-labelled 1-strata, and 0-strata labelled with appropriate crossing morphisms of~$\mcX$, resembling the one in the definition of the pipe functor as shown in Figure~\ref{fig:stratification_2};
\item
one adds the horizontal $\A$-labelled strata as in Figure~\ref{fig:stratification_3}, to make all $3$-strata contractible;
\item
the $\psi$- and $\phi$-insertions in~$\mathcal T$ are as in the definition of an $\mcA$-coloured ribbon diagram:
in our present case all $2$-strata are contractible, so those that touch the boundary receive a $\psi$-insertion and those that do not receive a $\psi^2$-insertion; there are four $3$-strata, all of which touch the boundary and therefore receive a $\phi$-insertion.
\end{itemize}

\begin{figure}
	\captionsetup{format=plain}
	\centering
	\hspace{-10pt}
    \begin{subfigure}[b]{0.33\textwidth}
		\centering
		\pic[1.25]{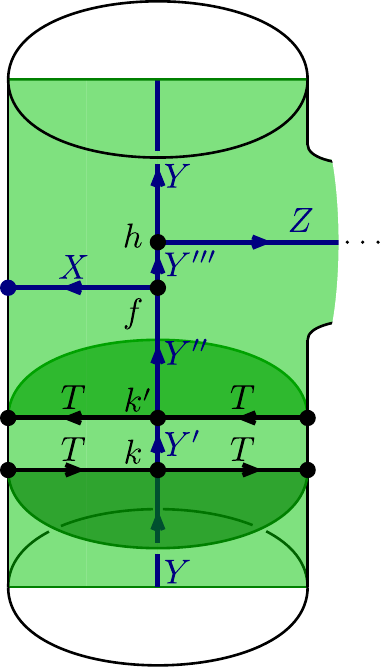}
		\caption{}
		\label{fig:HgG}
	\end{subfigure}
	\begin{subfigure}[b]{0.33\textwidth}
		\centering
		\pic[1.25]{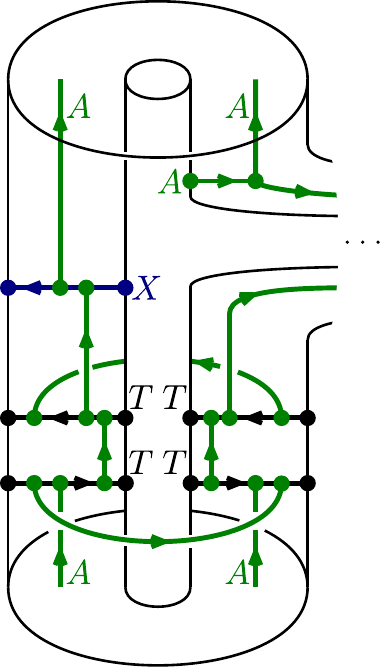}
		\caption{}
		\label{fig:SigmaG_cyl_ribbonised}
	\end{subfigure}
	\begin{subfigure}[b]{0.33\textwidth}
		\centering
		\pic[1.25]{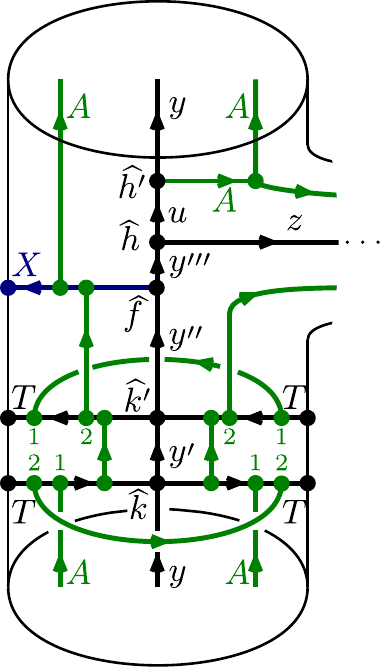}
		\caption{}
		\label{fig:HgG_ribbonised}
	\end{subfigure}
	\caption{
Bordisms used in the computation of $\zzc(\Sigma^\mcG)$. The bordism in (a) is in $\Borddefen{3}(\D^\mcC)$ and those in (b) and (c) are in $\Bordriben{3}(\mcC)$.
}
	\label{fig:zzc-Sigma_computation}
	
	\captionsetup{format=plain}
	\centering
	\hspace{-10pt}
	\begin{subfigure}[b]{0.33\textwidth}
		\centering
		\pic[1.25]{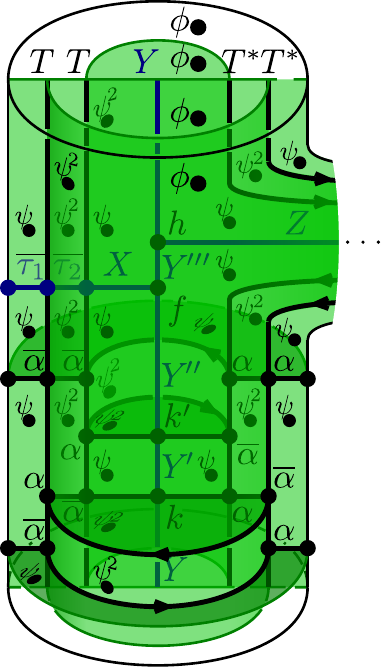}
		\caption{}
		\label{fig:ss_calc_1}
	\end{subfigure}
	\begin{subfigure}[b]{0.33\textwidth}
		\centering
		\pic[1.25]{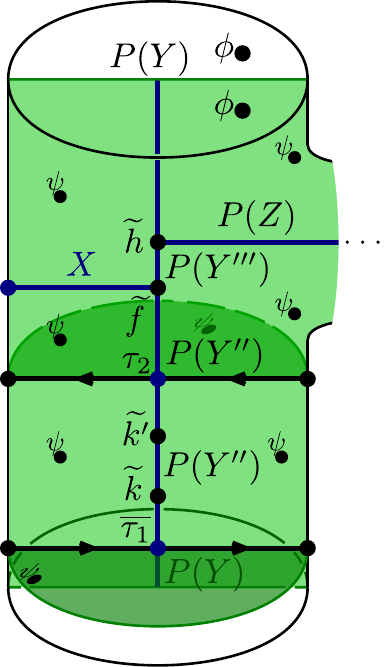}
		\caption{}
		\label{fig:ss_calc_2}
	\end{subfigure}
	\begin{subfigure}[b]{0.33\textwidth}
		\centering
		\pic[1.25]{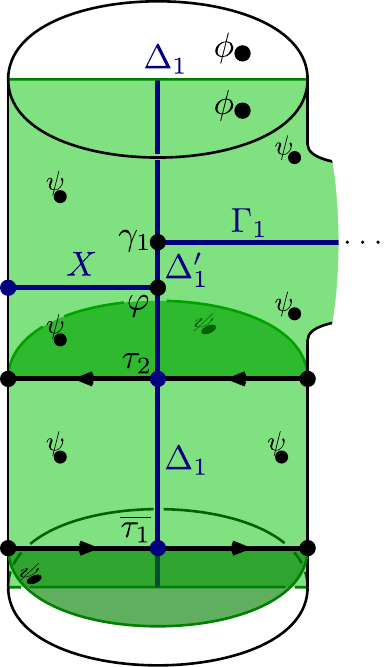}
		\caption{}
		\label{fig:ss_calc_3}
	\end{subfigure}
	\caption{
		Calculation of the space $\zzca(\Sigma)$ for a punctured surface $\Sigma\in\Bordriben{3}(\mcC_\mcA)$.
	}
	\label{fig:ss_calc}
\end{figure}

Let $\mcG$ be the $\mcA$-coloured 1-skeleton of the punctured surface $\Sigma$, obtained by restricting the ribbon diagram $\mcT$ to its incoming (or equivalently outgoing) boundary.
Denote the foamifications of~$\Sigma$ and~$C_\Sigma$ by $\Sigma^\mcG := F(C_\Sigma, \mcG)$ and $C_\Sigma^\mcT := F(C_\Sigma,\mcT)$, respectively. 
We need to compute the image of the idempotent
\begin{equation}
    \Psi_\mcG^\mcG := \big[\zzc(C_\Sigma^\mcT)\colon \zzc(\Sigma^\mcG)\lra\zzc(\Sigma^\mcG)\big] \, .
\end{equation}
As an intermediate step one therefore needs to know the vector space $\zzc(\Sigma^\mcG)$.
To this end, let $H_g^\mcG(k,k',f,h,\dots)$ be the solid defect genus-$g$ handlebody depicted in Figure~\ref{fig:HgG}, seen as a defect bordism $\varnothing\lra\Sigma^\mcG$, with the lines labelled by $A$-$A$-bimodules $Y, Y', Y'', Y''', Z, \dots$ and the points by $A$-$A\otimes A$-bimodule morphisms
\begin{equation}
k\colon T \otimes_1 Y \lra Y' \otimes_2 T \, , \quad
k'\colon Y' \otimes_0 T \lra T \otimes_2 Y'' \, , \quad\dots
\end{equation}
and $A$-$A$-bimodule morphisms
\begin{equation}
f\colon Y'' \lra X \otimes_A Y''' \, , \qquad
h\colon Y''' \lra Y \otimes_A Z \, , \quad\dots \, .
\end{equation}

\begin{lemma}
The vector space $\zzc(\Sigma^\mcG)$ is spanned by elements of the form $\zzc(H_g^\mcG(k,k',f,h,\dots))$.
\end{lemma}
\begin{proof}
As described in Construction~\ref{constr:DefectRT}, 
the state space $\zzc(\Sigma^\mcG)$ is the image of the idempotent $\Phi_\tau^\tau$ acting on $\zrt(R(\Sigma, \tau))$ as in \eqref{eq:ZC_sos_as_im}. 
In the present case, $\Phi_\tau^\tau$ is given by applying $\zrt$ to the bordism in Figure~\ref{fig:SigmaG_cyl_ribbonised}.
Acting on a spanning set of the state space $\zrt(R(\Sigma, \tau))$ in terms of handlebodies with ribbon graphs along their core (similar to \eqref{eq:HgR}, but for $\mathcal{C}$ and with $T$-insertions in addition to the $X$-insertion) results in the handlebody $\varnothing \lra R(\Sigma, \tau)$ shown in Figure~\ref{fig:HgG_ribbonised}.
This effectively replaces the objects $y,y',\dots$ by the induced bimodules $A \otimes y \otimes A$, $A \otimes y' \otimes A$, \dots. One verifies that the morphisms $\widehat{k}, \dots$ get mapped to morphisms commuting with $A$-actions as required by the defect bordism in Figure~\ref{fig:HgG} (but with $Y$ replaced by $A \otimes y \otimes A$, etc.).
Expanding the induced bimodules into simple bimodules shows the claim.
\end{proof}

By the previous lemma, the image of $\Psi_\mcG^\mcG$ is spanned by vectors of the form \begin{equation}
\label{eq:img_HgG}
    \Psi_\mcG^\mcG\big(\zzc(H_g^\mcG(k,k',f,h,\dots))\big) = \zzc\big(C_\Sigma^\mcT\circ H_g^\mcG(k,k',f,h,\dots)\big) \, .
\end{equation}
The defect handlebody in the argument of $\zzc$
on the right-hand side of \eqref{eq:img_HgG} is shown in Figure~\ref{fig:ss_calc_1}, where, as in Figure~\ref{fig:stratification}, only the first handle is depicted.
Using the pipe functor $P\colon\mcACA\lra\mcC_\mcA$, the stratification inside the handlebody can be exchanged for the one depicted in Figure~\ref{fig:ss_calc_2} without changing the value under $\zzc$. Here 
the points are labelled by morphisms
\begin{align}
&\widetilde{f} = \quad\pic[1.25]{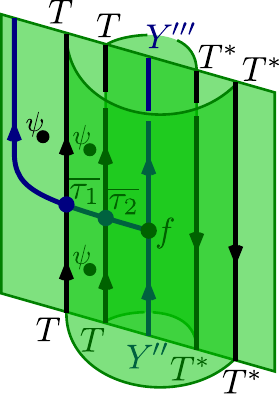},
&&\widetilde{h} = \quad\pic[1.25]{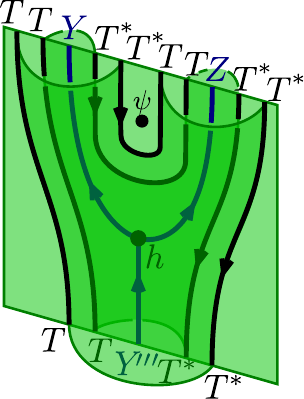},\\
&\widetilde{k} = \phi \cdot \pic[1.25]{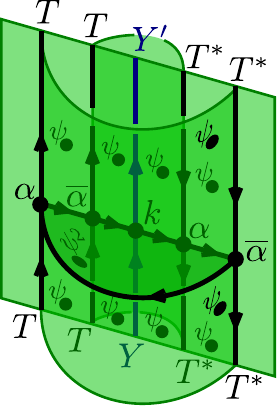},
&&\widetilde{k'} = \phi \cdot \pic[1.25]{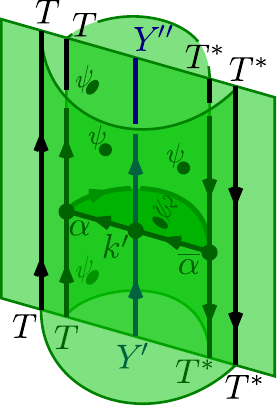},
\end{align}
all of which satisfy \eqref{eq:CA_morphism_cond} and are therefore morphisms in $\mcC_\mcA$.
By the semisimplicity of $\mcC_\mcA$, one can decompose $\Psi_\mcG^\mcG(H_g^\mcG)$ into a linear combination of handlebodies with stratification as in Figure~\ref{fig:ss_calc_3}, where -- using the same labels as in \eqref{eq:HgR} -- $\Delta_1, \Delta'_1, \Gamma_1$ are simple objects of $\mcC_\mcA$, and $\varphi\colon\Delta_1\lra X \otimes_A \Delta'_1$ and $\gamma_1\colon\Delta'_1 \lra \Delta_1 \otimes_A \Gamma$ are morphisms in $\mcC_\mcA$.

Denote the handlebody in Figure~\ref{fig:ss_calc_3} by $H_g^\A(\varphi,\gamma_1,\gamma_2,\dots,\gamma_g)$.
Then, in summary, $\zzca(\Sigma)$ is spanned by vectors of the form $\zzca(H_g^\A(\varphi,\gamma_1,\gamma_2,\dots,\gamma_g))$.

\begin{lemma}
\label{lem:RT-orb_state_spaces}
One has
\begin{equation}
    \dim \zrtca(\Sigma) \geqslant \dim \zzca(\Sigma) \, .
\end{equation}
\end{lemma}

\begin{proof}
Denote by $V$ the direct sum on the right-hand side of \eqref{eq:RT-state-space-direct-sum}.
Define the linear map $f\colon V \lra \zzca(\Sigma)$ 
via
\begin{equation}
\varphi \otimes \gamma_1 \otimes \gamma_2 \otimes \cdots \otimes \gamma_g 
\lmt
\zzca\big( \,H_g^\A(\varphi,\gamma_1,\gamma_2,\dots,\gamma_g) \, \big)
\, .
\end{equation}
Since the elements on the right-hand side span $\zzca(\Sigma)$, the map $f$ is surjective. 
Since by \eqref{eq:RT-state-space-direct-sum} we have $\zrtca(\Sigma) \cong V$, the claim follows.
\end{proof}

This completes the proof of Theorem~\ref{thm:tqfts_isomorphic}.

\newpage

\appendix

\section{Defining identities}
\label{app:identities}

\begin{figure}[!h]
    \captionsetup{format=plain}
	\captionsetup[subfigure]{labelformat=empty}
	\centering
	\begin{subfigure}[b]{0.5\textwidth}
		\centering
		\pic[1.25]{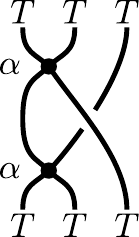}$=$\pic[1.25]{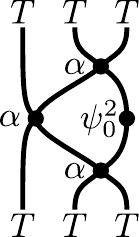}
		\caption{}
		\label{eq:O1}
	\end{subfigure}\hspace{-2em}\raisebox{5.5em}{(O1)}\\
	\begin{subfigure}[b]{0.4\textwidth}
		\centering
		\pic[1.25]{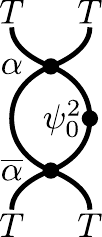}$=$\pic[1.25]{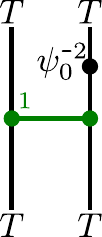}
		\caption{}
		\label{eq:O2}
	\end{subfigure}\hspace{-2em}\raisebox{5.5em}{(O2)}
	\begin{subfigure}[b]{0.4\textwidth}
		\centering
		\pic[1.25]{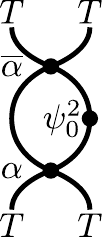}$=$\pic[1.25]{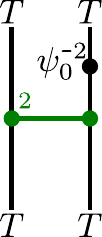}
		\caption{}
		\label{eq:O3}
	\end{subfigure}\hspace{-2em}\raisebox{5.5em}{(O3)}\\
	\begin{subfigure}[b]{0.4\textwidth}
		\centering
		\pic[1.25]{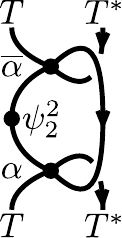}$=$\pic[1.25]{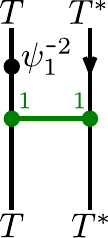}
		\caption{}
		\label{eq:O4}
	\end{subfigure}\hspace{-2em}\raisebox{5.5em}{(O4)}
	\begin{subfigure}[b]{0.4\textwidth}
		\centering
		\pic[1.25]{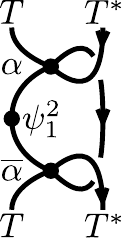}$=$\pic[1.25]{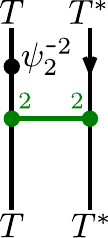}
		\caption{}
		\label{eq:O5}
	\end{subfigure}\hspace{-2em}\raisebox{5.5em}{(O5)}\\
	\begin{subfigure}[b]{0.4\textwidth}
		\centering
		\pic[1.25]{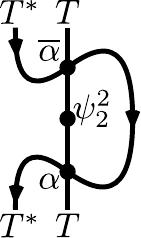}$=$\pic[1.25]{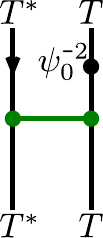}
		\caption{}
		\label{eq:O6}
	\end{subfigure}\hspace{-2em}\raisebox{5.5em}{(O6)}
	\begin{subfigure}[b]{0.4\textwidth}
		\centering
		\pic[1.25]{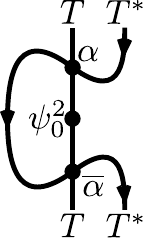}$=$\pic[1.25]{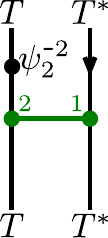}
		\caption{}
		\label{eq:O7}
	\end{subfigure}\hspace{-2em}\raisebox{5.5em}{(O7)}\\
	\begin{subfigure}[b]{0.8\textwidth}
		\centering
		\pic[1.25]{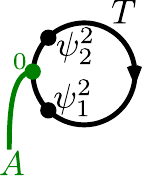}$=$\pic[1.25]{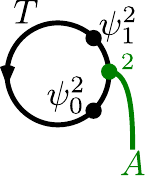}$=$\pic[1.25]{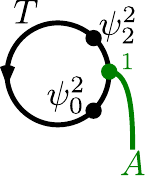}$=$\pic[1.25]{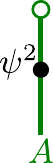}$\cdot \, \phi^{-2}$
		\caption{}
		\label{eq:O8}
	\end{subfigure}\hspace{-1em}\raisebox{4.25em}{(O8)}

\vspace*{-1em}

	\caption{%
    Defining conditions on an orbifold datum $\A = (A,T,\alpha,\overline{\alpha}, \psi,\phi)$ in a given modular fusion category~$\mcC$, cf.\ Section~\ref{subsubsec:SpecialOrbifoldData}.}
	\label{fig:SpecialOrbifoldDatum}
\end{figure}

\begin{figure}
    \captionsetup{format=plain}
	\captionsetup[subfigure]{labelformat=empty}
	\centering
	\vspace{-50pt}
	\begin{subfigure}[b]{0.5\textwidth}
		\centering
		\includegraphics[scale=0.85, valign=c]{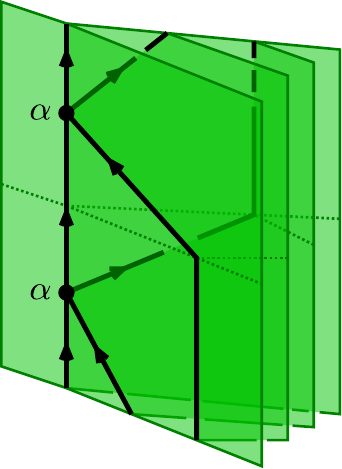} $=$
		\includegraphics[scale=0.85, valign=c]{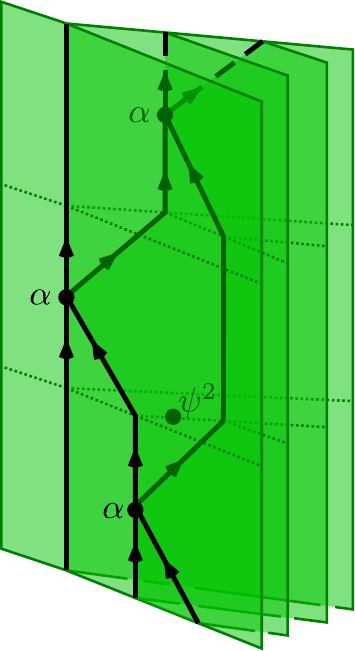}
		\caption{}
		\label{eq:OPSI1}
	\end{subfigure}\raisebox{8em}{(O1)}\\
	\vspace{-15pt}
	\hspace{-60pt}
	\begin{subfigure}[b]{0.53\textwidth}
		\centering
		\includegraphics[scale=0.85, valign=c]{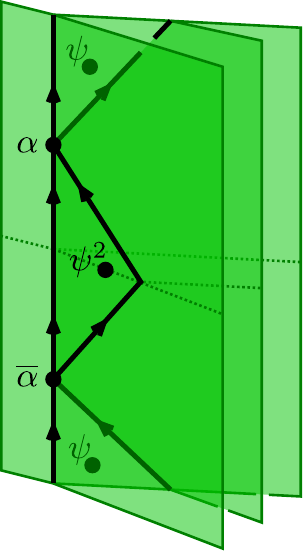} $=$
		\includegraphics[scale=0.85, valign=c]{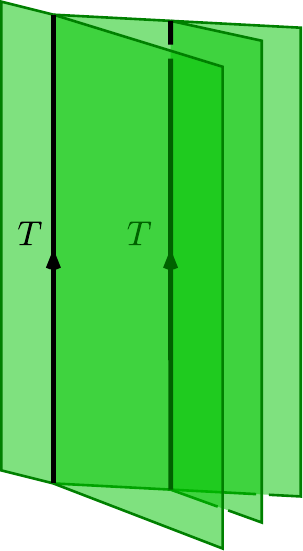}
		\caption{}
		\label{eq:OPSI2}
	\end{subfigure}\hspace{-2em}\raisebox{6.5em}{(O2)}
	\hspace{-15pt}
	\begin{subfigure}[b]{0.53\textwidth}
		\centering
		\includegraphics[scale=0.85, valign=c]{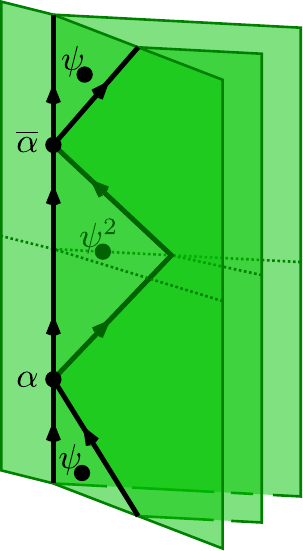} $=$
		\includegraphics[scale=0.85, valign=c]{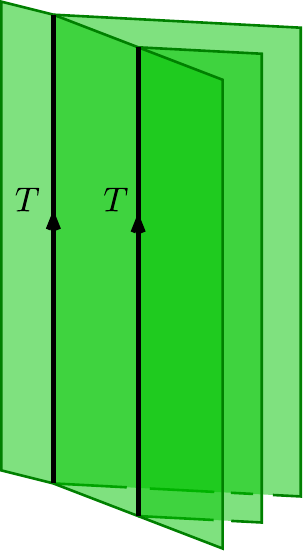}
		\caption{}
		\label{eq:OPSI3}
	\end{subfigure}\hspace{-2em}\raisebox{6.5em}{(O3)}\\
	\vspace{-15pt}
	\hspace{-60pt}
	\begin{subfigure}[b]{0.53\textwidth}
		\centering
		\includegraphics[scale=0.85, valign=c]{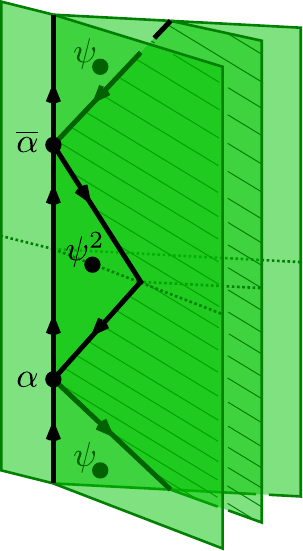} $=$
		\includegraphics[scale=0.85, valign=c]{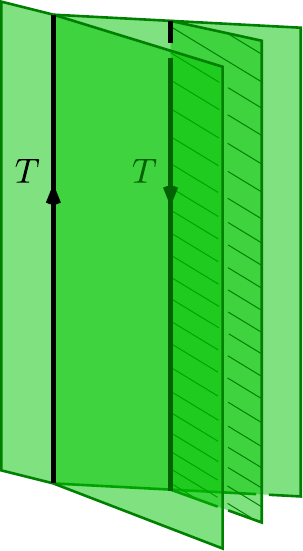}
		\caption{}
		\label{eq:OPSI4}
	\end{subfigure}\hspace{-2em}\raisebox{6.5em}{(O4)}
	\hspace{-15pt}
	\begin{subfigure}[b]{0.53\textwidth}
		\centering
		\includegraphics[scale=0.85, valign=c]{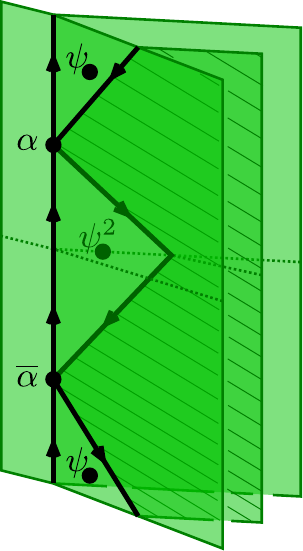} $=$
		\includegraphics[scale=0.85, valign=c]{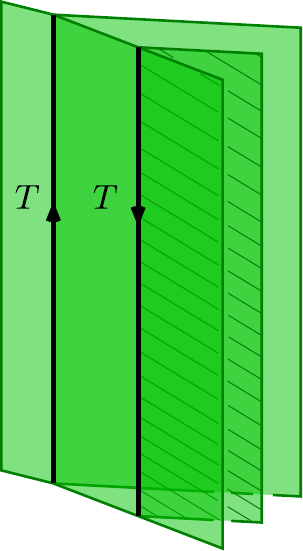}
		\caption{}
		\label{eq:OPSI5}
	\end{subfigure}\hspace{-2em}\raisebox{6.5em}{(O5)}\\
	\vspace{-15pt}
	\hspace{-60pt}
	\begin{subfigure}[b]{0.53\textwidth}
		\centering
		\includegraphics[scale=0.8, valign=c]{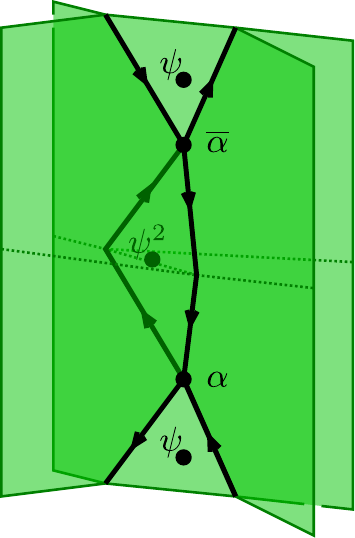} $=$
		\includegraphics[scale=0.8, valign=c]{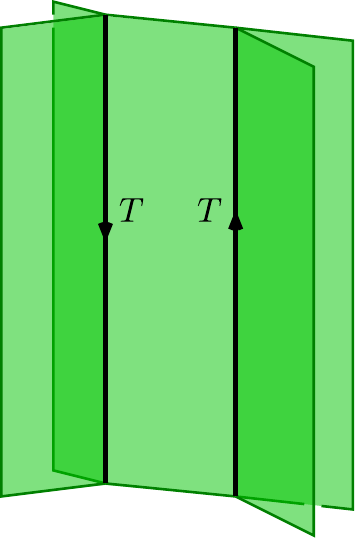}
		\caption{}
		\label{eq:OPSI6}
	\end{subfigure}\hspace{-1.5em}\raisebox{6em}{(O6)}
	\hspace{-15pt}
	\begin{subfigure}[b]{0.53\textwidth}
		\centering
		\includegraphics[scale=0.8, valign=c]{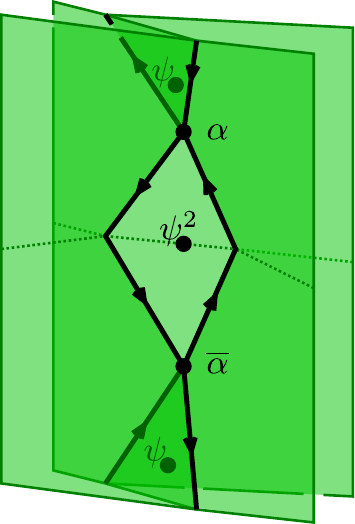} $=$
		\includegraphics[scale=0.8, valign=c]{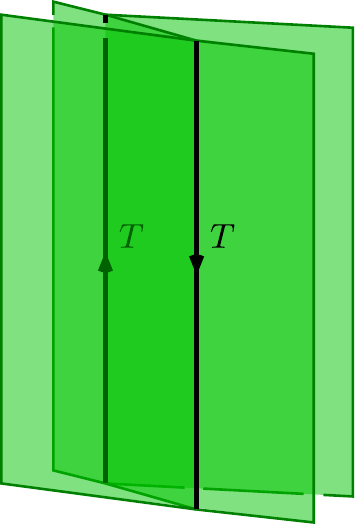}
		\caption{}
		\label{eq:OPSI7}
	\end{subfigure}\hspace{-1.5em}\raisebox{6em}{(O7)}\\
	\vspace{-15pt}
	\hspace{-60pt}
	\begin{subfigure}[b]{1.0\textwidth}
		\centering
		\includegraphics[scale=0.8, valign=c]{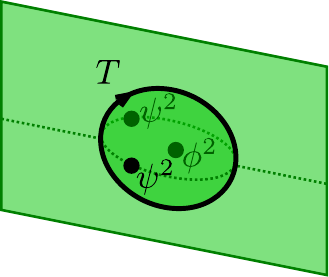} $=$
		\includegraphics[scale=0.8, valign=c]{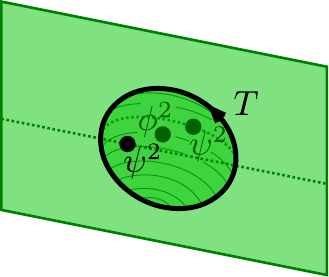} $=$
		\includegraphics[scale=0.8, valign=c]{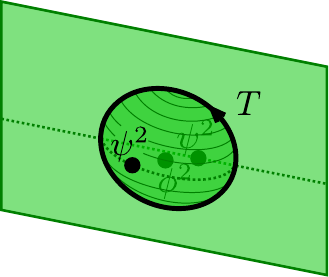} $=$
		\includegraphics[scale=0.8, valign=c]{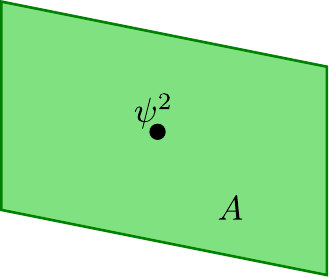}
		\caption{}
		\label{eq:OPSI8}
	\end{subfigure}\hspace{-2em}\raisebox{4em}{(O8)}
	\vspace{-15pt}
	\caption{%
	Defining conditions on an orbifold datum $\A=(A,T,\a,\abar,\psi,\phi)$ as local changes in defect configurations.
	The application of $\zzc$ on each side of the equations is implied.	
}
	\label{fig:SpecialOrbifoldDataPhiPsi}
\end{figure}

\begin{figure}
	\captionsetup[subfigure]{labelformat=empty}
	\centering
	\begin{subfigure}[b]{0.45\textwidth}
		\centering
		\pic[1.25]{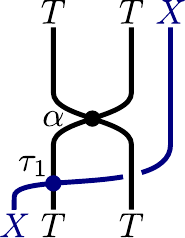}$=$\pic[1.25]{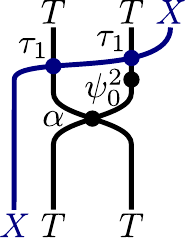}
		\caption{}
		\label{eq:T1}
	\end{subfigure}\hspace{-2em}\raisebox{5.5em}{(T1)}
	\begin{subfigure}[b]{0.45\textwidth}
		\centering
		\pic[1.25]{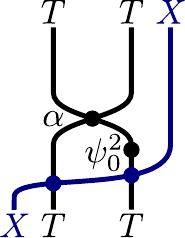}$=$\pic[1.25]{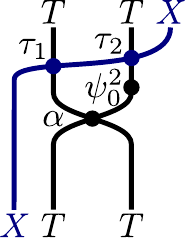}
		\caption{}
		\label{eq:T2}
	\end{subfigure}\hspace{-2em}\raisebox{5.5em}{(T2)}\\
	\begin{subfigure}[b]{0.45\textwidth}
		\centering
		\pic[1.25]{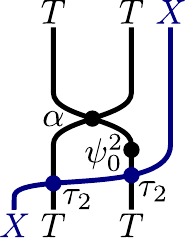}$=$\pic[1.25]{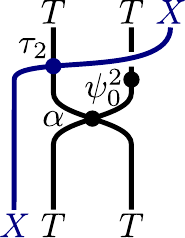}
		\caption{}
		\label{eq:T3}
	\end{subfigure}\hspace{-2em}\raisebox{5.5em}{(T3)}\\
	\begin{subfigure}[b]{0.35\textwidth}
		\centering
		\pic[1.25]{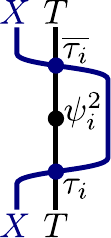}$=$\pic[1.25]{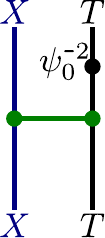}
		\caption{}
		\label{eq:T4}
	\end{subfigure}\hspace{-1em}\raisebox{5.5em}{(T4)}
	\begin{subfigure}[b]{0.35\textwidth}
		\centering
		\pic[1.25]{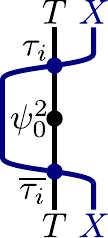}$=$\pic[1.25]{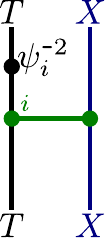}
		\caption{}
		\label{eq:T5}
	\end{subfigure}\hspace{-1em}\raisebox{5.5em}{(T5)}\hspace{4em}\raisebox{5.5em}{$i \in \{1,2\}$}\\
	\begin{subfigure}[b]{0.35\textwidth}
		\centering
		\pic[1.25]{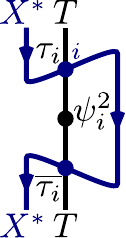}$=$\pic[1.25]{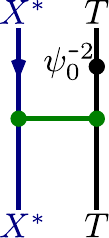}
		\caption{}
		\label{eq:T6}
	\end{subfigure}\hspace{-1em}\raisebox{5.5em}{(T6)}
	\begin{subfigure}[b]{0.35\textwidth}
		\centering
		\pic[1.25]{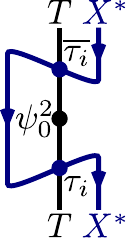}$=$\pic[1.25]{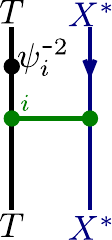}
		\caption{}
		\label{eq:T7}
	\end{subfigure}\hspace{-1em}\raisebox{5.5em}{(T7)}\hspace{4em}\raisebox{5.5em}{$i \in \{1,2\}$}
	\caption{%
	    Defining identities for the category $\mcC_\mcA$ constructed from an orbifold datum~$\A$ in~$\mcC$, cf.\ Section~\ref{subsubsec:SpecialOrbifoldData}. 
	} 
	\label{fig:CA_identities} 
\end{figure}

\begin{figure}
	\captionsetup[subfigure]{labelformat=empty}
	\centering
	\vspace{-15pt}
	\makebox[1.2\textwidth]{
		\hspace{-100pt}
		\begin{subfigure}[b]{0.6\textwidth}
			\centering
			\includegraphics[scale=1.0, valign=c]{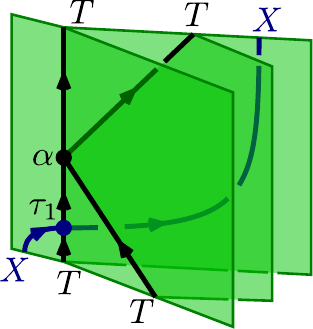} $=$
			\includegraphics[scale=1.0, valign=c]{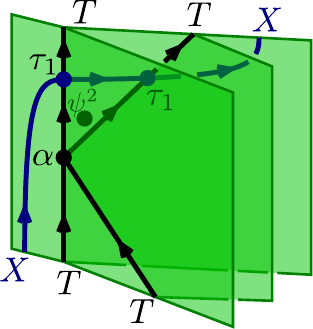}
			\caption{}
			\label{eq:TPSI1}
		\end{subfigure}\hspace{-2em}\raisebox{5.5em}{(T1)}
		\hspace{-30pt}
		\begin{subfigure}[b]{0.6\textwidth}
			\centering
			\includegraphics[scale=1.0, valign=c]{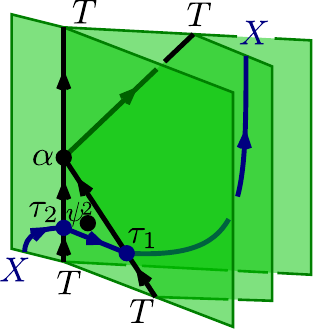} $=$
			\includegraphics[scale=1.0, valign=c]{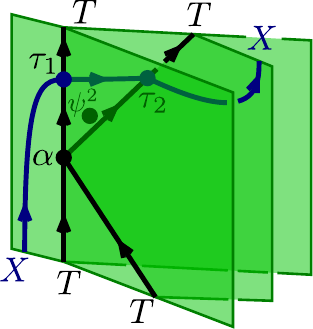}
			\caption{}
			\label{eq:TPSI2}
		\end{subfigure}\hspace{-2em}\raisebox{5.5em}{(T2)}}\\
	\vspace{-10pt}
	\begin{subfigure}[b]{0.6\textwidth}
		\centering
		\includegraphics[scale=1.0, valign=c]{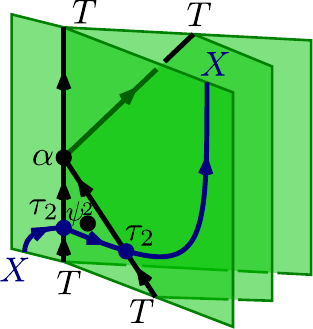} $=$
		\includegraphics[scale=1.0, valign=c]{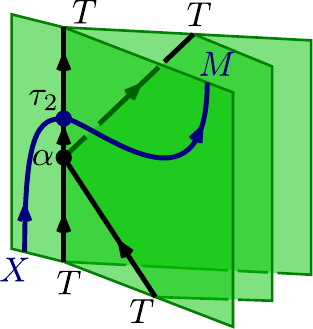}
		\caption{}
		\label{eq:TPSI3}
	\end{subfigure}\hspace{-2em}\raisebox{5.5em}{(T3)}\\ 
	\vspace{-15pt}
	\begin{subfigure}[b]{0.95\textwidth}
		\centering
		\includegraphics[scale=1.0, valign=c]{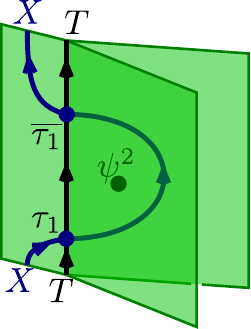} $=$
		\includegraphics[scale=1.0, valign=c]{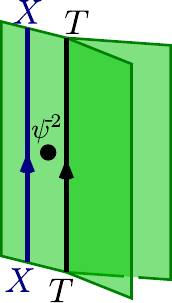}  , \quad
		\includegraphics[scale=1.0, valign=c]{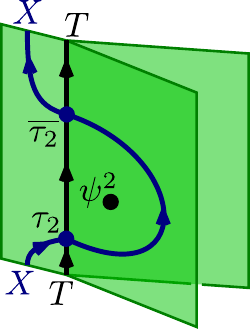} $=$
		\includegraphics[scale=1.0, valign=c]{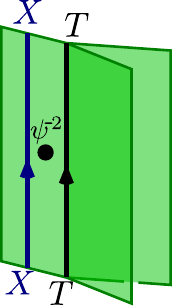}
		\caption{}
		\label{eq:TPSI4}
	\end{subfigure}\hspace{-1em}\raisebox{5.5em}{(T4)}\\
	\vspace{-15pt}
	\begin{subfigure}[b]{0.95\textwidth}
		\centering
		\includegraphics[scale=1.0, valign=c]{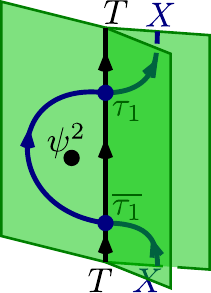} $=$
		\includegraphics[scale=1.0, valign=c]{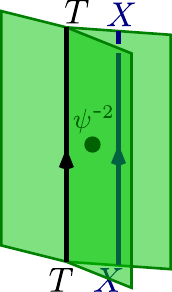}  , \quad
		\includegraphics[scale=1.0, valign=c]{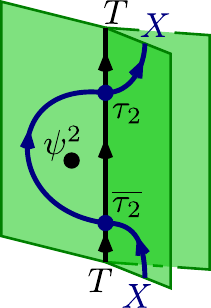} $=$
		\includegraphics[scale=1.0, valign=c]{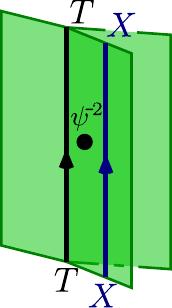}
		\caption{}
		\label{eq:TPSI5}
	\end{subfigure}\hspace{-1em}\raisebox{5.5em}{(T5)}\\
	\vspace{-15pt}
	\begin{subfigure}[b]{0.95\textwidth}
		\centering
		\includegraphics[scale=1.0, valign=c]{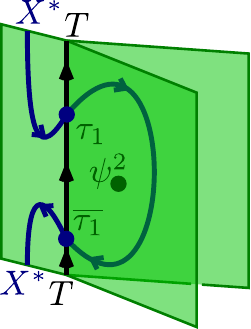} $=$
		\includegraphics[scale=1.0, valign=c]{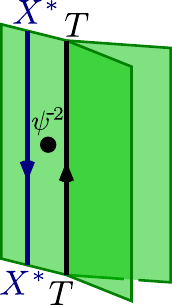}  , \quad
		\includegraphics[scale=1.0, valign=c]{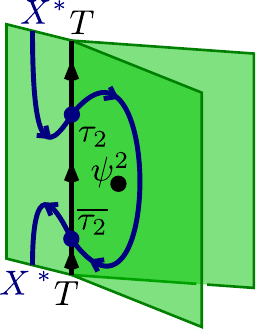} $=$
		\includegraphics[scale=1.0, valign=c]{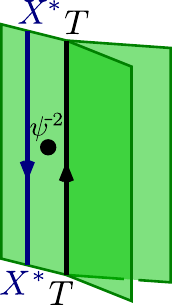}
		\caption{}
		\label{eq:TPSI6}
	\end{subfigure}\hspace{-1em}\raisebox{5.5em}{(T6)}\\
	\vspace{-15pt}
	\begin{subfigure}[b]{0.95\textwidth}
		\centering
		\includegraphics[scale=1.0, valign=c]{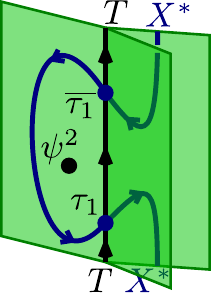} $=$
		\includegraphics[scale=1.0, valign=c]{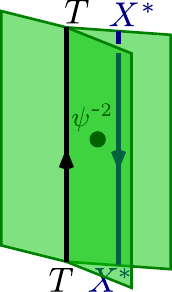}  , \quad
		\includegraphics[scale=1.0, valign=c]{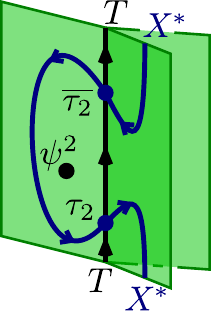} $=$
		\includegraphics[scale=1.0, valign=c]{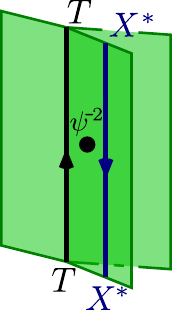}
		\caption{}
		\label{eq:TPSI7}
	\end{subfigure}\hspace{-1em}\raisebox{5.5em}{(T7)}
	\vspace{-15pt}
	\caption{%
	      Defining identities for the category $\mcC_\mcA$ constructed from an orbifold datum~$\A$ for any 3-dimensional defect TQFT, cf.\ \cite{CMRSS1}. 
	}
	\label{fig:CrossingIdentitiesPhiPsi}
\end{figure}

\clearpage

\section{Pipe functors}
\label{app:pipe_functors}

Throughout this section we fix a 3-dimensional defect TQFT $\mcZ \colon \Borddef_3(\D) \too \Vect$ over defect data $\D$ and we also fix an orbifold datum $\mcA$ for $\mcZ$ (which does not need to be simple).\footnote{We assume in this appendix that $\mcZ$ is an Euler-complete TQFT in order to avoid $\phi$- and $\psi$-insertions. However this is not really a limitation as for an arbitrary 3-dimensional defect TQFT this can always be achieved by passing to its ``Euler completion'' $\mcZ^{\odot}$; see also \cite[Appendix~B]{CMRSS1}.}
In this section we describe the pipe functors $P_1,P_2\colon \mcW \too \mcW_\mcA$, where $\mcW$ and $\mcW_\mcA$ are the monoidal categories associated to~$\mcZ$ and~$\mcA$ in \cite[Sect.\,4.2]{CMRSS1}. 
In the case where $\mcZ$ is the Reshetikhin--Turaev TQFT and $\mcA$ is the kind of orbifold data considered in \cite{MuleRunk} the construction we present here specialises
to the pipe functors introduced in \cite[Sect.\,3.3]{MuleRunk}, cf.~\eqref{eq:P-functor}. 
The objects of $\mcW$ are decorated lines that live on an $\mcA_2$-decorated plane and the category $\mcW_\mcA$ has as objects triples $(X, \tau_1, \tau_2)$ where $X \in \mcW$ and $\tau_1, \tau_2$ are isomorphisms that let an $X$-decorated line pass through $\mcA_1$ decorated lines.
For the construction of $P_1$, $P_2$ we introduce two intermediate categories $\mcD_1$, $\mcD_2$. The former has as objects tuples $(X,\tau_1)$ where $X$ is an object of $\mcW$ and $\tau_1$ is an isomorphism
\begin{equation}
    \tau_1 = \pic[1.25]{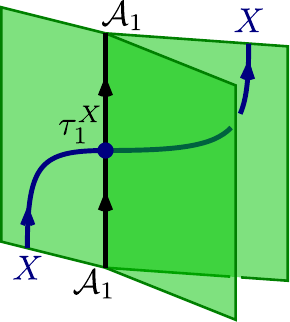}
    \quad\textrm{with inverse}\quad 
    \overline{\tau_1} = \pic[1.25]{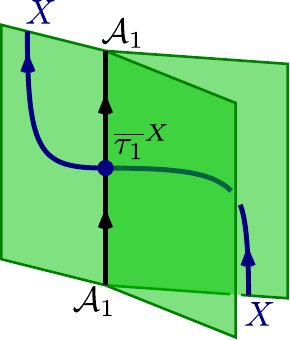}
    \,.
\end{equation}
We require that $\tau_1$, $\overline{\tau}_1$ satisfy the equations (T1) and (T4)--(T7) from \cite[p.\,40]{CMRSS1}. Similarly $\mcD_2$ has as objects tuples $(X,\tau_2)$ where $X \in \mcW$ and $\tau_2$ is an isomorphism 
\begin{equation}
    \tau_2 = \pic[1.25]{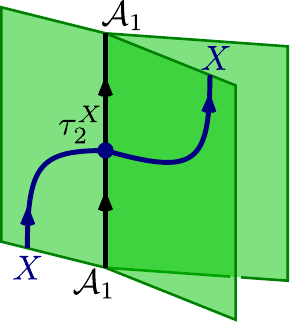}
    \quad\textrm{with inverse}\quad
    \overline{\tau_2} = \pic[1.25]{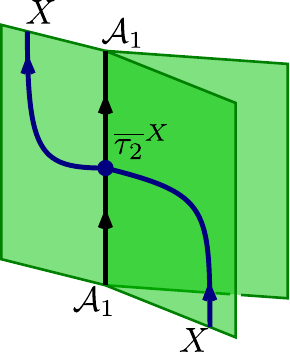}
    \,
\end{equation}
satisfying (T1) and (T4)--(T7) from \cite[p.\,40]{CMRSS1}. One obtains a strictly commuting diagram of forgetful functors
\begin{equation}
    \begin{tikzcd}
        \mcW & \mcD_1 \arrow[l, "U_1"'] \\
        \mcD_2 \arrow[u, "U_2"] & \mcW_{\mcA} \arrow[l, "U_{21}"] \arrow[u, "U_{12}"'] \arrow[ul, "U"']
    \end{tikzcd}.
\end{equation}

We now construct the ``half-pipe'' functors 
\begin{equation}
    \begin{tikzcd}
        \mcW \arrow[r, "H_1"] \arrow[d, "H_2"'] & \mcD_1 \arrow[d, "H_{12}"] \\
        \mcD_2 \arrow[r, "H_{21}"'] & \mcW_{\mcA}
    \end{tikzcd}\,,
\end{equation}
which will turn out to form biadjunctions with the respective forgetful functors, satisfying a separability condition. For $X \in \mcC$ we set
\begin{equation}
    H_1X = \pic[1.25]{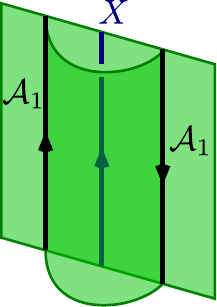} \quad \textrm{ and } \quad (\tau_1)_{H_1X} = \pic[1.25]{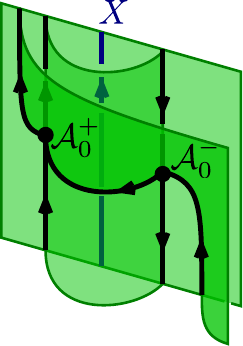} \,.
\end{equation} 
It can be readily checked that $\tau_{1,H_1X}$ indeed endows $H_1X$ with the structure of an object of $\mcD_1$. $H_2$ is defined similarly, except that each picture is reflected at the paper plane so that the half-pipe is added in the back:
\begin{equation}
    H_2X = \pic[1.25]{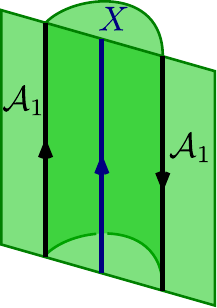} \quad \textrm{ and } \quad (\tau_2)_{H_2X} = \pic[1.25]{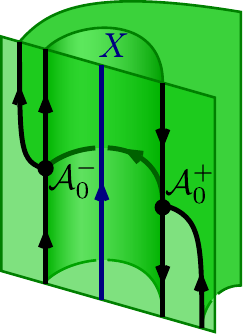} \,.
\end{equation} 
$H_{21}$ acts in the same way as $H_1$, by adding a half-pipe in the front. 
For any $(X,\tau_2) \in \mcD_2$ this equips $H_{21}(X,\tau_2)$ with a $\tau_1$ analogously to the case of $H_1$, so that we indeed end up with an object of $\mcW_{\mcA}$. 
Entirely analogously, $H_{12}$ acts as $H_2$ does.

\begin{proposition}
    For each $i \in \{1,2,12,21\}$, the functor $H_i$ is biadjoint to the corresponding
    forgetful functor, thus giving rise to two Frobenius monads $U_iH_i$ and $H_iU_i$. Each $U_iH_i$ is separable. Let $X \in \mcW$ and $(Y, \tau_1) \in \mcD_1$. 
    The unit $\eta$ and counit $\eps$ of the adjunction $H_1 \dashv U_1$ are given in components by
    \begin{equation}
        \eta_X = \pic[1.25]{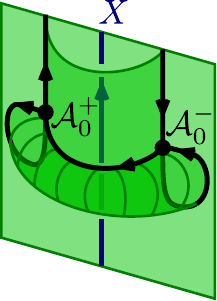}\,,  \qquad
        \eps_{(Y,\tau_1)} =\pic[1.25]{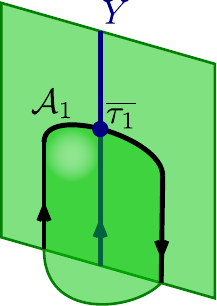}\,,
    \end{equation}
    and the unit $\widetilde{\eta}$ and counit $\widetilde{\eps}$ of the adjunction $H_1 \vdash U_1$ by
    \begin{equation}
        \widetilde\eta_{(Y,\tau_1)} = \pic[1.25]{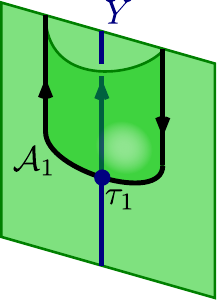}\,, \qquad 
        \widetilde\eps_{X}=\pic[1.25]{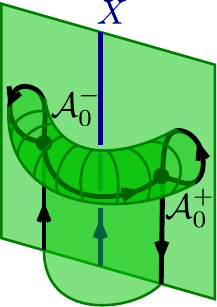}\,.
    \end{equation}
    The units and counits for the other adjunctions are obtained similarly.
\end{proposition}

\begin{proof}
    We check the necessary identities for the adjunctions involving $H_1$. 
    The remaining cases are similar and we leave them to the reader. 
    To see that $H_1 \dashv U_1$ we calculate
\begin{align}
        \eps_{H_1X} \circ H_1\eta_X &= 
        \pic[1.25]{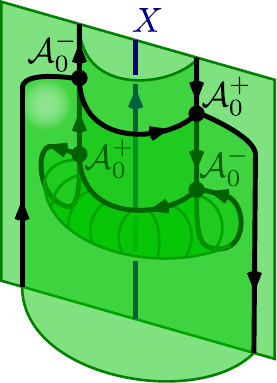}
        \overset{(1)}= \pic[1.25]{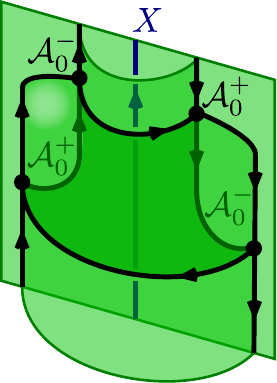} 
        \nonumber
        \\
        &\overset{(2)}= \pic[1.25]{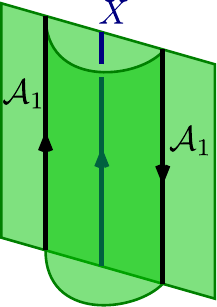} 
        = 1_{H_1X}\,.
    \end{align}
    In step (1) we first drag the end of the inner pipe upwards along the outside of the inner pipe. 
    We then attach it to the inside of the outer pipe using the defining equations for orbifold data, arriving at the second picture.
    In step (2) we pop the $\mcA$-decorated bubble as indicated, which is again possible by using the axioms of orbifold data. Furthermore we have
    \begin{align}
        U_1\eps_{(Y,\tau_1)} \circ \eta_{U_1(Y,\tau_1)} &=
        \pic[1.25]{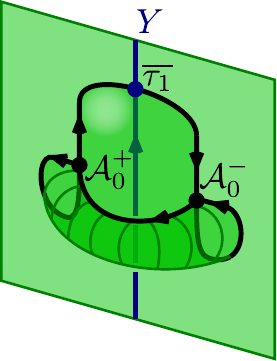} \overset{(1)}=
        \pic[1.25]{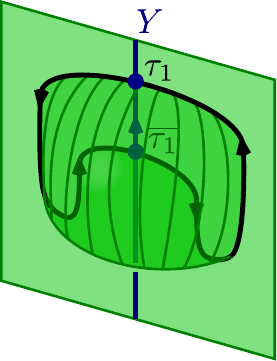}
        \nonumber
        \\ 
        &\overset{(2)}=
        \pic[1.25]{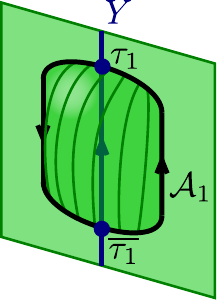} \overset{(3)}=
        \pic[1.25]{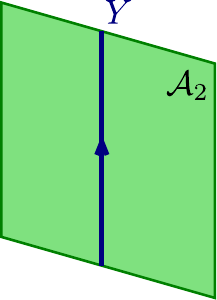} = 1_{U_1(Y,\tau_1)} \,.
    \end{align}
    In step (1) we use the compatibility of $\tau_1$ with $\mcA_0$ from the orbifold datum, step (2) is a matter of plugging in the definition of dual morphisms, while step (3) follows from $\overline{\tau_1}\tau_1 = 1$ and the fact that $\mcA$-decorated bubbles can be removed. 
    
    To see that $H_1 \vdash U_1$ we have to show that $\widetilde\eps_{U_1(Y,\tau_1)} \circ U_1\widetilde\eta_{(Y,\tau_1)} = 1_{U_1(Y,\tau_1)}$ and $H_1\widetilde\eps_{X} \circ \widetilde\eta_{H_1X} = 1_{H_1X}$. Both identities follow pictorially very similarly as above, only that the pictures are rotated by 180° around the axis that is orthogonal to the paper plane and the orientation of $X$- and $Y$-decorated lines is reversed. 

    Finally, the separability of the Frobenius monad $U_1H_1$ 
    follows from
    \begin{equation}
        \eps_{(Y,\tau_1)}\widetilde\eta_{(Y,\tau_1)} = 
        \pic[1.25]{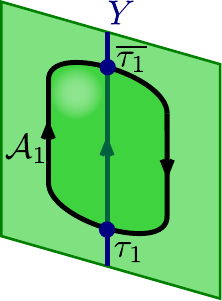} =
        \pic[1.25]{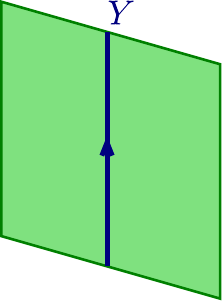} =
        1_{(Y,\tau_1)}\,.
    \end{equation}
\end{proof}

\begin{definition}
    The \textsl{pipe functors} $P_1,P_2 \colon \mcW \longrightarrow \mcW_\mcA$ are given by $P_1 := H_{12}H_1$ and $P_2 := H_{21}H_2$. They act by
    \begin{equation}
        P_1X = \pic[1.25]{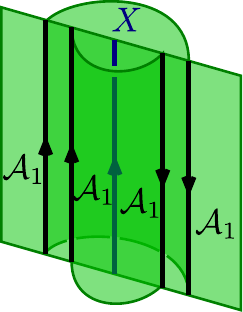}\,, \quad P_2X = \pic[1.25]{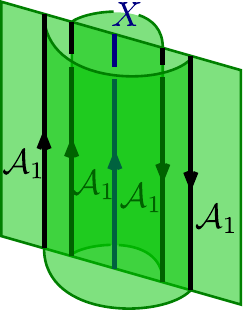}\,.
    \end{equation}
\end{definition}
It follows from the above proposition that both $P_1$ and $P_2$ are biadjoint to the forgetful functor $U\colon \mcW_\mcA \longrightarrow \mcW$ where the units and counits are obtained by composing the corresponding ones of the halfpipe functors. 
There is a unique isomorphism $\varphi \colon P_1 \overset{\cong}{\longrightarrow} P_2$ that is compatible with the biadjunction data and one verifies that it is given in components by
\begin{equation}
    \varphi_X = \pic[1.25]{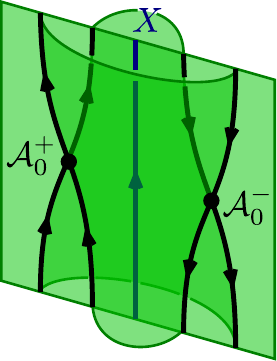}
    \quad\textrm{with inverse}\quad
    \varphi_X^{-1} = \pic[1.25]{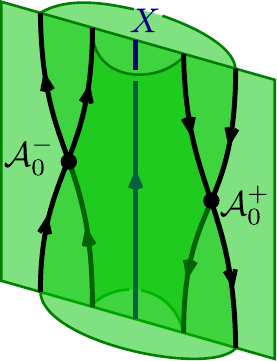} \, .
\end{equation}

\end{document}